\documentclass[11pt]{amsart}
\usepackage[usenames,dvipsnames]{pstricks}
\usepackage[mathscr]{eucal}
\usepackage{amsfonts,amsmath,amssymb,amsthm,amscd,amsxtra,cite}
\usepackage{enumerate,verbatim}
\usepackage[all,2cell,ps]{xy}
\usepackage[notcite, notref]{}
\usepackage[pagebackref]{hyperref}
\usepackage{todonotes}
\usepackage{mathptmx}
\usepackage{calc,color}
\usepackage{wasysym}

\theoremstyle{plain} 
\newtheorem{thm}{Theorem}[section]

\newtheorem{prop}[thm]{Proposition}
\newtheorem{cor}[thm]{Corollary}

\theoremstyle{definition}
\newtheorem{lem}[thm]{Lemma}
\newtheorem{dfn}[thm]{Definition}

\newtheorem{eg}[thm]{Example}

\newtheorem{discussion}[thm]{Discussion}

\newtheorem*{prb}{Problem}
\newtheorem{rmk}[thm]{Remark}

\newtheorem{chunk}[thm]{\hspace*{-1.065ex}\bf}
\numberwithin{equation}{section}

\newcommand{\fm}{\mathfrak{m}}
\newcommand{\fn}{\mathfrak{n}}
\newcommand{\fp}{\mathfrak{p}}
\newcommand{\fq}{\mathfrak{q}}
\newcommand{\fa}{\mathfrak{a}}

\newcommand{\up}[1]{{{}^{#1}\!}}

\newcommand{\RR}{\mathbb{R}}

\newcommand{\ZZ}{\mathbb{Z}}
\newcommand{\NN}{\mathbb{N}}

\newcommand{\lo}{\longrightarrow}

\DeclareMathOperator{\ann}{ann}
\DeclareMathOperator{\Ass}{Ass}

\DeclareMathOperator{\CI}{\textnormal{CI-dim}}
\DeclareMathOperator{\G-dim}{\textnormal{G-dim}}
\DeclareMathOperator{\gr}{\textnormal{grade}}
\DeclareMathOperator{\cx}{cx}
\DeclareMathOperator{\Tr}{Tr}
\DeclareMathOperator{\X}{X}

\DeclareMathOperator{\V}{\mathcal{V}}
\DeclareMathOperator{\mz}{\mathcal{Z}}
\DeclareMathOperator{\mhom}{\mathcal{H}om}

\DeclareMathOperator{\coker}{coker}

\DeclareMathOperator{\depth}{depth}

\DeclareMathOperator{\Ext}{Ext}

\DeclareMathOperator{\T}{T}

\DeclareMathOperator{\Ht}{ht}
\DeclareMathOperator{\hh}{H}
\DeclareMathOperator{\E}{E}
\DeclareMathOperator{\Hom}{Hom}

\DeclareMathOperator{\im}{im}

\DeclareMathOperator{\pd}{pd}
\DeclareMathOperator{\Md}{Mod}
\DeclareMathOperator{\md}{mod}

\DeclareMathOperator{\Supp}{Supp}
\DeclareMathOperator{\Spec}{Spec}
\DeclareMathOperator{\Sing}{Sing}
\DeclareMathOperator{\Syz}{\Omega}

\DeclareMathOperator{\Tor}{Tor}

\DeclareMathOperator{\cod}{codim}

\DeclareMathOperator{\NF}{NF}

\def\urltilda{\kern -.15em\lower .7ex\hbox{\~{}}\kern .04em}
\def\urldot{\kern -.10em.\kern -.10em}\def\urlhttp{http\kern -.10em\lower -.1ex
\hbox{:}\kern -.12em\lower 0ex\hbox{/}\kern -.18em\lower 0ex\hbox{/}}

\begin{document}

\title[A study of the cohomological rigidity property ]{A study of the cohomological rigidity property }

\author[Asgharzadeh, Celikbas, Sadeghi]{Mohsen Asgharzadeh, Olgur Celikbas, Arash Sadeghi}

\address{Mohsen Asgharzadeh\\
Hakimieh, 16599-19556, Tehran, Iran}
\email{mohsenasgharzadeh@gmail.com}

\address{Olgur Celikbas\\
	Department of Mathematics \\
	West Virginia University\\
	Morgantown, WV 26506-6310, U.S.A}
\email{olgur.celikbas@math.wvu.edu}

\address{Arash Sadeghi\\
	 Kooy-e-Eram, 14847-76771, Tehran, Iran} \email{sadeghiarash61@gmail.com}



\begin{abstract} In this paper, motivated by a work of Luk and Yau, and Huneke and Wiegand, we study various aspects of the cohomological rigidity property of tensor product of modules over commutative Noetherian rings. We determine conditions under which the vanishing of a single local cohomology module of a tensor product implies the vanishing of all the lower ones, and obtain new connections between the local cohomology modules of tensor products and the Tate homology. Our argument yields bounds for the depth of tensor products of modules, as well as criteria for freeness of modules over complete intersection rings. Along the way, we also give a splitting criteria for vector bundles on smooth complete intersections.
\end{abstract}

\maketitle{}

\setcounter{tocdepth}{1}
\tableofcontents

\section{Introduction}

A vector bundle $\mathcal{E}$ on a projective scheme $X$ equipped with a very ample line bundle $\mathcal{O}(1)$ is said to have the \emph{cohomological rigidity property} if there is a positive integer $i$ such that the vanishing of
\begin{equation}\notag{}
\hh^i_\ast(X,\mathcal{\mathcal{E}}\otimes \mathcal{\mathcal{\mathcal{E}}}^\ast):=\bigoplus_{n\in\mathbb{Z}}\hh^i(X,\mathcal{E}\otimes \mathcal{E}^\ast(n))
\end{equation}
implies that $\mathcal{E}$ is trivial, i.e., isomorphic to a direct sum of line bundles. A paradigm for this rigidity property is a result of  Luk and Yau \cite{Luk}, which is concerned with $(\mathbb{P}^n_{\mathbb{C}},\mathcal{O}(1))$. 

Our motivation in this paper comes from a beautiful work of Huneke and Wiegand \cite{HW1}, which reproves and extends the aforementioned rigidity result of Luk and Yau via the machinery of commutative algebra. Huneke and Wiegand investigates suitable conditions under which if one local cohomology module of a tensor product of finitely generated modules vanishes, then all lower ones vanish. As a consequence, Huneke and Wiegand \cite{HW1} obtained remarkable results that relate Serre's conditions  to the vanishing of a single local cohomology module of a tensor product.

In this paper, following the work of Huneke and Wiegand \cite{HW1}, we study the (non) vanishing of local cohomology modules, and investigate depth and torsion properties of tensor products of modules. The new advantage we have is that we work with local cohomology functors with respect to specialization-closed subsets of $\Spec R$; this allows us to generalize the results of Huneke and Wiegand in this direction, as well as various results from the literature, especially those stated in terms of Serre's conditions. 

One of our main results is Theorem \ref{t2}, which can be considered as a generalization of rigidity theorem of Huneke and Wiegand \cite[2.4]{HW1}. A special case of Theorem \ref{t2} can be stated as follows; see Proposition \ref{c4}.

\begin{thm}\label{it2} Let $R$ be a regular local ring and let $M$ and $N$ be non-zero finitely generated $R$-modules. Let $\mz\subset\Spec R $ be a specialization-closed subset and let $n\geq 0$ be an integer. Assume the following hold:
\begin{enumerate}[\rm(i)]
\item $\NF(M)\cap\NF(N)\subseteq\mz$.
\item $\hh_{\mz}^n(M\otimes_RN)=0$.
\item $\gr_R(\mz,M)\geq n$ and $\gr_R(\mz,N)\geq n$.
\end{enumerate}
Then it follows that $\hh_{\mz}^i(M\otimes_RN)=0$ for all $i=0, \ldots, n$, and $\Tor_j^R(M,N)=0$ for all $j\geq 1$.
\end{thm}

In Theorem \ref{it2}, $\NF(M)$ denotes the set $\{\fp\in \Spec(R): M_{\fp} \text{ is not free over } R_{\fp} \}$, i.e., the \emph{non-free locus} of $M$. Recall that a subset $\mz$ of $\Spec R$ is called \emph{specialization-closed} if every prime ideal of $R$ containing some prime ideal in $\mz$ belongs to $\mz$. Clearly, every closed subset of $\Spec R$ in the Zariski topology is specialization-closed. For an integer $n$, we denote by $\hh^n_{\mz}(-)$ the $n$-th local cohomology functor with respect to $\mz$. Moreover, the \emph{grade of $M$ with respect to  $\mz$} is denoted by $\gr_R(\mz,M)$, and is defined as the infimum of the set of integers $n$ such that $\hh^n_{\mz}(M)$ is non-zero; see \ref{lsc} and \ref{gsc} for  further details.

The torsion in tensor product of modules was initially studied by Auslander in his seminal paper ``Modules over unramified regular local rings" \cite{A2}. For nonzero finitely generated modules $M$ and $N$ over an unramified (or equi-characteristic) regular local ring, Auslander proved that $M$ and $N$ are torsion-free and $\Tor_i^R(M,N)=0$ for all $i\geq 1$, provided that $M\otimes_RN$ is torsion-free; subsequently, the ramified case of Auslander's result was established by Lichtenbaum \cite{L}. Theorem \ref{it2} provides a generalization of the aforementioned results of Auslander and Lichtenbaum; this is because the torsion submodule of a module $M$ can be characterized by the $\mz$-torsion submodule $\Gamma_{\mz}(M)$ of $M$ for a suitable choice of $\mz$; see Proposition \ref{Serre} and Corollary \ref{AL}. To the best of our knowledge, the conclusion of Theorem \ref{it2} is new even for the closed subsets of $\Spec(R)$.

In Section $3$ we give a proof of Theorem \ref{it2}. Sections 4 and 5 are devoted to several applications; see, for example, Corollaries \ref{c1c}, \ref{h1zero} and \ref{t10}. 
In Section 6 we apply our results and study the vanishing of local cohomology modules of the Frobenius powers over local rings of prime characteristic; see Theorem \ref{tt}. We make use of the fact that the Frobenius endomorphism $\up{\varphi^r}R$ is Tor-rigid over such complete intersection rings, and obtain the following as an application of Theorem \ref{tt}; see Corollary \ref{Fr}.

\begin{thm} \label{ft}Let $(R,\fm)$ be a $d$-dimensional local complete intersection ring of prime characteristic. Then $R$ is regular provided that at least one of the following conditions hold:
\begin{enumerate}[\rm(i)]
\item $R$ is reduced and $\hh^{d-1}_{\fm}(\up{\varphi^r}R\otimes_R\up{\varphi^s}R)=0$ for some integers $r,s\geq 1$.
\item $R$ is normal and $\hh^{d-2}_{\fm}(\up{\varphi^r}R\otimes_R\up{\varphi^s}R)=0$ for some integers $r,s \geq 1$.
\end{enumerate}
\end{thm}

In Section 7 we determine a new connection between local cohomology of tensor products of modules and the Tate homology $\widehat{\Tor}_{}$; see \ref{Tate} for the definition. The next result follows from Theorem \ref{t12}; it is the second main theorem of this paper besides Theorem \ref{t2}.

\begin{thm} \label{7.1intro} Let $R$ be a Gorenstein local ring, and let $M$ and $N$ be finitely generated $R$-modules. Assume $\mz\subset\Spec R$ is a specialization-closed subset, and that the following conditions hold:
\begin{enumerate}[\rm(i)]
\item $\NF(M)\cap\NF(N)\subseteq\mz$. 
\item $\depth_{R_\fp}(M_\fp)+\depth_{R_\fp}(N_\fp)\geq\depth(R_{\fp})+n$ for each $\fp\in\mz$ and for some integer $n\geq 0$.
\end{enumerate}	
Then it follows that:
\begin{enumerate}[\rm(1)]
\item $\hh^i_{\mz}(M\otimes_RN)\cong\widehat{\Tor}_{-i}^R(M,N)$ for all $i=0, \ldots, n-1$. 
\item There is an injection $\widehat{\Tor}_{-n}^R(M,N)\hookrightarrow\hh^n_{\mz}(M\otimes_RN)$.
\end{enumerate}	
\end{thm}

As far as we know, Theorem \ref{7.1intro} is new, even for closed subsets of $\Spec(R)$. Various applications of Theorem \ref{7.1intro} corroborating the literature include Corollaries \ref{cxmm}, \ref{c2} and \ref{hhyper}. Furthermore, Theorem \ref{7.1intro} determines a useful bound on depth of tensor products $M\otimes_RN$ of certain modules $M$ and $N$ over complete intersection rings:

\begin{cor}\label{cd} Let $R$ be a local complete intersection ring of codimension $c$, and let $M$ and $N$ be nonzero finitely generated $R$-modules, each of which is locally free on the punctured spectrum of $R$.  If  $\depth_R(M)+\depth_R(N)-\depth(R) \geq c$, then $\depth_R(M\otimes_RN)\leq \depth_R(M)+\depth_R(N)-\depth(R)$.
\end{cor}

The conclusion of Corollary \ref{cd} seems interesting to us since  it does not assume the vanishing of Tor modules; see Corollary \ref{c} and cf. \cite[3.1]{CST}.

Our aim in Section 8 is, motivated by the work of Huneke and Wiegand \cite{HW1}, to determine some new criteria for freeness of modules in terms of the vanishing of local cohomology.  More precisely, we  are concerned with the cohomological rigidity property of tensor products in the sense that the freeness of a module $M$ follows from the vanishing of $ \hh^{i}_{\fm}(M\otimes_RM^\ast)$ for some integer $i$. Our work extends several results from the literature. For example, as a consequence of Theorem \ref{t7}, we proved the following result in Corollary \ref{gor}; it extends \cite[3.9]{CeD} which establishes the case where $n=0$:

\begin{cor} Let $R$ be a local complete intersection ring of even dimension $d$, and let $M$ be a maximal Cohen-Macaulay $R$-module which is locally free on the punctured spectrum of $R$. If $\hh^n_{\fm}(M\otimes_RM^*)=0$ for some integer $n$ where $0\leq n<d$, then $M$ is free.
\end{cor}


The cohomological rigidity property on hypersurfaces of odd dimension was first studied by Dao \cite{D1}: a vector bundle $\mathcal{E}$ on an odd dimensional hypersurface of dimension at least $3$ splits if and only if $\hh^1 (X,\mathcal{E}\otimes\mathcal{E}^\ast(i))=0$ for all $i\in\ZZ$ \cite[1.5]{D1}. This result has been recently studied by \v{C}esnavi\v{c}ius, who proved that a vector bundle $\mathcal{E}$ on a smooth complete intersection of dimension at least $ 3$ splits into a sum of line bundles if and only if $\hh^1 (X,\mathcal{E}\otimes\mathcal{E}^\ast(i))=0=\hh^2 (X,\mathcal{E}\otimes\mathcal{E}^\ast(i))$ for all $i\in\ZZ$ \cite[1.2]{ces}. As an application of our study, for an arithmetically Cohen--Macaulay vector bundle on a smooth complete intersection of odd dimension we obtain a stronger result; see Corollary \ref{artCM}:

\begin{cor}
	Let $k$ be a field and let $X\subset \mathbb{P}^n_k$ be a globally complete intersection of odd dimension $d$. Assume $\mathcal{E}$ is an arithmetically Cohen--Macaulay vector bundle and that $\hh^{i}(X,\mathcal{E}\otimes\mathcal{E}^*(j))=0$ for all $j\in\ZZ$ and for some $i$ where $0<i<d$. Then $\mathcal{E}$ is a direct sum of powers of $\mathcal{O}(1)$.
\end{cor}
Furthermore, for an arithmetically Cohen--Macaulay vector bundle on a hypersurface we have:
\begin{cor}
	Let $k$ be a field and let $X\subset \mathbb{P}^n_k$ be a hypersurface. Assume $\mathcal{E}$ is an arithmetically Cohen-Macaulay vector bundle such that 
	$\hh^i(X,\mathcal{E}\otimes\mathcal{E}^*(j))=0$ for all $j\in\ZZ$ and for some even integer $i$ where $0<i<\dim X$. Then $\mathcal{E}$ is direct sum of powers of  $\mathcal{O}(1)$.
\end{cor}



\section{Notations and preliminary results}


Throughout, $R$, $\Md(R)$, $\md(R)$ and $\mathcal{P}(R)$ denote a commutative Noetherian ring, the category of $R$-modules, the  category of finitely generated $R$-modules, and the subcategory of $\md(R)$ of finitely generated projective $R$-modules, respectively.

If $R$ is a local ring, i.e., a commutative Noetherian local ring, then $\fm$ denotes the unique maximal ideal of $R$, and $k$ denotes the residue field of $R$.

$(-)^*$ stands for the algebraic dual $\Hom_R(-,R)$, and $e_M$ is the natural map $M\to M^{**}$.

$M,N \in \Md(R)$ are said to be \emph{stably isomorphic}, denoted by $M\approx N$, provided that $M\oplus P\cong N\oplus Q$ for some $P, Q\in\mathcal{P}(R)$. 

\begin{chunk}\textbf{Right and left projective approximations.} An $R$-module homomorphism $f:X\to M$ (respectively, $f:M\to X$) with $X\in\mathcal{P}(R)$ is called a \emph{right} (respectively, \emph{left}) \emph{projective approximation} of $M$ provided that every $R$-module homomorphism $g:Y\to M$ (respectively, $g:M\to Y$) with $Y\in\mathcal{P}(R)$ factors through $f$, that is, $g=f\circ h$ (respectively, $g=h\circ f$) for some $R$-module homomorphism $h:Y\to X$ (respectively, $h:X\to Y$). 

A right (respectively, left)  projective approximation $f:X\to M$ (resp. $f:M\to X$) is called \emph{minimal} if every endomorphism $g:X\to X$ satisfying $f = f\circ g$ (resp. $f = g\circ f$) is an automorphism. 

Note that a right  projective approximation (respectively,  a minimal right projective approximation) is nothing but a surjective homomorphism from a projective $R$-module (respectively, a projective cover). Note also that, if $M\in \md(R)$ and $\phi:P\twoheadrightarrow M^*$ is an epimorphism with $P\in\mathcal{P}(R)$, then $e_M\circ\phi^*$ is a left projective approximation of $M$.
\end{chunk}
\begin{chunk}\label{a1}\textbf{Auslander transpose and (co)syzygy.}
Let $M\in \md(R)$ that has a right projective approximation  $P_0\overset{\partial_0}{\to} M$. Then the kernel of $\partial_0$ is called the first
\emph{syzygy} of $M$; it is denoted by $\Omega^1M$ and unique up to projective equivalence. Inductively,
we define the $n$-th syzygy module of $M$ as $\Omega^nM:= \Omega^1(\Omega^{n-1}M)$ for all $n\geq 1$. We set, by convention, $\Omega^0M=M$.

Let  $P_1\overset{\partial_1}{\rightarrow}P_0 \overset{\partial_0}{\to} M\to0$ be a finite projective presentation of $M$. Then the \emph{transpose} of $M$, denoted by $\Tr M$, is $\coker {\partial_1}^*$ given in the following exact sequence
\begin{equation} \tag{\ref{a1}.1}
0\rightarrow M^*\rightarrow P_0^*\overset{\partial_1^{\ast}}{\rightarrow} P_1^*\rightarrow \Tr M\rightarrow 0.
\end{equation}
Note that $M^*\approx\Omega^2\Tr M$. Note also that $\Tr M$  is unique, up to projective equivalence, and the minimal projective
presentations of $M$ represent isomorphic transposes of $M$. 

For every $M\in \md(R)$  and $N\in\Md(R)$, there exists the following exact sequence:
\begin{equation}\tag{\ref{a1}.2}
0\rightarrow\Ext^1_R(\Tr M,N)\rightarrow M\otimes_RN\rightarrow\Hom_R(M^*,N)\to\Ext^2_R(\Tr M,N)\to0,
\end{equation}
where the middle map is the evaluation map \cite[2.6]{AB}. In  particular, setting $N=R$, we see that the canonical
map $M\to M^{**}$ is part of the exact sequence
\begin{equation}\tag{\ref{a1}.3}
0\rightarrow\Ext^1_R(\Tr M,R)\rightarrow M\rightarrow M^{**}\to\Ext^2_R(\Tr M,R)\to0.
\end{equation}
Also, there is a $4$-term exact sequence \cite[2.8]{AB}:
\begin{equation}\tag{\ref{a1}.4}
\Tor_2^R(\Tr\Omega^nM,N)\rightarrow\Ext^n_R(M,R)\otimes_RN\rightarrow\Ext^n_R(M,N)\rightarrow\Tor_1^R(\Tr\Omega^nM,N)\rightarrow0.
\end{equation}

Suppose $M\in\md(R)$ equipped with a left  projective approximation $M\overset{\partial_{-1}}{\to} P_{-1}$. Then we call the cokernel of $\partial_{-1}$ the first
\emph{cosyzygy} of $M$ and denote it by $\Omega^{-1}M$. Inductively, we define the $n$-th cosyzygy module
of $M$ as $\Omega^{-n}M:= \Omega^{-1}(\Omega^{-(n-1)}M)$ for all $n\geq 1$. It follows that $\Omega^{-i}M\approx\Tr\Omega^i\Tr M$ for all $i\geq 1$; see \cite[2.4]{ST}. Hence, by replacing $M$ with $\Tr M$ in  (\ref{a1}.4), we obtain the following exact sequence:
\begin{equation}\tag{\ref{a1}.5}
\Tor_2^R(\Omega^{-n}M,N)\rightarrow\Ext^n_R(\Tr M,R)\otimes_RN\rightarrow\Ext^n_R(\Tr M,N)\rightarrow\Tor_1^R(\Omega^{-n}M,N)\rightarrow0.
\end{equation}

For each $M\in\md(R)$ and integer $i\geq 1$, it follows by the definition that there exists an exact sequence:
\begin{equation}\tag{\ref{a1}.6}
0\to\Ext^i_R(M,R)\to\Tr\Omega^{i-1}M\to X\to0,
\end{equation}
where $X\approx\Omega\Tr\Omega^{i}M$. By replacing $M$ with $\Tr M$ in  (\ref{a1}.6) and using the fact that $\Omega^{-i}M\approx\Tr\Omega^i\Tr M$ for all $i\geq 1$, we obtain the 
following fact that will be used throughout the paper:
\begin{equation}\tag{\ref{a1}.7}
\text{If } \Ext^i_R(\Tr M,R)=0 \text{ for some } i\geq 1, \text{ then } \Omega^{-(i-1)}M\approx\Omega\Omega^{-i}M.
\end{equation}
\end{chunk}

\begin{chunk}\textbf{Complete intersection and Gorenstein dimensions.}\label{HD} The notion of Gorenstein dimension was initially introduced by Auslander \cite{A1} and subsequently developed by Auslander and Bridger in \cite{AB}. 
	
An $R$-module $M$ is called \emph{totally reflexive} provided that the natural map $M\rightarrow M^{**}$ is an isomorphism and $\Ext^i_R(M,R)=0=\Ext^i_R(M^*,R)$ for all $i\geq 1$. The \emph{Gorenstein dimension} of $M$, denoted $\G-dim_R(M)$, is defined to be the infimum of all nonnegative integers $n$, such that there exists an exact sequence $$0\rightarrow G_n\rightarrow\cdots\rightarrow G_0\rightarrow  M \rightarrow 0,$$ in which each $G_i$ is a totally reflexive $R$-module.

Every finitely generated module over a Gorenstein ring has finite Gorenstein dimension. Moreover, if $R$ is local and $\G-dim_R(M)<\infty$, then it follows that $\G-dim_R(M)=\depth R-\depth_R(M)$; see \cite[4.13]{AB}.

A \emph{quasi-deformation} of $R$ is a diagram $R\rightarrow A\twoheadleftarrow Q$ of local homomorphisms, in which $R\rightarrow A$ is faithfully flat, and $A\twoheadleftarrow Q$ is surjective with kernel generated by a regular sequence. The \emph{complete intersection dimension} of $M$, introduced by Avramov, Gasharov and Peeva \cite{AGP}, is:
	$$\CI_R(M)=\inf\{\pd_Q(M\otimes_RA)-\pd_Q(A)\mid R\rightarrow A\twoheadleftarrow Q \text{ is a quasi-deformation}\}.$$
Therefore, an $R$-module $M$ has finite complete intersection dimension if there exists a quasi-deformation $R \rightarrow A \twoheadleftarrow Q$ for which $\pd_Q(M\otimes_RA)$ is finite.

Note that, if $M$ is finitely generated $R$-module, then it follows $\G-dim_R(M)\leq \CI_R(M)\leq \pd_R(M)$, and $\CI_R(M)<\infty$ if $R$ is a complete intersection ring. Moreover, if $R$ is local and $\CI_R(M)<\infty$, then one has by \cite[1.4]{AGP} that
\begin{equation}\tag{\ref{HD}.1}
\G-dim_R(M)=\CI_R(M)=\depth R-\depth_{R}(M).
\end{equation}
	
	
\end{chunk}

\begin{chunk}\label{cx}\textbf{Complexity. }
Assume $(R,\fm)$ is local and $M,N\in\md(R)$. Then the \emph{complexity} of the pair $(M,N)$, defined by Avramov and Buchweitz \cite{AvBu}, is:
$$\cx_R(M,N)=\inf\{ b\in\NN \mid \exists a\in\RR \text{ such that } \nu_R(\Ext^n_R(M,N))\leq a n^{b-1} \text{ for all } n\gg 0\},$$
where $\nu_R(-)$ denotes the minimal number of generators. Accordingly, the complexity $\cx_R(M)$ of $M$, initially introduced by Avramov \cite{Av1} in local algebra, can be given as $\cx_R(M,k)$. 

Avramov, Gasharov and Peeva  \cite[5.3]{AGP} proved that every module of finite complete intersection dimension also has finite complexity.

Note, $\cx_R(M,N)=0$ if and only if $\Ext^i_R(M,N)=0$ for all $i\gg0$. Also, it follows by the definition that $\pd_R(M)<\infty$ if and only if $\cx_R(M)=0$, and $M$ has bounded Betti numbers if and only if $\cx_R(M)\leq1$. 

If $R$ is a complete intersection, then one has \cite[5.7]{AvBu}:
\begin{equation}\tag{\ref{cx}.1}
\cx_R(M,N)=\cx_R(N,M)\leq\min\{\cx_R(M),\cx_R(N)\}\leq\cod R.
\end{equation}

\end{chunk}

\begin{chunk}\label{df}\textbf{Depth formula.}  If $R$ is local, then a pair $(M,N)$ in $\md(R)$ is said to satisfy the \emph{depth formula} provided that the following equality holds:
$$\depth_R(M)+\depth_R(N)=\depth R+\depth_R(M\otimes_RN).$$
\end{chunk}

Auslander \cite[1.2]{A2} proved that, if $(M,N)$ is Tor-independent (i.e., $\Tor_i^R(M,N)=0$ for all $i\geq 1$) and $\pd_R(M)<\infty$, then the depth formula holds for $(M,N)$.
Auslander's result has been extended by Huneke and Wiegand for complete intersection rings: Tor-independent modules over complete intersection rings satisfy the depth formula; see \cite[2.5]{HW}.
More generally, one has: 

\begin{thm}\label{t4}(Araya and Yoshino \cite{AY}) Assume $R$ is a local ring and $M, N\in\md(R)$. If $\CI_R(M)<\infty$ and $\Tor_i^R(M,N)=0$ for all $i\geq 1$, then $(M,N)$ satisfies the depth formula; see \cite[2.5]{AY}.
\end{thm}


\begin{chunk}\label{depen}\textbf{Dependency formula of Jorgensen.} Assume $R$ is local and $M,N\in\md(R)$. If $\CI_R(M)<\infty$ and $\Tor_{i}^R(M,N)=0$ for $i\gg0$, then it follows from  \cite[2.2]{J2} that: 
	\begin{equation}\tag{\ref{depen}.1}
		\sup\{i\mid\Tor_{i}^R(M,N)\neq0\}=\sup\{\depth R_\fp-\depth_{R_{\fp}}(M_\fp)-\depth_{R_{\fp}}(N_\fp)\mid\fp\in\Spec R\}.	
	\end{equation}
	\end{chunk}

 
\begin{chunk}\label{lsc}\textbf{Specialization-closed subsets of the spectrum \cite{sga2}.} A subset $\mz\subset\Spec R$ is called \emph{specialization-closed} provided that the following condition  holds: 

If $\fp, \fq \in \Spec(R)$, where $\fp \in \mz$ and $\fp \subseteq \fq$, then it follows that $\fq\in \mz$.
\end{chunk}

We collect some  examples of specialization-closed subsets of $\Spec R$:

\begin{eg} $\phantom{}$ 
		\begin{enumerate}[\rm(i)]
			\item{Every closed subset of $\Spec(R)$ with respect to Zariski toplogy is specialization-closed.}
			\item{If $R$ is a domain, then $\Spec(R)\setminus\{0\}$ is specialization-closed subset of $\Spec(R)$.}
			\item{ Let $0\leq n\leq\dim R$. Then $\{\fp\in\Spec(R)\mid\Ht(\fp)\geq n\}$ is a specialization-closed subset of $\Spec(R)$.}
		\end{enumerate}
\end{eg}

One of the aims of this paper is to extend the following result of Auslander and Lichtenbaum \cite{A2, L}: if $R$ is a regular local ring, $M, N\in\md(R)$ and $M\otimes_RN$ is torsion-free, then $M$ and $N$ are Tor-independent torsion-free modules.  Our technique relies upon using local cohomology theory with respect to specialization-closed subsets \cite{sga2}.

\begin{chunk}\label{LC}\textbf{Local cohomology with respect to specialization-closed subsets.} Let $\mz$ be a specialization-closed subset of $\Spec R$. For each $M\in\Md(R)$, we define the following submodule of $M$:
$$\Gamma_{\mz}(M)=\{m\in M\mid\Supp_R(Rm)\subseteq\mz\}.$$
For $i\in\NN_0$, the $i$-th right derived functor of $\Gamma_{\mz}(-)$, denoted
by $\hh^i_{\mz}(-)$, is referred to as the $i$-th \emph{local cohomology functor with respect to $\mz$}. 

If $\mz$ is a closed subset of $\Spec R$, i.e., $\mz=\V(\fa)$ for some ideal $\fa$ of $R$, then we denote $\Gamma_{\mz}(-)$ (respectively, $\hh^i_{\mz}(-)$) by $\Gamma_{\fa}(-)$ (respectively, $\hh^i_{\fa}(-)$). 

$M\in\Md(R)$ is called \emph{torsion-free with respect to $\mz$} provided that $\Gamma_{\mz}(M)=0$, and is called \emph{ torsion with respect to $\mz$} precisely
when $\Gamma_{\mz}(M)=M$. Note that $M$ is torsion with respect to $\mz$ if and only if $\Supp_R(M)\subseteq\mz$.
\end{chunk}

We set, for a specialization-closed subset $\mz$ of $\Spec R$, that $\Sigma=\{\fa\triangleleft R\mid\V(\fa)\subseteq\mz\}$, and proceed by collecting some of the basic properties of local cohomology modules with respect to specialization-closed subsets in the following; for details, we refer the reader to \cite{Bsh,sga2}.

\begin{thm} \label{LC} Let $\mz$ and $\mathcal{W}$ be specialization-closed subsets of $\Spec(R)$. Then,
\begin{enumerate}[\rm(i)]
	\item $\Gamma_{\mz}(M)=\bigcup_{\fa\in\Sigma}\Gamma_{\fa}(M)$. 
	\\
	\item Each short exact sequence $0\to M_1\to M_2\to M_3\to 0$ in $\Md(R)$ yields a  long exact sequence: $$\cdots\to \hh^i_{\mz}(M_1)\to \hh^i_{\mz}(M_2)\to \hh^i_{\mz}(M_3)\to \hh^{i+1}_{\mz}(M_1)\to \cdots.$$ 
	\item If $M$ is torsion with respect to $\mz$, i.e., $\Gamma_{\mz}(M)=M$, then $\hh^i_{\mz}(M)=0$ for all $i\geq 1$.	
	\item $\Gamma_{\mz}(\Gamma_{\mathcal{W}}(M))=\Gamma_{\mz\cap\mathcal{W}}(M)=\Gamma_{\mathcal{W}}(\Gamma_{\mz}(M))$.
\end{enumerate}	
\end{thm}

The following results are used throughout the paper; although they are well-known, we recall them for the convenience of the reader:

\begin{chunk} \textbf{The local duality and Grothendieck's (non-)vanishing theorem.}\label{LD}
\begin{thm}\cite[Theorem 3.5.7]{BH}\label{ld1}
Let $(R,\fm)$ be a local ring and let $M\in\md(R)$ be a module of depth $t$ and dimension $d$. Then
\begin{enumerate}[\rm(i)]
	\item $\hh^i_{\fm}(M)=0$ for $i<t$ and $i>d$.
	\item $\hh^t_{\fm}(M)\neq0$ and $\hh^d_{\fm}(M)\neq0$.
\end{enumerate}
\end{thm}
\begin{thm}\cite[Corollary 3.5.9]{BH}\label{Ld}
	Let $(R,\fm,k)$ be a Cohen-Macaulay local ring of dimension $d$ with canonical module $\omega_R$. Then for all $M\in\md(R)$ and all integers $i$ there
	exist natural isomorphisms $$\hh^i_{\fm}(M)\cong\Hom_R(\Ext^{d-i}_R(M,\omega_R),\E_R(k)),$$
	where $\E_R(k)$ denotes the injective envelope of the reside field.
\end{thm}
\end{chunk}


\begin{chunk}\label{gsc}\textbf{Grade of a module with respect to a specialization-closed subset.}

Let $\mz\subset \Spec R$ be a specialization-closed subset and let $M\in\Md(R)$. Then the {\em  grade of $M$ with respect to $\mz$} is:
$$\gr_R(\mz,M)=\inf\{i\in \NN_{0}\mid\hh_{\mz}^{i}(M)\neq0\}.$$	
If $\mz=\V(\fa)$ for some ideal $\fa$ of $R$, then we denote $\gr_R(\mz,M)$ by $\gr_R(\fa,M)$.  Note that, if $M\in\md(R)$, then 
$\gr(\fa,M)$ is equal to the maximal length of an $M$-regular sequence contained in $\fa$.
\end{chunk}


The next result, for the case where $\mz$ is a closed subset of $\Spec R$, was initially proved by Grothendieck; see \cite[III.2.9]{sga2}. This fact can be viewed as a generalization of \cite[4.1]{TTY}.
As the result plays an important role for our arguments in this paper, we provide the details.

\begin{prop}\label{grz} Let $\mz\subset \Spec R$ be  specialization-closed and $M\in\md(R)$. Then it follows that
	$$\gr_R(\mz,M)=\inf\{\depth(M_{\fp})| \fp\in\mz\}.$$
\end{prop}

\begin{proof}
	Let $0\to M\to E^0(M)\overset{\partial_0}{\to}E^1(M)\overset{\partial_{1}}{\to}\cdots$ be a minimal injective resolution of $M$, where we have that $E^i(M)= \underset{\fq\in\Spec R}{\bigoplus} E(R/\fq)^{\mu_i(\fq,M)}$. Apply the functor $\Gamma_{\mz}(-)$ to the above resolution, we get the complex
	\begin{equation}\tag{\ref{grz}.1}
	0\longrightarrow\Gamma_{\mz}(M)\longrightarrow\Gamma_{\mz}(E^0(M))
	\overset{\Gamma_{\mz}(\partial_0)}{\longrightarrow}\Gamma_{\mz}(E^1(M))
	\overset{\Gamma_{\mz}(\partial_1)}{\longrightarrow}\cdots.
	\end{equation}
	Set $n:=\inf\{\depth(M_{\fp})| \fp\in\mz\}$. Hence, there exists a prime ideal $\fp_0\in\mz$ such that $\depth_{R_{\fp_0}}(M_{\fp_0})=n$. It follows from \cite[Theorem 2]{Ro} that
\begin{equation}\tag{\ref{grz}.2}
\mu_n(\fp_0,M)\neq0=\mu_i(\fp,M) \text{ for all } \fp\in\mz \text{ and } i<n.
\end{equation}
Note that $\Supp_R(E(R/\fp))\subseteq\V(\fp)$ for all $\fp\in\Spec R$. As $\mz$ is specialization-closed, we see that
    \begin{equation}\tag{\ref{grz}.3}
	\Gamma_{\mz}(E(R/\fq))=\left\lbrace
	\begin{array}{cl}
	E(R/\fq)\ \ & \text{ if}\ \ \fq\in\mz\\
	0\ \ & \text{ if }\ \ \fq\notin\mz.\\
	\end{array}\right.
	\end{equation}
	It follows from (\ref{grz}.3) that
	\begin{equation}\tag{\ref{grz}.4}
	\Gamma_{\mz}(E^i(M))\cong\bigoplus_{\fq\in\Spec R}\Gamma_{\mz}(E(R/\fq))^{\mu_i(\fq,M)}\cong\bigoplus_{\fq\in\mz}\Gamma_{\mz}(E(R/\fq))^{\mu_i(\fq,M)}.
	\end{equation}
	Hence, by (\ref{grz}.2) and (\ref{grz}.4) we have
	\begin{equation}\tag{\ref{grz}.5}
	\Gamma_{\mz}(E^n(M))\neq0=\Gamma_{\mz}(E^i(M)) \ \text{ for  all } i<n.
	\end{equation}
	Therefore, by (\ref{grz}.1) and (\ref{grz}.5) we see that $\hh_{\mz}^{i}(M)\cong\ker(\Gamma_{\mz}(\partial_i))/\im(\Gamma_{\mz}(\partial_{i-1}))=0$ for all $i<n$. In other words,
	$\gr_R(\mz,M)\geq n=\inf\{\depth(M_{\fp})| \fp\in\mz\}$. Thus, it is enough to show that $\hh^n_{\mz}(M)\neq0$. By using (\ref{grz}.1) and (\ref{grz}.5) we obtain the following exact sequence
	\begin{equation}\tag{\ref{grz}.6}
	0\lo \hh_{\mz}^{n}(M)\lo \Gamma_{\mz}(E^n(M))\overset{\Gamma_{\mz}(\partial_{n})}{\lo} \Gamma_{\mz}(E^{n+1}(M)).
	\end{equation}
	Note that $E^n(M)$ is an essential extension of $\im\partial_{n-1}$. In other words, $\im\partial_{n-1}\cap L\neq0$ for any non-zero submodule $L$ of 
	$E^n(M)$. It follows from (\ref{grz}.5) that $\Gamma_{\mz}(E^n(M))$ is a non-zero submodule of $E^n(M)$. Therefore, by the exact sequence (\ref{grz}.6) we obtain the following:
	$$\hh_{\mz}^{n}(M)=\ker(\Gamma_{\mz}(\partial_{n}))=\ker\partial_{n}\cap(\Gamma_{\mz}(E^n(M)))=\im\partial_{n-1}\cap(\Gamma_{\mz}(E^n(M)))\neq0,$$
	as desired.
\end{proof}

The next result can be found in \cite[1.4.19]{BH} when $\mz$ is a closed subset of $\Spec R$.

\begin{lem}\label{gr}
Let $\mz\subset \Spec R$ be a specialization-closed subset, and let $M\in\md(R)$ and $N\in\Md(R)$. Then it follows that $\gr_R(\mz,\Hom_R(M,N))\geq\min\{2,\gr_R(\mz,N)\}$.	
\end{lem}

\begin{proof}
Let	$F_1\to F_0\to M\to0$ be a free presentation of $M$.
Applying the functor $\Hom_R(-,N)$ to the above exact sequence, we obtain the following exact sequence $0\to\Hom_R(M,N)\to\Hom_R(F_0,N)\to\Hom_R(F_1,N)$ which induces the following exact sequences: 
$$0\to\Hom_R(M,N)\to\Hom_R(F_0,N)\to X\to0\ \ \text{and}\ \ 0\to X\to\Hom_R(F_1,N).$$
By Theorem \ref{LC}(ii), the above exact sequences induce the following exact sequences:
\begin{equation}\tag{\ref{gr}.1}
\Gamma_{\mz}(\Hom_R(M,N))\hookrightarrow\Gamma_{\mz}(\Hom_R(F_0,N))\to\Gamma_{\mz}(X)\to\hh^1_{\mz}(\Hom_R(M,N))\to\hh^1_{\mz}(\Hom_R(F_0,N)),
\end{equation}
\begin{equation}\tag{\ref{gr}.2}
0\to\Gamma_{\mz}(X)\to\Gamma_{\mz}(\Hom_R(F_1,N)).
\end{equation}
Now the assertion follows easily from   (\ref{gr}.1) and (\ref{gr}.2).
\end{proof}


\begin{chunk}\label{SL}{\bf Serre's condition and torsion submodule.} If $M\in\md(R)$ and $n\geq 0$ is an integer, then $M$ is said to satisfy \emph{Serre's condition} $(S_n)$ provided that $\depth_{R_{\fp}}(M_\fp)\geq\min\{n,\Ht\fp\}$ for all $\fp\in\Spec R$ (recall that, by convention, $\depth_R(0)=\infty$.)

If $N\in\Md(R)$, then the \emph{torsion submodule} of $N$, denoted by $\T(N)$, is the kernel of the natural homomorphism $N\to Q(R)\otimes_RN$ where $Q(R)$ is the total quotient ring of $R$. The module $N$ is said to be \emph{torsion-free} if $\T(N)=0$, and  \emph{torsion} if $\T(N)=N$. Note that, if $R$ is unmixed, i.e., all associated prime ideals of $R$ are minimal, e.g., $R$ is reduced, and $N\in \md R$, then $\T(N)=0$ if and only if $N$ satisfies $(S_1)$. 

If $M\in\Md(R)$ and $n\geq 1$, then $M$ is called $n$-\emph{torsion-free} if $\Ext^i_R(\Tr M,R)=0$ for all $i=1, \ldots, n$.

\end{chunk}

In the following we investigate the relation between Serre's condition and the vanishing of local cohomology with respect to specialization-closed subsets of $\Spec R$.

\begin{prop}\label{Serre} Let $M\in\Md(R)$ and let $N\in\md(R)$. 
\begin{enumerate}[\rm(i)]
\item If $R$ is unmixed and $\mz=\{\fp\in\Spec(R) \mid \Ht\fp \geq 1\}$, then it follows that $\Gamma_{\mz}(M)=\T(M)$.
\item Let $\mz=\{\fp\in\Spec R\mid\Ht\fp\geq n\}$ for some integer $n\geq 1$. Then $N$ satisfies $(S_n)$ if and only if $\hh^i_{\mz}(N)=0$ for all $i<n$ and $N_\fp$ is a maximal Cohen-Macaulay $R_\fp$-module for all $\fp\in\Spec R\setminus\mz$. 
 \item Assume $R$ is local and Cohen-Macaulay, $\mz$ is a specialization-closed subset of $\Spec R$, and $n\geq 1$ is an integer. Assume further that $\Supp_R(\Ext^i_R(\Tr N,R))\subseteq\mz$ for $1\leq i\leq n$ (e.g., $N_\fp$ is totally reflexive for all $\fp\in\Supp_R(N)\setminus\mz$.) If $\gr_R(\mz,N)\geq n$, then $N$ is $n$-torsion-free. 
\end{enumerate}	
\end{prop}

\begin{proof} Note that part (i) follows easily form the definition, and part (ii) follows from Proposition \ref{grz}. Hence we proceed to prove part (iii). 

Set $t=\inf\{i\geq 1 \mid\Ext^i_R(\Tr N,R)\neq0\}$ and let $\fp\in\Ass_R(\Ext^t_R(\Tr N,R))$. Assume contrarily that $t\leq n$. Hence by our assumption $\fp\in\mz$. As $\Ext^i_R(\Tr N,R)_\fp=0$ for all $0<i<t$, we have by (\ref{a1}.7) that
$\Omega^{-(i-1)}_{R_\fp}N_\fp\approx\Omega_{R_\fp}\Omega^{-i}_{R_\fp}N_\fp$ for all $0<i<t$. Hence, we obtain the following
\begin{equation}\tag{\ref{Serre}.1}
N_\fp\approx\Omega^{t-1}_{R_\fp}\Omega^{-(t-1)}_{R_\fp}N_\fp.
\end{equation}
It follows from the inclusion $\Ext^t_R(\Tr N,R)\hookrightarrow\Omega^{-(t-1)}N$ (see (\ref{a1}.6)) that $\fp\in\Ass_R(\Omega^{-(t-1)}N)$. Now it is easy to see that $\depth_{R_{\fp}}(\Omega^{t-1}_{R_\fp}\Omega^{-(t-1)}_{R_\fp}N_\fp)=t-1$ and so by (\ref{Serre}.1) $\depth_{R_{\fp}}(N_\fp)=t-1$. This is a contradiction because Proposition \ref{grz} implies that $t\leq n\leq\gr_R(\mz,N)\leq\depth_{R_\fp}(N_\fp)$.
\end{proof}	


\begin{chunk}\label{Tate}{\bf Tate (co)homology.} We call a (homologically indexed) complex \emph{acyclic} if it has zero homology. An acyclic complex $T$ of free $R$-modules is called \emph{totally
acyclic} if the dual complex $\Hom_R(T,R)$ is also acyclic. For $M\in\md(R)$, it follows that $M$ is totally reflexive if and only if there is a totally acyclic complex $T$ with $M\cong\coker(T_1\to T_0)$. 

A \emph{complete resolution} of $M\in \md(R)$ is a diagram $T\overset{\tau}{\to}P\overset{\simeq}{\to} M$, where $P\overset{\simeq}{\to}M$ is a projective resolution, $T$ is a totally acyclic complex of free $R$-modules, and $\tau_i$ is an isomorphism for all $i\gg 0$.  It is known that $M$ has finite Gorenstein dimension if and only if it has a complete resolution \cite{AvM}.

Suppose $M\in \md(R)$ is equipped with a complete resolution $T\to P\to M$. For each $N\in \Md(R)$ and for each $i\in\ZZ$, the \emph{Tate (co)homology} of $M$ and $N$ is defined as:
$$\widehat{\Tor}^R_i(M,N)=\hh_i(T\otimes_RN)\ \ \text{ and } \ \ \widehat{\Ext}^i_R(M,N)=\hh^i(\Hom_R(T,N)).$$

In the following we catalog some basic properties of Tate (co)homology; for details, please see \cite{AvBu,AvM,CJ}.

\begin{thm}\label{Tate hom} Let $M \in \md(R)$, $N \in \Md(R)$, and assume $\G-dim_R(M)<\infty$. 
\begin{enumerate}[\rm(i)]	
\item{If $\pd_R(M)<\infty$, then $\widehat{\Tor}_i^R(M,N)=0=\widehat{\Ext}^i_R(M,N)$ for all $i\in\ZZ$.}
\item{If $\G-dim_R(M)=0$, i.e., if $M$ is totally reflexive, then $\widehat{\Tor}_i^R(M,N)\cong\widehat{\Ext}^{-i-1}_R(M^*,N)$ for all $i\in\ZZ$.}
\item{$\widehat{\Tor}^R_{i+n}(M,N)\cong\widehat{\Tor}^R_i(\Omega^n M,N)$ and $\widehat{\Ext}^{i+n}_R(M,N)\cong\widehat{\Ext}^i_R(\Omega^n M,N)$ for all $i\in\ZZ$ and for all $n\geq0$.}
\item {$\widehat{\Tor}^R_i(M,N)\cong\Tor_i^R(M,N)$ and $\widehat{\Ext}^i_R(M,N)\cong\Ext^i_R(M,N)$ for all $i>\G-dim_R(M).$
	}
\end{enumerate}	
\end{thm}

Next is a property for the vanishing of Tate homology for modules of finite complete intersection dimension:

\begin{thm}\label{Tate2}(\cite[4.9]{AvBu}) Let $M\in\md(R)$ and let $N\in\Md(R)$. Assume $\CI_R(M)<\infty$. Then,
$$\widehat{\Tor}_i^R(M,N)=0 \ \text{ for all } i\in\ZZ \Longleftrightarrow
\widehat{\Tor}_i^R(M,N)=0\ \text{ for all } i\ll0 \Longleftrightarrow
\Tor_i^R(M,N)=0 \ \text{ for all } i\gg0.$$
\end{thm}
\end{chunk}


\section{Tor-rigidity and the vanishing of local cohomology}

In this section we prove one of the main results of this paper, namely Theorem \ref{t2}, and extend the following celebrated result of Auslander and Lichtenbaum:

\begin{thm}(Auslander and Lichtenbaum \cite{A2, L})\label{AuLi}
Let $R$ be a regular local ring and let $M, N\in\md(R)$. If $M\otimes_RN$ is nonzero and torsion-free, then $M$ and $N$ are torsion-free and Tor-independent. 	 
\end{thm}	

Besides the generalization we obtain in Theorem \ref{t2}, a consequence of our argument yields a new proof of Theorem \ref{AuLi}, which we record at the end of this section.

The proof of Theorem \ref{t2} requires substantial preparation; we proceed and prove several lemmas prior to giving a proof for Theorem \ref{t2}. 

The first item of the following lemma is to be compared with \cite[4.1]{bv}.
\begin{lem}\label{t1} Let $\mz\subset \Spec R$ be a specialization-closed subset, $N\in\Md(R)$ and let $M\in\md(R)$. Assume $t$ is an integer such that $2\leq t\leq\gr_R(\mz,N)$ 	
	and $\Supp_R(\Ext^i_R(M,N))\subseteq\mz$ for all $i=1, 
\ldots, t-1$. Then the following hold:
	\begin{enumerate}[\rm(i)]
		\item
		$\Ext^i_R(M,N)\cong\hh_{\mz}^{i+1}(\Hom_R(M,N))$ for all $i=1, \ldots, t-2$.
		\item There is an injection $\Ext^{t-1}_R(M,N)\hookrightarrow\hh_{\mz}^{t}(\Hom_R(M,N))$.
	\end{enumerate}
\end{lem}

\begin{proof}
	We prove both part (i) and part (ii) simultaneously. 

Recall that $M$ admits a free resolution
	$$\cdots\longrightarrow F_{1}\overset{\partial_{1}}{\longrightarrow} F_0\overset{\partial_0}{\longrightarrow}M\longrightarrow0,$$
	where $F_i$ is a finitely generated free $R$-module for all $i\geq0$. Applying the functor $(-)^{\triangledown}:=\Hom_R(-,N)$ to the above resolution we obtain the following complex:
	$$0\longrightarrow M^{\triangledown}\overset{\partial^{\triangledown}_0}{\longrightarrow}(F_0)^{\triangledown}\overset{\partial^{\triangledown}_1}{\longrightarrow}\cdots
	\overset{\partial^{\triangledown}_{i-1}}{\longrightarrow}(F_{i-1})^{\triangledown}\overset{\partial^{\triangledown}_i}
	{\longrightarrow}(F_{i})^{\triangledown}\overset{\partial^{\triangledown}_{i+1}}{\longrightarrow}(F_{i+1})^{\triangledown}\longrightarrow\cdots,$$
	which induces the following exact sequences:
	$$0\longrightarrow\Ext^i_R(M,N)\longrightarrow T_i\longrightarrow X_i\longrightarrow0,$$
	$$0\longrightarrow X_i\longrightarrow (F_{i+1})^{\triangledown}\longrightarrow T_{i+1}\longrightarrow0,$$
	where $T_i:=\frac{(F_i)^{\triangledown}}{\im(\partial_i^{\triangledown})}$ and $X_i:=\im(\partial_{i+1}^{\triangledown})$ for all $i>0$. For each $i\geq0$, we note that $\gr_R(\mz,(F_i)^{\triangledown})=\gr_R(\mz,N)\geq t$. By our assumption, $\Ext^i_R(M,N)$ is torsion with respect to $\mz$ for all $0<i<t$. Then
	$\Ext^i_R(M,N)=\Gamma_{\mz}(\Ext^i_R(M,N))$ and $\hh^j_{\mz}(\Ext^i_R(M,N))=0$ for all $j>0$ and $0<i<t$ (see Theorem \ref{LC}(iii)).
	Applying the functor $\Gamma_{\mz}(-)$ to the above exact sequences, we get the isomorphisms
	\begin{equation}\tag{\ref{t1}.1}
	\Ext^i_R(M,N)=\Gamma_{\mz}(\Ext^i_R(M,N))\cong\Gamma_{\mz}(T_i) \text{ for all } 0<i<t,
	\end{equation}
	\begin{equation}\tag{\ref{t1}.2}
	\hh_{\mz}^j(T_i)\cong\hh_{\mz}^{j}(X_i) \text{ for all } 0<i<t\textit{ and }j>0,
	\end{equation}
	\begin{equation}\tag{\ref{t1}.3}
	\hh_{\mz}^j(T_{i+1})\cong\hh_{\mz}^{j+1}(X_i) \text{ for all } 0\leq j<n-1\textit{ and }i\geq0.
	\end{equation}
	It follows from (\ref{t1}.1), (\ref{t1}.2) and (\ref{t1}.3) that
	\begin{align*}\tag{\ref{t1}.4}
	\Ext^i_R(M,N)&\cong\Gamma_{\mz}(T_i)\cong\hh_{\mz}^{1}(X_{i-1})\cong\hh_{\mz}^{1}(T_{i-1})\\
	\cong\hh_{\mz}^{2}(X_{i-2})
	&\cong\hh_{\mz}^{2}(T_{i-2})
	\cong\cdots\cong\hh_{\mz}^{i-1}(T_1),
	\end{align*}
	for each $0<i<t$.
	Applying the functor $\Gamma_{\mz}(-)$ to the exact sequences,  $0\rightarrow M^{\triangledown}\rightarrow (P_0)^{\triangledown}\rightarrow X_0\rightarrow0$ and $0\rightarrow X_0\rightarrow(P_1)^{\triangledown}\rightarrow T_1\rightarrow0$, we obtain the isomorphism
	\begin{equation}\tag{\ref{t1}.5}
	\hh_{\mz}^{i-1}(T_1)\cong\hh_{\mz}^{i+1}(M^{\triangledown}) \text{ for all } 0<i<n-1.
	\end{equation}
	Also, there is an injection
	\begin{equation}\tag{\ref{t1}.6}
	0\longrightarrow\hh_{\mz}^{t-1}(X_0)\longrightarrow\hh_{\mz}^{t}(M^\triangledown).
	\end{equation}
	The first assertion follows from (\ref{t1}.4) and (\ref{t1}.5). Note that by (\ref{t1}.3) and (\ref{t1}.4)
	\begin{align*}\tag{\ref{t1}.7}
	\Ext^{t-1}_R(M,N)\cong\hh_{\mz}^{t-2}(T_1)
	\cong\hh_{\mz}^{t-1}(X_0).
	\end{align*}
	The second assertion follows from (\ref{t1}.6) and (\ref{t1}.7).
\end{proof}

\begin{rmk} The injection $\Ext^{t-1}_R(M,N)\hookrightarrow\hh_{\mz}^{t}(\Hom_R(M,N))$ from Lemma \ref{t1} is not  necessarily an isomorphism: to  see this, we consider $R=k[[x,y]]$, $\fm=(x,y)$, $M=N=R$ and $\mz=\V(\fm)$. Then it follows $t=2$ and $0=\Ext^{1}_R(R,R)=\Ext^{t-1}_R(M,N)\hookrightarrow\hh_{\fm}^{t}(\Hom_R(M,N))=\hh_{\fm}^{2}(R)\neq 0$. \pushQED{\qed} 
\qedhere
\popQED	
\end{rmk}

For $M\in \md(R)$, we denote by $\NF(M)$ the \emph{non-free locus} of $M$. It is well-known that $\NF(M)$ is a closed subset of $\Spec R$. In passing, we record an immediate consequence of Lemma \ref{t1}:

\begin{cor} Let $M\in\md(R)$, $N\in\Md(R)$, and let $\fa$ be an ideal of $R$. Assume $\NF(M)\subseteq\V(\fa)$, e.g., $\fa$ is a defining ideal of $\NF(M)$. If $n$ is an integer such that $2\leq n\leq\gr_R(\fa,N)$, then the following hold:
	\begin{enumerate}[\rm(i)]
		\item$\Ext^i_R(M,N)\simeq\hh_{\fa}^{i+1}(\Hom_R(M,N))$ for all $i=1, \ldots, n-2$.
		\item There is an injection $\Ext^{n-1}_R(M,N)\hookrightarrow\hh_{\fa}^{n}(\Hom_R(M,N))$.
	\end{enumerate}
\end{cor}

Next we use Lemma \ref{t1} and obtain:

\begin{lem}\label{l2} Let $\mz\subset \Spec R$ be a specialization-closed subset, $M\in\md(R)$ and let $N\in\Md(R)$. Assume $n$ is an integer such that $0\leq n\leq\gr_R(\mz,N)$ and $\Supp_R(\Ext^i_R(\Tr M,N))\subseteq\mz$ for all $i=1, \ldots, n+1$. Then the following hold:
\begin{enumerate}[\rm(i)]
	\item $\hh^{i}_{\mz}(M\otimes_RN)\cong\Ext^{i+1}_R(\Tr M,N)$ for all $i=0, \ldots, n-1$.
	\item There is an injection
	$\Ext^{n+1}_R(\Tr M,N)\hookrightarrow\hh^{n}_{\mz}(M\otimes_RN)$.
\end{enumerate}
\end{lem}

\begin{proof}
Consider the exact sequence (\ref{a1}.2):
\begin{equation}\tag{\ref{l2}.1}
0\rightarrow\Ext^1_R(\Tr M,N)\rightarrow M\otimes_RN\rightarrow\Hom_R(M^*,N)\rightarrow\Ext^2_R(\Tr M,N)\rightarrow0.
\end{equation}
By applying the functor $\Gamma_{\mz}(-)$ to the 4-term exact sequence (\ref{l2}.1)
and noting that $\Ext^1_R(\Tr M,N)$ is torsion with respect to $\mz$ we get the injection
$\Ext^{1}_R(\Tr M,N)\hookrightarrow\Gamma_{\mz}(M\otimes_RN)$. Hence, from now on we may assume that $n>0$. The exact sequence (\ref{l2}.1) induces the following short exact sequences:
\begin{equation}\tag{\ref{l2}.2}
0\rightarrow\Ext^1_R(\Tr M,N)\rightarrow M\otimes_RN\rightarrow X\rightarrow0,
\end{equation}
\begin{equation}\tag{\ref{l2}.3}
0\rightarrow X\rightarrow\Hom_R(M^*,N)\rightarrow\Ext^2_R(\Tr M,N)\rightarrow0.
\end{equation}
As $\Ext^i_R(\Tr M,N)$ is torsion with respect to $\mz$ for all $1\leq i\leq n$, by Theorem \ref{LC}(ii) we have
\begin{equation}\tag{\ref{l2}.4}
\Gamma_{\mz}(\Ext^i_R(\Tr M,N))=\Ext^i_R(\Tr M,N) \text{ and } \hh^j_{\mz}(\Ext^i_R(\Tr M,N))=0 \text{ for all } j>0 \text{ and } 1\leq i\leq n.
\end{equation}
By Lemma \ref{gr}, $\gr_R(\mz,\Hom_R(M^*,N))\geq\min\{2,\gr_R(\mz,N)\}>0$. 
It follows from the exact sequence (\ref{l2}.3) that $\gr_R(\mz,X)>0$. In other words, $\Gamma_{\mz}(X)=0$. Applying the functor $\Gamma_{\mz}(-)$ to the exact sequences (\ref{l2}.2), (\ref{l2}.3) and using (\ref{l2}.4) we get the following isomorphisms and exact sequence:
\begin{equation}\tag{\ref{l2}.5}
\Gamma_{\mz}(M\otimes_RN)\cong\Gamma_{\mz}(\Ext^1_R(\Tr M,N))=\Ext^1_R(\Tr M,N) \ \text{ and }\ \hh^1_{\mz}(X)\cong\hh^1_{\mz}(M\otimes_RN),
\end{equation}
\begin{equation}\tag{\ref{l2}.6}
\Ext^2_R(\Tr M,N)=\Gamma_{\mz}(\Ext^2_R(\Tr M,N))\hookrightarrow\hh^1_{\mz}(X)\rightarrow\hh^1_{\mz}(\Hom_R(M^*,N)),
\end{equation}
\begin{equation}\tag{\ref{l2}.7}
\hh^{i}_{\mz}(M\otimes_RN)\cong\hh^i_{\mz}(\Hom_R(M^*,N)) \text{ for all } i>1.
\end{equation}
The case $n=1$ is clear by (\ref{l2}.5) and (\ref{l2}.6).
Let $n>1$. Therefore, $\gr_R(\mz,\Hom_R(M^*,N))\geq2$ and so $\hh^1_{\mz}(\Hom_R(M^*,N))=0$. It follows from (\ref{l2}.5) and (\ref{l2}.6) that
\begin{equation}\tag{\ref{l2}.8}
\hh^1_{\mz}(M\otimes_RN)\cong\Ext^2_R(\Tr M,N).
\end{equation}
Note that $M^*\approx\Omega^2\Tr M$. Hence, we have the following exact sequence
\begin{equation}\tag{\ref{l2}.9}
\Ext^3_R(\Tr M,N)\cong\Ext^1_R(M^*,N)\hookrightarrow\hh^2_{\mz}(\Hom_R(M^*,N))\cong\hh^2_{\mz}(M\otimes_RN),
\end{equation}
where the injection follows from Lemma \ref{t1} and the last isomorphism follows from (\ref{l2}.7). Now the assertion is clear by (\ref{l2}.5), (\ref{l2}.8) and (\ref{l2}.9) for the case $n=2$. Finally, let $n>2$. By (\ref{l2}.7) and Lemma \ref{t1}, we get the following isomorphisms:
\[\begin{array}{rl}
\hh_{\mz}^{i}(M\otimes_RN)&\cong\hh_{\mz}^{i}(\Hom_R(M^*,N))\\
&\cong\Ext^{i-1}_R(M^*,N)\\
&\cong\Ext^{i+1}_R(\Tr M,N)
\end{array}\]
for all $1<i<n$ and also we obtain an injection $\Ext^{n+1}_R(\Tr M,N)\hookrightarrow\hh^{n}_{\mz}(M\otimes_RN)$ as desired. 
\end{proof}

The next result is used throughout the paper.

\begin{lem}\label{pr} Let $M, N\in\md(R)$, and let $\X\subseteq \Spec R$. Assume $M_\fp$ is totally reflexive over $R_\fp$ for all $\fp\in\X$. Then  the following conditions are equivalent:
\begin{enumerate}[\rm(i)]
\item $\widehat{\Tor}_i^{R_\fp}(M_\fp,N_\fp)=0$ for all $i\in\ZZ$ and for all $\fp\in\X$.
\item $\Supp_R(\Ext^i_R(\Tr M,N))\bigcup\Supp_R(\Tor_i^R(M,N))\subseteq\Spec R\setminus\X$ for all $i\geq 1$.
\end{enumerate}
\end{lem}

\begin{proof}   Let $\fp\in\X$.
	Note that $M^*\approx\Omega^2\Tr M$. We apply this along with parts (ii),(iii) and (iv) of Theorem \ref{Tate hom} to deduce, for all $i\leq 0$, that $$\widehat{\Tor}_i^{R_\fp}(M_\fp,N_\fp)\cong\widehat{\Ext}^{-i-1}_{R_\fp}(M^*_\fp,N_\fp)\cong\widehat{\Ext}^{-i+1}_{R_\fp}(\Tr_{R_\fp} M_\fp,N_\fp)\cong\Ext^{-i+1}_{R_\fp}(\Tr_{R_\fp} M_\fp,N_\fp).$$ Also, since $M_\fp$ is totally reflexive over $R_\fp$,
	we know $\widehat{\Tor}_j^{R_\fp}(M_\fp,N_\fp)\cong\Tor_{j}^{R_\fp}(M_\fp,N_\fp)$ 
	for all $j\geq 1$. Hence the assertion follows.
\end{proof}

Recall that $M\in\Md(R)$  is called \emph{Tor-rigid} provided that the vanishing of a single $\Tor^R_j(M,N)$ for some $N\in\Md(R)$ and for some $j\geq 1$ forces the vanishing of $\Tor^R_i(M,N)$ for all $i\geq j$. Clearly, over a local ring, every syzygy module of the residue field is Tor-rigid. Note that, it follows from Theorem \ref{AuLi} that each finitely generated module over a regular local ring is Tor-rigid. For more examples of Tor-rigid modules, see, for example, \cite{Dao}.

The following is the key for our proof of Theorem \ref{t2}.

\begin{lem}\label{l1} Assume $R$ is local, $M, N\in\md(R)$, and $N$ is Tor-rigid. Assume further, for a specialization-closed subset $\mz\subseteq\Spec R $, that the following conditions hold:
\begin{enumerate}[\rm(i)]
\item{$M_\fp$ is a totally reflexive $R_\fp$-module for all $\fp\in\Supp_R(M)\setminus\mz$ (e.g., $\NF(M)\subseteq\mz$).}
\item {$\widehat{\Tor}_i^{R_\fp}(M_\fp,N_\fp)=0$ for all $\fp\in\Spec R\setminus\mz$ and for $i\in\ZZ$ (e.g., $\NF(M)\subseteq\mz$).}
\item{$\hh^{n}_{\mz}(M\otimes_RN)=0$ for some $n$, where $0\leq n\leq\gr_R(\mz,N)$.}
\end{enumerate}

Then the following hold:
\begin{enumerate}[\rm(a)]
	\item $\Ext^{n+1}_R(\Tr M,R)=0$. Moreover, if $\gr_R(\mz,R)\geq n+1$, then $\hh^{n}_{\mz}(M)=0$.
	\item $\Tor_j^R(\Omega^{-n}M,N)=0$ for all $j\geq 1$.
	\item If $n\geq 1$, then $\hh^{n-1}_{\mz}(M\otimes_RN)\cong\Ext^{n}_R(\Tr M,R)\otimes_RN$. Therefore, if $n\leq\gr_R(\mz,R)$, then $\hh^{n-1}_{\mz}(M\otimes_RN)\cong\hh^{n-1}_{\mz}(M)\otimes_RN$.
	\end{enumerate}
\end{lem}

\begin{proof}
We prove the statements simultaneously.

Note that by Lemma \ref{pr}, $\Supp_R(\Ext^i_R(\Tr M,N))\subseteq\mz$ for all $i>0$. It follows from Lemma \ref{l2} and assumption (iii) that $\Ext^{n+1}_R(\Tr M,N)=0$.
By (\ref{a1}.5), there exists the following exact sequence
\begin{equation}\tag{\ref{l1}.1}
\Tor_2^R(\Omega^{-(n+1)}M,N)\rightarrow\Ext^{n+1}_R(\Tr M,R)\otimes_RN \rightarrow\Ext^{n+1}_R(\Tr M,N)\twoheadrightarrow
\Tor_1^R(\Omega^{-(n+1)}M,N).
\end{equation}	
It follows from (\ref{l1}.1) that $\Tor_1^R(\Omega^{-(n+1)}M,N)=0$. As $N$ is Tor-rigid, we have 
\begin{equation}\tag{\ref{l1}.2}
\Tor_j^R(\Omega^{-(n+1)}M,N)=0 \text{ for all } j\geq 1.
\end{equation}
 Again, by the exact sequence (\ref{l1}.1), we get that $\Ext^{n+1}_R(\Tr M,R)\otimes_RN=0$ which implies that $\Ext^{n+1}_R(\Tr M,R)=0$. Note that by assumption (i) we have $\Supp_R(\Ext^i_R(\Tr M,R))\subseteq\mz$ for all $i\geq 1$; see \cite[4.9]{AB}. If $\gr_R(\mz,R)\geq n+1$, then $\hh^{n}_{\mz}(M)\cong\Ext^{n+1}_R(\Tr M,R)=0$ by Lemma  \ref{l2} and so the  assertion (a) follows.

It follows from part (a) and (\ref{a1}.7) that $\Omega^{-n}M\approx\Omega\Omega^{-(n+1)}M$. Therefore, by (\ref{l1}.2), we get $\Tor_j^R(\Omega^{-n}M,N)=0$ for all $j>0$ which proves part (b).

Now assume that $n>0$. By part (b) and (\ref{l1}.1), we get the following isomorphism:
\begin{equation}\tag{\ref{l1}.3}
	\Ext^{n}_R(\Tr M,R)\otimes_RN\cong\Ext^{n}_R(\Tr M,N).
\end{equation}
Now the assertion (c) follows from (\ref{l1}.3) and Lemma \ref{l2}.
\end{proof}

Now we are ready to state and prove the main result of this section:

\begin{thm}\label{t2} \label{t23} Assume $R$ is local, $M, N\in\md(R)$, and $N$ is Tor-rigid. Assume further, for an integer $n\geq 0$ and a specialization-closed subset $\mz\subseteq\Spec R$, that the following conditions hold:
\begin{enumerate}[\rm(i)]
\item{$M_\fp$ is totally reflexive over $R_\fp$ for all $\fp\in\Supp_R(M)\setminus\mz$.}
\item {$\widehat{\Tor}_i^{R_\fp}(M_\fp,N_\fp)=0$ for all $i\in\ZZ$ and for all $\fp\in\Spec R\setminus\mz$.}
\item{$\gr_R(\mz,M)\geq n$ and $\gr_R(\mz,N)\geq n$.}
\item{$\hh_{\mz}^n(M\otimes_RN)=0$.}
\end{enumerate}
If, either $n\leq\gr_R(\mz,R)$, or $R$ is Cohen-Macaulay, then $\hh_{\mz}^i(M\otimes_RN)=0$ for all $i\leq n$, and $M$ and $N$ are Tor-independent, i.e., $\Tor_j^R(M,N)=0$ for all $j\geq 1$.
\end{thm}

\begin{proof} We proceed by induction on the integer $n$. 

If $n=0$, then the required result follows from the assumption $\hh_{\mz}^n(M\otimes_RN)=0$ and Lemma \ref{l1}(b).  Hence we assume $n\geq 1$. Note that it suffices to show $\hh^{n-1}_{\mz}(M\otimes_RN)=0$; as then the induction hypothesis implies $\hh_{\mz}^i(M\otimes_RN)=0$ for all $i\leq n-1$ and $\Tor_j^R(M,N)=0$ for all $j\geq 1$, as required.

First suppose $n\leq\gr_R(\mz,R)$. Then, by  Lemma \ref{l1}(c), we have the following isomorphism:
\begin{equation}\tag{\ref{t2}.1}
\hh^{n-1}_{\mz}(M\otimes_RN)\cong\hh^{n-1}_{\mz}(M)\otimes_RN.
\end{equation}
As $\gr_R(\mz,M)\geq n$, we conclude from (\ref{t2}.1) that $\hh^{n-1}_{\mz}(M\otimes_RN)=0$; see \ref{gsc}. 

Next suppose $R$ is Cohen-Macaulay. Then it follows from Lemma \ref{l1}(c) that the following holds:
\begin{equation}\tag{\ref{t2}.2}
\hh^{n-1}_{\mz}(M\otimes_RN)\cong\Ext^{n}_{R}(\Tr M,R)\otimes_RN.
\end{equation}
As $M_\fp$ is totally reflexive over $R_\fp$ for all $\fp\in\Supp_R(M)\setminus\mz$ and $\gr_R(\mz,M)\geq n$, Proposition \ref{Serre}(iii) shows that $M$ is $n$-torsion-free. Thus $\Ext^{n}_{R}(\Tr M,R)=0$ and hence (\ref{t2}.2) implies $\hh^{n-1}_{\mz}(M\otimes_RN)=0$; this completes the proof.
\end{proof}


The following example shows that the Tor-rigidity assumption in Theorem \ref{t2} is necessary:

\begin{eg}\label{t3ri} Let $R=k[\![x,y,u,v]\!]/(xy-uv)$, $M=(x,u)$, $N=M^\ast$, and $\mz=\V(\fm)$. It has been established in \cite[1.8]{HW1} that
$M$ is a maximal Cohen-Macaulay $R$-module which is locally free on the punctured spectrum of $R$. Also, one can check that there is an isomorphism $M\otimes_R N\simeq\fm$. 

We set $n=2$. Then all of the hypotheses of Theorem \ref{t2}, except Tor-rigidity, hold: $\Tor^R_2(M,N)=0$ and $\Tor^R_3(M,N)\simeq\Tor^R_1(M,N)\simeq k\neq 0$.
The desired claim is not true because $\hh_{\fm}^1(M\otimes_RN)\simeq\hh_{\fm}^1(\fm)\neq0$.
\end{eg}

Our next aim is to show that Theorem \ref{t2} yields a new proof for Theorem \ref{AuLi}; first we note:

\begin{lem}\label{l}  Let $M, N\in\md(R)$ and let $\X \subseteq \Spec R$ be a subset. Assume $\CI_{R_\fp}(M_\fp)=0$ for all $\fp\in\X.$ Then the following conditions are equivalent:
\begin{enumerate}[\rm(i)]
	\item $\widehat{\Tor}_i^{R_\fp}(M_\fp,N_\fp)=0$ for all $i\in\ZZ$ and for all $\fp\in\X$.
	\item $\Supp_R(\Tor_i^R(M,N))\subseteq\Spec R\setminus\X$ for all $i\gg0$.
	\item $\Supp_R(\Tor_i^R(M,N))\subseteq\Spec R\setminus \X$ for all $i\geq 1$.
\end{enumerate}	
\end{lem}
\begin{proof}
The equivalence of parts (i) and (ii) is due to Theorem \ref{Tate2}. Moreover, the equivalence of parts (ii) and (iii) follows from (\ref{HD}.1) and the dependency formula \ref{depen}.
\end{proof}

Theorem \ref{it2}, advertised in the introduction, is a special case of the following proposition: thanks to Theorem \ref{t2}, it can now be easily established.

\begin{prop}\label{c4} Assume $R$ is local and a complete intersection, $M, N\in\md(R)$ are nonzero, and  $\mz\subset\Spec R $ is a specialization-closed subset. Assume further $N$ is Tor-rigid and, for an integer $n\geq 0$, the following conditions hold:
\begin{enumerate}[\rm(i)]
	\item{$M_\fp$ is maximal Cohen-Macaulay over $R_\fp$ for all $\fp\in\Supp_R(M)\setminus\mz$.}
	\item {$\Supp_R(\Tor_i^{R}(M,N))\subseteq\mz$ for all $i\gg 0$.}
	\item{$\hh_{\mz}^n(M\otimes_RN)=0$.}
	\item{$\gr_R(\mz,M)\geq n$ and $\gr_R(\mz,N)\geq n$.}
\end{enumerate}
Then $\hh_{\mz}^i(M\otimes_RN)=0$ for all $i\leq n$, and $M$ and $N$ are Tor-independent.
\end{prop}

\begin{proof} Note, by (\ref{HD}.1) and assumption (i), we have that $\CI_{R_\fp}(M_\fp)=0$ for all $\fp\in\Supp_R(M)\setminus\mz$. 
Thus, letting $\X=\Supp_R(M)\setminus\mz$, we see that $M_\fp$ is totally reflexive over $R_\fp$ for all $\fp\in\Supp_R(M)\setminus\mz$. Moreover, Lemma \ref{l} yields the vanishing of $\Tor_i^{R_\fp}(M_\fp,N_\fp)$ for all $i\in\ZZ$ and for all $\fp\in\X$. So the assertion is clear by Theorem \ref{t23}.
\end{proof}

We can now make use of Proposition \ref{c4}, which is a vast generalization of Theorem \ref{AuLi}, and prove:

\begin{thm}(Auslander and Lichtenbaum)\label{AL} Assume $R$ is local and regular, and let $M, N\in\md(R)$. If $M\otimes_RN$ is nonzero and torsion-free, then $M$ and $N$ are torsion-free and Tor-independent.
\end{thm}	

\begin{proof}
Recall that finitely generated modules are Tor-rigid over regular local rings \cite{A2, L}.
Set $\mz:=\Spec(R)\setminus\{0\}$. By Proposition \ref{Serre}, $\Gamma_{\mz}(L)=\T(L)$ for every $L\in\md(R)$. Now the assertion follows from Lemma \ref{l1}(a), (b) and Proposition \ref{c4}.
\end{proof}


\section{Applications of the main theorem}


In this section we give several examples, and applications of the results we obtain from Section 3.  The first result, Corollary \ref{h1zerocor}, we aim to prove is an application of Theorem \ref{t2}. 

When $R$ is local, we denote by $\md_0(R)$, the subcategory of finitely generated $R$-modules which are locally free on the punctured spectrum of $R$, namely
$$\md_0(R)=\{M \in \md(R) \mid M_{\fp} \text{ is free over } R_{\fp} \text{ for each } \fp \in \Spec(R) -\{\fm\}\}.$$

\begin{prop}\label{h1zero} Assume $R$ is local such that $\depth(R)\geq 1$, and assume $I$ is an $\fm$-primary ideal of $R$ satisfying the following condition: 
\[
\text{ If } X\in \md(R)  \text{ and } \Tor^R_t(X,R/I)=0 \text{ for some integer } t\geq 0, \text{ then it follows that } \pd_R(X)<t. 
\]
Then $\hh^1_{\fm}(M\otimes_RI)\neq0$ for each $M\in\md_0(R)$ such that $0\neq M$ and $\depth_R(M)\geq 1$.
\end{prop}

\begin{proof} We assume contrary that $\hh^1_{\fm}(M\otimes_RI)=0$, and seek a contradiction.

Letting $N=I$, we obtain from Theorem \ref{t2} that $\Tor_i^R(M,I)=0$ for all $i\geq 1$,  and $\hh^0_{\fm}(M\otimes_RI)=0$, i.e., $\depth_R(M\otimes_RI)\geq 2$.
Now it follows by our assumption on the ideal $I$ that $\pd_R(M)<\infty$. Therefore Theorem \ref{t4} yields:
\begin{equation}\tag{\ref{h1zero}.1}
\pd_R(M)=\depth R-\depth_R(M)=\depth_R(I)-\depth_R(M\otimes_R I).
\end{equation}
As $I$ is $\fm$-primary, we have that $\depth_R(I)=1$. Thus (\ref{h1zero}) shows $\pd_R(M)<0$, or equivalently, $M=0$. This contradiction shows that $\hh^1_{\fm}(M\otimes_RI)\neq0$.
\end{proof}

Next we recall a class of ideals that enjoys the rigidity property stated in Proposition \ref{h1zero}:

\begin{chunk}(\cite[2.10]{CeTos}) \label{CK} Assume $R$ is local, and let $I$ and $J$ be ideals of $R$. Assume $\depth(R)\geq 1$, and the following conditions hold:
\begin{enumerate}[\rm(i)]
\item $(I:_RJ)=(\fm I:_R \fm J)$.
\item $I$ is $\fm$-primary and $0\neq I \subseteq \fm J$.
\end{enumerate}
If $\Tor^R_t(M,R/I)=0$ for some $M\in \md(R)$ and some integer $t\geq 0$, then it follows that $\pd_R(M)<t$.
\end{chunk}

When $R$ is local ring and $\depth(R)\geq 1$, an integrally closed $\fm$-primary ideal $I$ satisfies the conditions (i) and (ii) of \ref{CK} for the case where $J=R$; see \cite[the paragraph following 2.1 and 2.4]{CeTos}, \cite[2.1 and 2.10]{CSTT} and also \cite[3.3]{CHKV}. On the other hand, if $I$ is as in \ref{CK}, then it does not need to be integrally closed, in general; see \cite[2.7 and 2.8]{CSTT} and \cite[4.3]{CeTos}. Although it is difficult to check whether or not a given ideal is integrally closed, in view of \ref{CK}, one can easily construct examples of ideals that have the rigidity property stated in Proposition \ref{h1zero}:

\begin{rmk} \label{goodrmk} Assume $R$ is local and $\depth(R)\geq 1$, and let $A$ be an $\fm$-primary ideal of $R$ such that $A\neq \fm$. Set $I=(A:_R\fm)$. Then $I$ satisfies the conditions (i) and (ii) of \ref{CK} for the case where $J=R$. Therefore, if $\Tor^R_t(M,R/I)=0$ for some $M\in \md(R)$ and some integer $t\geq 0$, then it follows that $\pd_R(M)<t$; see \cite[2.2]{CSTT} and \cite[2.9]{CeTos}.
\end{rmk}

As a consequence of Proposition \ref{h1zero} and Remark \ref{goodrmk}, we conclude:

\begin{cor} \label{h1zerocor}  Assume $R$ is local and $I$ is an $\fm$-primary ideal of $R$. Set $J=(I:_R\fm)$. If $\depth(R)\geq 1$ and $I\neq \fm$, then $\hh^1_{\fm}(J \otimes_R J)\neq0$.
\end{cor}

When the module considered in Proposition \ref{h1zero} has depth at least two, we can say more about $\hh^1_{\fm}(M\otimes_RI)$.

\begin{prop} Assume $R$ is local, $I$ is an $\fm$-primary ideal of $R$, and $M\in\md(R)$ is such that $\depth_R(M)\geq 2$. Then it follows \begin{equation*}
	\hh^i_{\fm}(M\otimes_RI)\simeq\left\{
	\begin{array}{rl}
	\Tor_1^R(M,R/I) & \  \   \   \   \   \ \  \   \   \   \   \ \text{if }\   \   i=0\\
	M/IM& \  \   \   \   \   \ \  \   \   \   \   \ \text{if } \   \ i=1\\
	\hh^i_{\fm}(M)& \  \   \   \   \   \ \  \   \   \   \   \ \text{if }\   \  i>1
	\end{array} \right.
	\end{equation*}
\end{prop}

\begin{proof}
   Applying $-\otimes M$ to the exact sequence $0\to I\to R\to R/I\to 0$, we get the exact sequence $$0\lo \Tor_1^R(M,R/I)\lo I\otimes_R M\lo M\lo R/I\otimes_RM\lo 0.$$
	We break it down into the following exact sequences:
	$$(a):\ 0\to \Tor_1^R(M,R/I)\to I\otimes_R M\to X\to0, \ \
	(b):\ 0\to X\to M\to R/I\otimes_RM\to 0.$$
	As $\Tor_1^R(M,R/I)$ has finite length, $\hh^0_{\fm}(\Tor_1^R(M,R/I))=\Tor_1^R(M,R/I)$ and  $\hh^j_{\fm}(\Tor_1^R(M,R/I))=0$ for all $j>0$. It follows from the exact sequence $(b)$ that $\depth(X)>0$. The exact sequence $(a)$ induces the isomorphisms $\hh^0_{\fm}(M\otimes_RI)\simeq\Tor_1^R(M,R/I)$ and
    $\hh^j_{\fm}(I\otimes_R M)\simeq \hh^j_{\fm}( X)$ for all $j>0$. Applying the functor $\Gamma_{\fm}(-)$ to the exact sequence $(b)$ and noting that $\depth_R(M)>1$ we have $$0=\hh^0_{\fm}(M)\lo \hh^0_{\fm}(R/I\otimes_RM)\lo \hh^1_{\fm}(X)\lo \hh^1_{\fm}(M)=0,$$
    and $\hh^i_{\fm}(X)\cong\hh^i_{\fm}(M)$ for all $i>1$.
	Therefore, we obtain the following isomorphisms:
	$$\hh^1_{\fm}( I\otimes_R M)\simeq \hh^1_{\fm}( X)\simeq\hh^0_{\fm}(R/I\otimes_RM)=R/I\otimes_RM\simeq M/IM.$$Also, $\hh^i_{\fm}(I\otimes_RM)\simeq \hh^i_{\fm}( M)$ for all $i>1$.  
\end{proof}

Next our aim is to make use of Theorem \ref{t2} and prove the following:

\begin{thm}\label{c1c} If $R$ is regular of dimension $d$, then, for nonzero $M, N\in\md_0(R)$, the following hold:

	\begin{enumerate}[\rm(i)]
		\item $\depth_R(M\otimes_RN)\leq\min\{\depth_R(M),\depth_R(N)\}$. Furthermore, it follows that
		$\hh^{i}_{\fm}(M\otimes_RN)\neq0$ for each $i$, where $\depth_R(M\otimes_RN)\leq i\leq\min\{\depth_{R}(M), \depth_R(N)\}$.	
		
\item{Assume $\depth_R(M)+\depth_R(N)= d+n$ for some $n\geq 0$. Then, }
\begin{enumerate}[\rm(a)]
\item $\hh^n_{\fm}(M\otimes_RN)\neq0$ so that $\depth_R(M\otimes_RN)\leq n$. 
\item If $\hh^i_{\fm}(M\otimes_RN)=0$ for some $i$, where $0\leq i<n$, then $\depth_R(M\otimes_RN)=n$.
\end{enumerate}

		\item If $\depth_R(M)+\depth_R(N)\leq d$, then $\hh^i_{\fm}(M\otimes_RN)\neq0$ for each $i=0, \ldots, \min\{\depth_R(M),\depth_R(N)\}$. 

		\item If $\depth_R(M)\leq\frac{1}{2} d$, then $\hh^i_{\fm}(M\otimes_RM)\neq0$ for each $i=0, \ldots, \depth_R(M)$.
	\end{enumerate}
\end{thm}

The proof of Theorem \ref{c1c} requires some preparation; we start with:

\begin{prop}\label{bcid} Assume $R$ is a local complete intersection and assume $\mz\subset\Spec R $ is a specialization-closed subset. Assume further that $M, N\in\md(R)$ are nonzero such that $\NF(M)\subseteq\mz\subseteq\Supp_R(M)$ and $\gr_R(\mz,N)\leq\gr_R(\mz,M)$. If $N$ is Tor-rigid, then $\gr_R(\mz,M\otimes_RN)\leq\gr_R(\mz,N)$. Moreover, $\hh^{i}_{\mz}(M\otimes_RN)\neq0$ for all $i$, where $\gr_R(\mz,M\otimes_RN)\leq i\leq\gr_R(\mz,N)$.
\end{prop}

\begin{proof}
Set $n=\gr_R(\mz,N)$. Assume that $\hh^{j}_{\mz}(M\otimes_RN)=0$ for some $0\leq j\leq n$. By Corollary \ref{c4},
\begin{equation}\tag{\ref{bcid}.1}
\gr_R(\mz, M\otimes_RN)>j \text{ and also } \Tor_k^R(M,N)=0 \text{ for all } k\geq 1. 
\end{equation}
By Proposition \ref{grz}, $\depth_{R_{\fp}}(N_\fp)=n$ for some $\fp\in\mz\subseteq\Supp_R(M)$.
As $\gr_R(\mz, M\otimes_RN)>j$, again by Proposition \ref{grz}, we conclude that
$\depth_{R_\fp}(M_\fp\otimes_{R_\fp}N_\fp)>j$. Therefore,
in view of Theorem \ref{t4}, (\ref{HD}.1) and (\ref{bcid}.1) we have
$$0\leq\CI_{R_\fp}(M_\fp)=\depth R_\fp-\depth_{R_\fp}(M_\fp)=\depth_{R_\fp}(N_\fp)-\depth_{R_\fp}(M_\fp\otimes_{R_\fp}N_\fp)<n-j.$$
It follows from the above inequality that $j<n$. In other words, $\hh^{n}_{\mz}(M\otimes_RN)\neq0$. In particular, $\gr_R(\mz,M\otimes_RN)\leq\gr_R(\mz,N)$. The second assertion follows from (\ref{bcid}.1).	
\end{proof}


In passing we record:

\begin{cor}\label{greg} Assume $R$ is local and regular, and $\mz\subset\Spec R $ is a specialization-closed subset. Assume further that $M, N\in\md(R)$ are nonzero such that $\NF(M)\cup\NF(N)\subseteq\mz\subseteq\Supp_R(M\otimes_RN)$. Then  $\gr_R(\mz,M\otimes_RN)\leq\min\{\gr_R(\mz,M), \gr_R(\mz,N)\}$. Moreover, $\hh^{i}_{\mz}(M\otimes_RN)\neq0$ for all $i$, where $\gr_R(\mz,M\otimes_RN)\leq i\leq\min\{\gr_R(\mz,M), \gr_R(\mz,N)\}$.
\end{cor}

\begin{proof}
Recall that over regular local rings each finitely generated module is Tor-rigid \cite{A2, L}. Hence the assertion is clear by Proposition \ref{bcid}.
\end{proof}


The non-vanishing results we prove in the next theorem are, to the best of our knowledge, new even over regular local rings.

\begin{thm}\label{3.16} Assume $R$ is a local complete intersection of dimension $d$, $M\in\md_0(R)$ and $N\in\md(R)$ is Tor-rigid. Then, for a nonnegative integer $n$, the following hold:
\begin{enumerate}[\rm(i)]
\item{If $\depth_R(M)+\depth_R(N)= d+n$, then $\hh^n_{\fm}(M\otimes_RN)\neq0$. In particular, $\depth_R(M\otimes_RN)\leq n$. In addition, if $\hh^i_{\fm}(M\otimes_RN)=0$ for some $0\leq i<n$, then $\depth_R(M\otimes_RN)=n$.}
\item{If $\depth_R(M)+\depth_R(N)\leq d$, then $\hh^i_{\fm}(M\otimes_RN)\neq0$ for all $0\leq i\leq\min\{\depth_R(M),\depth_R(N)\}$.}
\item {$\depth_{R}(M\otimes_RN)=\max\{0,\depth_R(M)+\depth_R(N)-d\}. $}
\end{enumerate}
\end{thm}

\begin{proof}
(i). As $\depth_R(M)+\depth_R(N)= d+n$, we conclude that $\depth_R(N)\geq n$. Assume contrarily that $\hh^n_{\fm}(M\otimes_RN)=0$. It follows from Theorem \ref{t2} that  $\Tor_i^R(M,N)=0$ for all $i>0$ and $\depth_R(M\otimes_RN)>n$. By Theorem \ref{t4}, $\depth_R(M)+\depth_R(N)=d+\depth_R(M\otimes_RN)>d+n,$ which is a contradiction. The second assertion can be proved similarly.

(ii). Set $n:=\min\{\depth_R(M),\depth_R(N)\}$. Assume contrarily that $\hh^i_{\fm}(M\otimes_RN)=0$ for some $0\leq i\leq n$. By  Theorem \ref{t2}
$\Tor_i^R(M,N)=0$ for all $i>0$ and $\depth_R(M\otimes_RN)>i$. Due to Theorem \ref{t4} $\depth_R(M)+\depth_R(N)=d+\depth_R(M\otimes_RN)>d$ which is a contradiction.

(iii). This is an immediate consequence of parts (i) and (ii).
\end{proof}

We are now ready to give a proof of Proposition \ref{c1c}:

\begin{proof}[Proof of Proposition \ref{c1c}] Recall that every finitely generated  module over a regular local ring is Tor-rigid \cite{A2, L}. Therefore part (i) follows from Proposition \ref{bcid} by setting $\mz=\V(\fm)$. Moreover, parts (ii) and (iii) are  immediate consequence of Theorem \ref{3.16}, and part (iv) is a special case of part (iii).
\end{proof}

We finish this section by giving two examples. Example \ref{e1section4} corroborates Theorem \ref{3.16} and show that our results are sharp. On the other hand, Example \ref{e} provides a detailed computation of the local cohomology of $\Omega^2k \otimes_R \Omega^2k$ over a regular local ring of dimension four; cf. Theorem \ref{c1c}.

In the following, we adopt the  notations of Theorem \ref{3.16} and collect some examples of specialization-closed subsets of $\Spec R$:

\begin{eg} \label{e1section4} $\phantom{}$ 
\begin{enumerate}[\rm(i)] 

\item Assume $R$ is a three-dimensional complete intersection ring and $x\in\fm$ is a non zero-divisor on $R$. Let $N=M=R/xR$. Then $\pd_R(M)=1$ so that $M$ is Tor-rigid.
We consider Theorem \ref{3.16}(i) for the case $n=1$; $\depth_R(M)+\depth_R(N)= d+n = 4$. However, $\hh^1_{\fm}(M\otimes_RN)=0$ since $M\otimes_RM\simeq M$ and $\depth_R(M)=2$. Note that 
$M\notin\md_0(R)$. Hence the example shows that the locally free assumption is necessary for Theorem \ref{3.16}(i).

\item Let $R=k[\![x,y,u,v]\!]/(xy-uv)$, $M=(x,u)$ and let $N=M^\ast$. Recall from Example \ref{t3ri} that
$M$ is locally free and maximal Cohen--Macaulay, and that $M\otimes_R N \cong \fm$.  We consider Theorem \ref{3.16}(i) for the cases where $n=3$ and $i=2$. Then it follows that $\depth_R(M)+\depth_R(N)= d+n=6$. Note that, although $\hh^{2}_{\fm}(M\otimes_R N)\simeq\hh^{2}_{\fm}(\fm)=0$, we have that $\depth_R(M\otimes_RN)=1 \neq n=3$.  As $N$ is not Tor-rigid, this example shows that the Tor-rigidity assumption is crucial for Theorem \ref{3.16}(i).

\item Assume $R$ is a five-dimensional regular local ring, and let $M=\Syz^4k$ and $N=\Syz^2k$. We consider Theorem \ref{3.16}(ii). Note that $M$ and $N$ are locally free on the punctured spectrum of $R$, $N$ is Tor-rigid, and $\depth_R(M)+\depth_R(N)=4+2>5=\dim R$. Moreover,  in view of \cite[5.5(ii)]{finitsup}, we have that $\hh_{\fm}^0(M\otimes_RN)=0$. This example shows that the depth condition in Theorem \ref{3.16}(ii) is necessary.
\end{enumerate}
\end{eg}

In the following, $\beta_n(M)$ denotes the $n$-th \emph{Betti number} of a given module $M$.

\begin{eg} \label{e} Assume $R$ is a four-dimensional regular local ring and let $M=\Omega^2k$. Then $M$ is locally free on the punctured spectrum of $R$, has depth two and
\begin{equation*}
\ell(\hh^i_{\fm}(M\otimes _RM))=\left\{
\begin{array}{rl}
\beta_4(k) & \  \   \   \   \   \ \  \   \   \   \   \ \text{if }\   \   i=0\\
\beta_3(k)& \  \   \   \   \   \ \  \   \   \   \   \ \text{if } \   \ i=1\\
2\beta_1(k)& \  \   \   \   \   \ \  \   \   \   \   \ \text{if }\   \  i=2\\
\beta_0(k)& \  \   \   \   \   \ \  \   \   \   \   \ \text{if }\   \  i=3
\end{array} \right.
\end{equation*} Therefore, we have that $\hh^i_{\fm}(M\otimes_RM)\neq0$ for all $i\leq 4$.
\end{eg}

\begin{proof} It follows from the Theorem \ref{ld1}(i) that $\hh^4_{\fm}(M\otimes_RM) \neq0$. Therefore, we proceed to show the other claims of the example.

Consider the following exact sequences:
$$
(a):\ 0\to \Omega^2k\to R^{\beta_1}\to \Omega^1k\to0, \ \ \ (b):\
0\to\Omega^1k\to R\to k\to 0.$$
The exact sequence (a) induces the following exact sequence:
 $$0\to \Tor_1^R(\Omega^1k,\Omega^2k)\to\Omega^2k\otimes_R\Omega^2k\to R^{\beta_1}\otimes_R\Omega^2k\to \Omega^1k\otimes_R\Omega^2k\to0.$$ We break it down into
$$ (a.1):\ 0\to \Tor_1^R(\Omega^1k,\Omega^2k)\to\Omega^2k\otimes_R\Omega^2k\to X\to0,$$
$$ (a.2):\ 0\to X\to R^{\beta_1}\otimes_R\Omega^2k\to \Omega^1k\otimes_R\Omega^2k\to0.$$
It follows from $(a.2)$ that $\depth(X)>0$. Hence the exact sequence $(a.1)$
induces the isomorphisms
\begin{equation}\tag{\ref{e}.1}
\hh^0_{\fm}(\Syz^2k\otimes_R\Syz^2k)\simeq\hh^0_{\fm}(\Tor_1^R(\Syz^1k,\Syz^2k))\simeq\Tor_1^R(\Syz^1k,\Syz^2k)\simeq\Tor_4^R(k,k).
\end{equation}
Therefore, $\ell(\hh^0_{\fm}(M\otimes_RM))=\beta_4(k)$.
Since $\Tor_1^R(\Syz^1k,\Syz^2k)$ has finite length, $\hh^i_{\fm}(\Tor_1^R(\Syz^1k,\Syz^2k))=0$ for all $i>0$. Thus, the exact sequence $(a.1)$ induces the isomorphism
$\hh^j_{\fm}(\Syz^2k\otimes_R\Syz^2k)\simeq\hh^j_{\fm}(X)$ for all $j>0$.
Note that $\hh^0_{\fm}(\Syz^2k)=\hh^1_{\fm}(\Syz^2k)=0$. Therefore, the exact sequence $(a.2)$ induces the exact sequence: $$0=\hh^0_{\fm}(R^{\beta_1}\otimes_R\Syz^2k)\to\hh^0_{\fm}(\Syz^2k\otimes_R \Syz^1k)\to\hh^1_{\fm}(X)\to\hh^1_{\fm}(R^{\beta_1}\otimes_R\Syz^2k)=0.$$
Hence, we get the following isomorphisms
\begin{equation}\tag{\ref{e}.2}
\hh^1_{\fm}(\Syz^2k\otimes_R\Syz^2k)\simeq\hh^1_{\fm}(X)\simeq\hh^0_{\fm}(\Syz^2k\otimes_R \Syz^1k).
\end{equation}
The exact sequence $(b)$ induces the following exact sequence:
$$0\lo \Tor_1^R(k,\Syz^2k)\to\Syz^1k\otimes_R\Syz^2k \lo \Syz^2k\to k\otimes_R\Syz^2k\lo0.$$
We break it down into the following short exact sequences:
$$(b.1):\ 0\to \Tor_1^R(k,\Syz^2k)\to\Syz^1k\otimes_R\Syz^2k\to Y\to0,$$
$$(b.2):\ 0\to Y\to \Syz^2k\to k\otimes_R\Syz^2k\to0.$$
It follows from $(b.2)$ that $\depth(Y)>0$. Hence, the exact sequence $(b.1)$ induces the isomorphism
\begin{equation}\tag{\ref{e}.3}
\hh^0_{\fm}(\Syz^2k\otimes \Syz^1k)\simeq\hh^0_{\fm}(\Tor_1^R(k,\Syz^2k))\simeq\Tor_1^R(k,\Syz^2k)\simeq\Tor_3^R(k,k).
\end{equation}
By (\ref{e}.2) and (\ref{e}.3), we have
\begin{equation}\tag{\ref{e}.4}
\hh^1_{\fm}(\Syz^2k\otimes_R\Syz^2k)\simeq\hh^0_{\fm}(\Syz^2k\otimes_R \Syz^1k)\simeq\Tor_3^R(k,k).
\end{equation}
Therefore, $\ell(\hh^1_{\fm}(M\otimes_RM))=\beta_3(k)$.
Recall that  $\hh^1_{\fm}(\Syz^2k) =\hh^3_{\fm}(\Syz^2k)=0$ and that $\hh^2_{\fm}(\Syz^2k)=k$ (see \cite[3.3(2)]{goto}).
Therefore, applying the functor $\Gamma_{\fm}(-)$ to the exact sequence $(a.2)$, we get the long exact sequence:
\begin{equation}\tag{\ref{e}.5}
0\to\hh^1_{\fm}(\Syz^2k\otimes_R\Syz^1k)\to\hh^2_{\fm}(X)\to\hh^2_{\fm}(R^{\beta_1}\otimes_R\Syz^2k)\to\hh^2_{\fm}(\Syz^2k\otimes_R\Syz^1k)\to\hh^3_{\fm}(X)\to0.
\end{equation}
Also, the exact sequences $(b.1)$ and $(b.2)$ induce the following isomorphisms:
\begin{equation}\tag{\ref{e}.6}
\hh^i_{\fm}(\Syz^2k\otimes_R\Syz^1k)\simeq\hh^i_{\fm}(Y)\ \text{and} \ \
\hh^{j}_{\fm}(Y)\simeq\hh^{j}_{\fm}(\Syz^2k)\ \text{ for all } i>0\text{ and }\ j>1,
\end{equation}
\begin{equation}\tag{\ref{e}.7}
k^{\oplus\beta_3}\simeq k\otimes_R\Syz^2k\simeq\hh^0_{\fm}(k\otimes_R\Syz^2k)\simeq\hh^1_{\fm}(Y)\simeq\hh^1_{\fm}(\Syz^2k\otimes_R\Syz^1k).
\end{equation}
Therefore, by using (\ref{e}.6) and (\ref{e}.7) one can rewrite the exact sequence (\ref{e}.5) as follows:
\begin{equation}\tag{\ref{e}.8}
0\to k^{\oplus\beta_3}\to\hh^2_{\fm}(X)\to k^{\oplus\beta_1}\stackrel{f}\to k\stackrel{g}\to\hh^3_{\fm}(X)\to0.
\end{equation}
{\bf Claim.} $\hh^3_{\fm}(M\otimes M)\cong\hh^3_{\fm}(X)\neq0$.

To establish the claim, note, by \cite[4.2]{HW}, we have the following isomorphisms:
\begin{equation}\tag{\ref{e}.9}
\hh^3_{\fm}(X)^v\simeq\hh^3_{\fm}(M\otimes M)^v\simeq \hh^2_{\fm}(M^\ast\otimes M^\ast)\simeq\hh^2_{\fm}(\Syz^3k\otimes \Syz^3k),
\end{equation}
where $(-)^v$ is the Matlis dual.
Recall that $\depth(\Syz^3k)+\depth(\Syz^3k)=6\lvertneqq 4+2+1$. In view of \cite[2.4]{HW1} we see that $\hh^2_{\fm}(\Syz^3k\otimes \Syz^3k)\neq 0$. Now the assertion follows from (\ref{e}.9) and the proof of the claim is complete.

Since $\hh^3_{\fm}(X)$ is non-zero and a homomorphic image of $k$,
we conclude that $\hh^3_{\fm}(X)\simeq k$. In particular, $\ell(\hh^3_{\fm}(M\otimes M))=\ell(\hh^3_{\fm}(X))=\beta_{0}(k)$.
It follows from the exact sequence (\ref{e}.8) that $g$ is an isomorphism and $f=0$
so that $\ell(\hh^2_{\fm}(M\otimes M))=\ell(\hh^2_{\fm}(X))=\beta_3(k)+\beta_1(k)= 2 \beta_1(k)$, as required.
\end{proof}


\section{On the depth of tensor powers of modules}

Auslander studied the torsion-freeness of tensor powers of modules and obtained a criteria for freeness. More precisely, over an unramified regular local  ring $R$ of dimension $d$, Auslander \cite{A2} proved that an $R$-module $M$ is free provided that the $d$-fold tensor product $M^{\otimes d}$ of $M$  is torsion-free. The aim of this section is to generalize Auslander's result over more general rings, and find some new criteria for freeness of modules in terms of the local cohomology; see Corollary \ref{t10c}.

We call $M\in \Md(R)$ a \emph{self-test} module provided that the following condition holds:
$$\Tor_i^R(M,M)=0 \text{ for all } i\gg0 \Longleftrightarrow\pd_R(M)<\infty.$$

It is an open question whether or not each module is self-test over local rings. However, it is known that, if $R$ is a complete intersection or Golod, then each module in $\md(R)$ is self-test; see \cite[Theorem IV]{AvBu} and  \cite[3.6]{J2}. We refer the reader to \cite{CH} for further details about the self-test modules.

The next theorem is the main result of this section; to the best of our knowledge, it is new, even if the ring in question is regular.

\begin{thm}\label{gt10} Assume $R$ is local, $\mz\subset\Spec R $ is a specialization-closed subset, and $M\in\md(R)$ is nonzero and Tor-rigid such that  $\NF(M)\subseteq\mz$. 
\begin{enumerate}[\rm(i)]
\item If $\gr_R(\mz,M)<\gr_R(\mz,R)$, then $\hh^{\gr_R(\mz,M)}_{\mz}(\otimes_R^nM)\neq0$ for all $n\geq 1$.
\item Assume, for some $n\geq2$, and for some $i$, where $i<\min\{\gr_R(\mz,M), \gr_R(\mz,R)\}$, that we have $\hh^i_{\mz}(\otimes_R^nM)=0$. Then:
\begin{enumerate}[\rm(a)] 
\item $\Tor_k^R(M,M)=0$ for all $k\geq 1$, and $\hh^j_{\mz}(\otimes_R^mM)=0$ for all $m=1, \ldots, n$ and for all $j=0, \ldots, i$. 
\item If $M$ is self-test, then $\pd_R(M)<(\depth R-i)/n$ and also $\dim_R(\NF(M))<\dim R-n-i$.
\end{enumerate}
\end{enumerate}
\end{thm}

\begin{proof}
We prove the statements simultaneously.	For each $j\geq 1$, set $M_j:=\overset{j}{\otimes}M$. Assume $\hh^i_{\mz}(\overset{n}{\otimes}M)=0$ for some $i\leq\gr_R(\mz,M)$ and $n\geq2$. If $i<\gr_R(\mz,R)$, then by Lemma \ref{l1}(a), we get $\hh^i_{\mz}(M_{n-1})=0$. By repeating this argument, inductively we get the following:
\begin{equation}\tag{\ref{gt10}.1}
\hh^i_{\mz}(M_k)=0 \text{ for all } 1\leq k\leq n.
\end{equation}
In particular, $\hh^i_{\mz}(M)=0$. Hence, $i<\gr_R(\mz,M)$. This prove part (i).
Next we claim the following:\\
{\bf Claim:} $\gr_R(\mz,M_{j+1})>i$ and $\Tor_k^R(M_j,M)=0$ for all $1\leq j\leq n-1$ and $k>0$.\\
Proof of Claim. We argue by induction on $j$. By (\ref{gt10}.1), $\hh^i_{\mz}(M\otimes_RM)=0$. It follows from Theorem \ref{t2} that $\Tor_k^R(M,M)=0$ for $k>0$ and $\gr_R(\mz,M_2)>i$. Hence the case $j=1$ follows. Now assume $j>1$. By induction hypothesis, $\gr_R(\mz,M_{j})>i$ and $\Tor_k^R(M_{j-1},M)=0$ for all $k>0$. By (\ref{gt10}.1), $\hh^i_{\mz}(M_{j}\otimes_RM)=0$. It follows from Theorem \ref{t2} that $\Tor_k^R(M_j,M)=0$ for all $k>0$ and $\gr_R(\mz,M_{j+1})>i$. Thus, the proof of part (a), as well as that of the claim is completed. 

Now assume $M$ is self-test. By the Claim $\Tor_k^R(M_j,M)=0$ for all $1\leq j\leq n-1$ and $k>0$. Hence, $\pd_R(M_j)<\infty$ for all $1\leq j\leq n$ and so by \cite[1.3]{A1} we get the following equality:
\begin{equation}\tag{\ref{gt10}.2}
\pd_R(M_{j+1})=\pd_R(M)+\pd_R(M_{j}) \text{ for all } 1\leq j\leq n-1.
\end{equation}
It follows from Auslander--Buchsbaum formula and (\ref{gt10}.2) that 
$$\depth R-\depth_{R}(M_n)=\pd_R(M_n)=n.\pd_R(M).$$ By the claim, $\gr_R(\mz,M_n)>i$. In particular, by Proposition \ref{grz}, $\depth_R(M_n)>i$. Therefore, $\pd_R(M)<(\depth R-i)/n$. 

Let $\dim_R(\NF(M))=t=\dim(R/\fp)$ for some $\fp\in\NF(M)$. By the Claim  $\Tor_{k}^{R}(M_j,M)_\fp=0$ for  all $k>0$ and $1\leq j\leq n-1$. Thus, we get the equality $\pd_{R_\fp}((M_{j+1})_\fp)=\pd_{R_\fp}((M_{j})_\fp)+\pd_{R_\fp}(M_\fp) \text{ for all } 1\leq j\leq n-1.$ Therefore, we obtain the following equality:
\begin{equation}\tag{\ref{gt10}.3}
n.\pd_{R_{\fp}} (M_{\fp})=\pd_{R_{\fp}}((M_n)_{\fp}).
\end{equation}
Assume contrary that $\dim(R/\fp)=\dim_R(\NF(M))\geq \dim R-n-i$. Hence, we obtain the inequality:
\begin{equation}\tag{\ref{gt10}.4}
\depth R_{\fp}\leq\dim R_{\fp}\leq\dim R-\dim(R/\fp)\leq n+i.
\end{equation}
By the Claim and Proposition \ref{grz}, $\depth_{R_\fp}((M_n)_{\fp})>i$ and so by (\ref{gt10}.4) we have
\begin{equation}\tag{\ref{gt10}.5}
\pd_{R_{\fp}}((M_n)_{\fp})=\depth R_\fp-\depth_{R_\fp}((M_n)_{\fp})<n.
\end{equation}
It follows from (\ref{gt10}.3) and (\ref{gt10}.5) that $\fp\notin\NF(M)$ which is a contradiction.
\end{proof}

Next is a consequence of Theorem \ref{gt10}; part (iii) of Corollary \ref{t10} is a generalization of \cite[3.2]{A2}.

\begin{cor}\label{t10} Assume $R$ is a local complete intersection of dimension $d$ and $M\in\md(R)$ is  Tor-rigid (e.g., $R$ is regular). Then, for $n\geq 2$, the following hold:
\begin{enumerate}[\rm(i)]
\item If $\fa$ is an ideal of $R$, $\NF(M)\subseteq\V(\fa)$ and $\hh^i_{\fa}(\otimes_R^nM)=0$ for some $0\leq i<\gr_R(\fa,M)$, then $\pd_R(M)<(d-i)/n$. Therefore, if $d\leq i+n$, then it follows that $M$ is free.
\item If $\NF(M)\subseteq\V(\fm)$ and $\hh^{i}_{\fm}(\otimes_R^nM)=0$ for some $0\leq i<\depth_R(M)$ and $n\geq d-i$, then $M$ is free.
\item If $M_{\fp}$ is free for each $\fp \in \Ass(R)$ and $\otimes_R^nM$ is torsion-free for some $n\geq d$, then $M$ is free.
\end{enumerate}
\end{cor}

\begin{proof}
First we note, by \cite[Theorem IV]{AvBu}, that $M$ is a self-test module. 

(i) The assertion follows from Theorem \ref{gt10} by setting $\mz=\V(\fa)$.

(ii) This is a special case of part (i).
 
(iii) We set $\mz=\{\fp\in\Spec R\mid\Ht(\fp)>0\}$. Then, by Proposition \ref{Serre}(i) and our assumption, it follows that $\Gamma_{\mz}(\otimes_R^nM)=\T(\otimes_R^nM)=0$. Note by Proposition \ref{grz} that $\gr_R(\mz,R)>0$. Hence, by Theorem \ref{gt10}(i), we may assume that $\gr_R(\mz,M)>0$. Now the assertion follows from Theorem \ref{gt10}(ii).
\end{proof}


We give several examples and show that  the conclusions of Corollary \ref{t10} are sharp:

\begin{eg}\label{notvb1} $\phantom{}$  
\begin{enumerate}[\rm(i)]
		
\item  Assume $R$ is a three-dimensional regular local ring and let $M= \Syz^2_Rk$. Then $M$ has depth two, is not free but is locally free on the punctured spectrum of $R$, and is Tor-rigid. In view of \cite[5.3]{finitsup} one has that $ \hh^{0}_{\fm}(M\otimes_RM) =0$. This example shows that the  bound on $n$ in parts (ii) and (iii) of Corollary \ref{t10} is necessary and sharp.

\item Let $R=\mathbb{C}[\![x,y,z]\!]$ and let $M=R/xR$. Then it follows that $M$ is not free, $M\otimes_R M \cong M$ and $\depth_R(M)=2$. Moreover, we have that $\hh^{\leq 1}_{\fm}(\otimes_R^nM)= \hh^{\leq 1}_{\fm}(M\otimes_RM) \cong \hh^{\leq 1}_{\fm}(M)=0$. This shows that, letting $i=1$ and $n=2$, the assumption $\NF(M)\subseteq\V(\fm)$ in Corollary \ref{t10}(ii) cannot be removed.

\item Let $R=k[\![x,y]\!]/(xy)$ and let $M=R/xR$. Then $R$ is reduced, $\dim(R)=1=\depth_R(M)$, $M$ is locally free on the punctured spectrum of $R$, and $\hh^0_{\fm}(\otimes_R^nM)=0$ for all $n\geq 2$. However, $M$ is not Tor-rigid: $\Tor_1^R(M, N)=0\neq \Tor_2^R(M,N)$, where $N=R/yR$. This example shows that the Tor-rigidity assumption in Corollary \ref{t10}(iii)  is needed.
\end{enumerate} 
\end{eg}

We proceed with another consequence of Theorem \ref{gt10}. 

\begin{cor}\label{t10c} Assume $R$ is a local ring of positive depth. Set $d=\dim(R)$. Then, for an integer $n$ with $n\geq\max\{2,d\}$, the following hold:
\begin{enumerate}[\rm(i)]
\item $\depth_R(\otimes_R^n(\Omega^ik))=0$ for all $i=0, \ldots, d-1$. 
\item If $R$ is not  regular, then $\depth_R(\otimes_R^n(\Omega^ik))=0$  for all $i\geq 0$.
\item The sequence $f(n)=\depth_R(\otimes_R^n(\Omega^ik))$ is eventually constant for all $i\geq0$.
\end{enumerate}
\end{cor}

\begin{proof} (i). Without loss of generality, we may assume that $0<i<d$. Assume that we have $\depth_R(\otimes_R^nM)\geq 1$, where $M=\Omega^ik$. Then $\hh^0_{\fm}(\otimes_R^nM)=0$. Since $M$ is Tor-rigid, locally free on the punctured spectrum of $R$, and self-test, we conclude from Theorem \ref{gt10}(ii) that $\pd_R(M)<\depth R/n\leq d/n\leq 1$. As $d/n\leq 1$, we see $M$ is free, which is a contradiction. Because the freeness of $M$ implies the regularity of $R$, in this case, $M$ is a module of depth $i$ and hence it cannot be free.

The proof of part (ii) is similar to that of part (i). Also, part (iii) is a combination of parts (i) and (ii).	
	
	
	 
\end{proof}

We finish this section by noting that Corollary \ref{t10c} relaxes the regularity hypothesis in \cite[3.2]{HW1}. Furthermore, it removes the restriction on $\Omega^1k$ in \cite[6.9]{finitsup} and gives an
answer to \cite[Problem 1.6]{finitsup} in the nontrivial case.




\section{Applications in prime characteristic}


The aim of this section is to apply our previous results to the study of the vanishing of local cohomology modules of the Frobenius powers. We prove two main results, namely Theorems \ref{fr} and \ref{tt}. As we proceed, we state these two theorems, establish several corollaries of them, give examples, and defer the proofs of Theorems \ref{fr} and \ref{tt} to the end of this section.

In this section, $R$ denotes a local ring of prime characteristic $p>0$, and $\varphi:R\to R$ denotes the \emph{Frobenius endomorphism} given by $\varphi(a)=a^{p}$ for $a\in R$. Each iteration $\varphi^n$ of $\varphi$ defines a new $R$-module structure on the set $R$, and this $R$-module is denoted by $\up{\varphi^n}R$, where $a\cdot b = a^{p^{n}}b$ for $a, b \in R$. More generally, given $M\in\md(R)$, we denote by $\up{\varphi^n}M$ the finitely generated $R$-module $M$ with the $R$-action given by $r \cdot x =\varphi^n(r)x$ for $r\in R$ and $x \in M$. For the proof of several results, we make use of the following well-known result of Kunz \cite{Ku} without reference: $R$ is regular if and only if $\up{\varphi^r}R$ is a flat $R$-module.

Recall that $R$ is said to be \emph{F-finite} if the Frobenius endomorphism makes $R$ into a module-finite $R$-algebra, i.e., if $\up{\varphi}R$ is a finitely generated $R$-module. It is known that, if $R$ is a local complete intersection, then the Frobenius endomorphism $\up{\varphi^r}R$ (not necessarily a finitely generated $R$-module) is Tor-rigid for all $r\geq 1$; see, for example, \cite[5.1.1]{Mi}.

Next is the statement of our first theorem:

\begin{thm}\label{fr} Assume $R$ is a complete intersection ring, $\mz\subseteq \Spec R$ is a specialization-closed subset, and $M\in\md(R)$ is such that $\NF(M)\subseteq\mz$. Then, for an integer $n\geq 1$, the following hold:
	\begin{enumerate}[\rm(i)]
		\item If $\hh^{i+1}_{\mz}(M\otimes_R\up{\varphi^n}R)=0$ for some $i$ with $i<\gr_R(\mz,R)$, then  $\hh^{i}_{\mz}(M\otimes_R\up{\varphi^n}R)\cong\hh^{i}_{\mz}(M)\otimes_R\up{\varphi^n}R$.
		\item $\hh^{\gr(\mz,M)}_{\mz}(M\otimes_R\up{\varphi^n}R)\neq0$. Therefore, it follows that $\gr_R(\mz,M\otimes_R\up{\varphi^n}R)\leq\gr_R(\mz,M)$.
		\item If $\hh^{i}_{\mz}(M\otimes_R\up{\varphi^n}R)=0$ for some $i$ with $0\leq i<\gr_R(\mz,M)$, then it follows that $\pd_R(M)<\infty$. Therefore, one has that $\gr_R(\mz,M)=\gr_R(\mz,M\otimes_R\up{\varphi^n}R)$.
		\item{Either $\gr_R(\mz,M\otimes_R\up{\varphi^n}R)=0$ or $\gr_R(\mz,M\otimes_R\up{\varphi^n}R)=\gr_R(\mz,M)$.}
	\end{enumerate}
\end{thm}

The first corollary of Theorem \ref{fr} determines a new freeness criteria in terms of the vanishing of local cohomology modules. Recall that $\md_0(R)$ denotes the category of all finitely generated $R$-modules that are locally free on the punctured spectrum of $R$.

\begin{cor}\label{fr1} Assume $R$ is a $d$-dimensional complete intersection, and $M\in\md_0(R)$ is a maximal Cohen-Macaulay $R$-module. If $\hh^{i}_{\fm}(M\otimes_R\up{\varphi^n}R)=0$ for some integer $n\geq 1$ and some $i$ with $0\leq i<d$, then $M$ is free.
\end{cor}

\begin{proof} This claim is an immediate consequence of Theorem \ref{fr}(iii).
\end{proof}

We now give an example and show that the conclusion of Corollary \ref{fr1} is sharp.

\begin{eg} \label{efr1} $\phantom{}$  
\begin{enumerate}[(i)]
\item Let $R=\mathbb{F}_2[[x,y]]/(x^2)$ and let $M=R/xR$. Then $M$ is maximal Cohen--Macaulay. Also, it follows that $M\otimes_R\up{\varphi }R\cong \frac{R}{x^2R}=R$. Therefore, we have that $\hh^{d-1}_{\fm}(M\otimes_R\up{\varphi }R)=\hh^{0}_{\fm}(R)=0$. As $M$ is not free, this example shows that the locally free hypothesis in Corollary \ref{fr1} is necessary.
\item Let $R=k[\![x,y]\!]$ and let $M=k$. Then it follows that $M$ is locally free over punctured spectrum of $R$, and $\hh^1_{\fm}(M\otimes_R\up{\varphi }R)=\hh^1_{\fm}( R/\fm^{[p]})=0$. However, $M$ is not maximal Cohen-Macaulay. This example shows that the conclusion of Corollary \ref{fr1} may not be true if the module $M$ in question is not maximal Cohen-Macaulay.
\end{enumerate}
\end{eg}

If $\X$ is a subset of $\Spec R$ (where $R$ is a Noetherian ring, not necessarily local or of prime characteristic $p$) and $M\in\md(R)$, we say $M$ is locally free on $\X$ provided that $M_\fp$ is a free $R_\fp$-module for all $\fp\in\X$. For an integer $i\geq 0$, we set $$\X^i(R)=\{\fp\in \Spec(R)\mid\Ht(\fp)\leq i\}.$$

Next, in passing, we make use of Theorem \ref{fr} and obtain a new proof of a result of Celikbas, Iyengar, Piepmeyer and Wiegand \cite{CIPW2}.

\begin{cor} (\cite[3.4]{CIPW2}) Assume $R$ is a complete intersection and $M\in \md(R)$ is locally free on $\X^0(R)$. Then, for an integer $n\geq 1$, the following conditions are equivalent:
\begin{enumerate}[\rm(i)]
\item $M\otimes_R\up{\varphi^n}R$ is a torsion-free $R$-module.
\item $M$ is a torsion-free $R$-module such that $\pd_R(M)<\infty$.
\end{enumerate}
\end{cor}

\begin{proof}
Set $\mz:=\{\fp\in\Spec R\mid\Ht(\fp)>0\}$. 
Note by Proposition \ref{Serre}(i) that $\Gamma_{\mz}(L)=\T(L)$ for each $L\in\Md(R)$. 
Therefore, $L$ is torsion-free if and only if $\gr_R(\mz,L)>0$. 

(i)$\Rightarrow$(ii). By our assumption $\gr_R(\mz, M\otimes_R\up{\varphi^n}R)>0$. It follows from parts (iii) and (iv) of Theorem \ref{fr} that $\gr_R(\mz,M)=\gr_R(\mz, M\otimes_R\up{\varphi^n}R)>0$ 
and $\pd_R(M)<\infty$. In particular, $M$ is torsion-free.

(ii)$\Rightarrow$(i). Since $M$ has finite projective dimension, as we have seen in the proof of Theorem \ref{fr}(iii), $\gr_R(\mz, M\otimes_R\up{\varphi^n}R)=\gr_R(\mz, M)$. As $M$ is torsion-free, $\gr_R(\mz, M)>0$. Therefore $ M\otimes_R\up{\varphi^n}R$ has a positive grade with respect to $\mz$. This, in turn,  is equivalent to the  torsion-free property of the module. The proof is now complete.
\end{proof}

Recall that the singular locus $\Sing(R)$ of $R$ is $\Sing(R)=\{\fp\in\Spec R\mid R_\fp \text{ is not regular}\}$. Note that $\Sing(R)$ is a Zariski closed set provided that $R$ is excellent.

\begin{cor} Assume $R$ is an $F$-finite complete intersection and let $\fa$ be a proper ideal of $R$ such that $\Sing(R)\subseteq\V(\fa)$. If $\hh^i_{\fa}(\up{\varphi^r}M\otimes_R\up{\varphi^s}R)=0$ for some integers $r\geq 1$, $s\geq 1$ and $i$ with $0\leq i<\gr_R(\fa,M)$, then $R$ is regular.
\end{cor}

\begin{proof}
	This is an immediate consequence of Theorem \ref{fr}(iii) and \cite[Theorem 1.1]{AHIY}.	
\end{proof}

Here is the second main result of this section; see \ref{SL} for the definition of Serre's condition.

\begin{thm}\label{tt} Assume $R$ is a $d$-dimensional complete intersection, where $d\geq 1$, and $M\in\md(R)$ is locally free on $\X^{d-n-1}(R)$ for some $n$ with $0\leq n \leq d-1$. Assume further that the following hold:
	\begin{enumerate}[\rm(i)]
	    \item $\hh^n_{\fm}(M\otimes_R\up{\varphi^r}R)=0$ for some integer $r\geq 1$.
	    \item $M$ satisfies Serre's condition $(S_n)$.    
	\end{enumerate}
	 Then it follows that $\pd_R(M)<d-n$.
\end{thm}


The following example corroborates Theorem \ref{tt} and show that the result is sharp.

\begin{eg} $\phantom{}$  
\begin{enumerate}[(i)]
\item Let $R=k[\![x,y,z, w]\!]$ and let $M=\Syz^2_Rk$. Then $M$ is locally free on the punctured spectrum of $R$, $\depth_R(M)=2$, $M$ satisfies $(S_2)$, and $\hh^2_{\fm}(M)\neq0$. Note, as $R$ is regular, $\up{\varphi^r}R$ is a flat $R$-module. Hence we conclude that $\hh^2_{\fm}(M\otimes_R\up{\varphi^r}R)\cong \hh^2_{\fm}(M)\otimes_R\up{\varphi^r}R \neq0$. Hence, since $\pd(M)=2$, letting $d=\dim(R)=4$ and $n=2$, we deduce from this example that the conclusion of Theorem \ref{tt} may not hold if the condition $\hh^n_{\fm}(M\otimes_R\up{\varphi^r}R)=0$ is removed.

\item Let $R=k[\![x,y]\!]$ and let $M=R\oplus k$. Letting $n=1$ and $d=\dim(R)=2$, we see that $M$ is locally free over $\X^{d-n-1}(R)$. Then it follows that
$$\hh^1_{\fm}(M\otimes_R\up{\varphi }R)=\hh^1_{\fm}( \up{\varphi }R)\oplus \hh^1_{\fm}( R/\fm^{[p]})=0.$$
Note that $\depth_R(M)=0$ so that $M$ $M$ doest not satisfy $(S_1)$. As $\pd(M)=2=d>d-n$, we deduce from this example that the conclusion of Theorem \ref{tt} may not hold if the module $M$ in question does not satisfy $(S_n)$.

\item Let $R=k[\![x,y]\!]$ and let $M=R$. Then, letting $d=\dim(R)=2$ and $n=1$, we see that $M$ satisfies all the hypotheses of Theorem \ref{tt}. Furthermore, $\pd_R(M)=d-n-1$. Hence we deduce from this example that the bound for the projective dimension of $M$ stated in the theorem is sharp.

\end{enumerate}
\end{eg}

Next we give two corollaries of Theorem \ref{tt}. The first one, Corollary \ref{regcr1}, is an immediate consequence of the theorem. We should note that the hypothesis that $R$ is reduced in Corollary \ref{regcr1} is necessary, even in the complete case; see Example \ref{efr1}.

\begin{cor}\label{regcr1} If $R$ is a $d$-dimensional complete intersection, and $M\in\md(R)$ is a maximal Cohen-Macaulay $R$-module such that $M$ is locally free on $\X^{0}(R)$ (e.g., $R$ is reduced) and $\hh^{d-1}_{\fm}(M\otimes_R\up{\varphi^r}R)=0$ for some integer $r\geq 1$, then $M$ is free.
\end{cor}


Next we prove the second corollary of Theorem \ref{tt}, namely Corollary \ref{Fr}; the corollary determines a new criteria for regularity in terms of the Frobenius endomorphism. Furthermore, Theorem \ref{ft} advertised in the introduction is a special case of Corollary \ref{Fr}. First we recall:

Given an integer $n\geq0$, a ring $R$ (where $R$ is a Noetherian ring, not necessarily local, or of characteristic $p$) satisfies Serre's condition $(R_n)$ provided that $R_\fp$ is a regular local ring for all $\fp\in\X^{n}(R)$. Note that $R$ is reduced  if and only if $R$ satisfies Serre's conditions $(R_0)$ and $(S_1)$, and $R$ is normal  if and only if $R$ satisfies Serre's conditions $(R_1)$ and $(S_2)$.

The following remark is well-known, but we record it here as we use it for the proof of Corollary \ref{Fr}, as well as for the proofs of Theorems \ref{fr} and \ref{tt}, to reduce the argument to the F-finite case.

\begin{rmk}\label{FF} There is a local ring extension $(S,\fn)$ of $(R,\fm)$ such that $\fm S=\fn$, $S$ is $F$-finite, $S$ is faithfully flat over $R$, and $S$ has infinite residue field.
For example, letting $\widehat{R}\cong k[[x_1,\cdots, x_m]]/I$ for some ideal $I$, we can pick $S=\bar{k}[[x_1,\cdots,x_m]]/I\bar{k}[[x_1,\cdots, x_m]]$, where $\bar{k}$ denotes the algebraic closure of $k$.
Note, if $M$ is an $R$-module, then it follows that $(M\otimes_R\up{\varphi^n}R)\otimes_R S\cong (M\otimes_RS)\otimes_S\up{\varphi^n}S$.
\end{rmk}


\begin{cor}\label{Fr} If $R$ is a $d$-dimensional complete intersection such that $R$ satisfies Serre's condition $(R_{n-1})$ and $ \hh^{d-n}_{\fm}(\up{\varphi^r}R\otimes_R\up{\varphi^s}R)=0$ for some integers $r \geq 1$, $s \geq 1$ and $n \geq 1$, then $R$ is regular.
\end{cor}

\begin{proof} It suffices to prove that we may assume $R$ is F-finite: in that case, $\up{\varphi^r}R$ is a (finitely generated) maximal Cohen-Macaulay $R$-module, and hence the assertion follows from Theorem \ref{tt} and \cite[2.1]{Ku}; see also \cite[Theorem 2]{Rod}.

We consider the local ring extension $(S,\fn)$ of $(R,\fm)$ that follows from Remark \ref{FF}. Then, by the flat base change theorem along with the independence theorem for local cohomology modules, we have 
$$\hh^{d-n}_{\fm}(\up{\varphi^r}R\otimes_R\up{\varphi^s}R)\otimes_RS\cong
\hh^{d-n}_{\fm}((\up{\varphi^r}R\otimes_R\up{\varphi^s}R)\otimes_RS)\cong\hh^{d-n}_{\fm}(\up{\varphi^r}S\otimes_S\up{\varphi^s}S)\cong\hh^{d-n}_{\fn}(\up{\varphi^r}S\otimes_S\up{\varphi^s}S). $$

Now let $P\in \X^{n-1}(S)$ and set $\fp=P\cap R$. Then it follows that $\fp \in \X^{n-1}(R)$. Therefore, $R_{\fp}$ is regular and hence $\up{\varphi}(R_{\fp})$ is flat as an $R_{\fp}$-module. 
As the localization commutes  with the Frobenius map, we see:
$$\up{\varphi}(S_{P})\cong(\up{\varphi}S)_{P} \cong (\up{\varphi}R\otimes_RS)_{P}\cong (\up{\varphi}R)_{\fp}\otimes_{R_{\fp}}  S_{P}\cong  \up{\varphi}(R_{\fp})\otimes_{R_{\fp}}  S_{P}.$$
Hence, since $R_{\fp} \to S_{P}$ is flat, we conclude that $\up{\varphi}(S_{P}) $ is flat as an $S_{P}$-module. Furthermore, if $S$ is regular, then so is $R$. Consequently, as the hypotheses and the conclusion do not change by passing to $S$, we may assume $R$ is F-finite, as claimed. This finishes the proof.
\end{proof}

We now proceed and give proofs of Theorems \ref{fr} and \ref{tt}. First we note:

\begin{rmk}\label{rl1} We say $N\in \Md(R)$ satisfies in the \emph{Nakayama's property} if $\Supp_R(M)\cap\Supp_R(N)=\Supp_R(M\otimes_RN)$ for each $M\in\md(R)$. Clearly, every finitely generated module over a local ring (not necessarily of prime characteristic) satisfies in the Nakayama's  property. We should note that Lemma \ref{l1} and Theorem \ref{t2} holds even if $N$ is not finitely generated, but satisfies in the Nakayama's property.
\end{rmk}

\begin{proof}[Proof of Theorem \ref{fr}] First note, by \cite[I.1.5]{PS}, we have that $\Supp_R(M\otimes_R\up{\varphi^n}R)=\Supp_R(M)$. In particular, $\up{\varphi^n}R$ satisfies in the Nakayama's property; see Remark \ref{rl1}. Also, we have $\gr_R(\mz,\up{\varphi^n}R)=\gr_R(\mz,R)$, and that $\up{\varphi^n}R$ is Tor-rigid.\\	
(i). This  follows from Lemma \ref{l1}(c) and Remark \ref{rl1}.\\			
(ii) and (iii). Assume that $\hh^i_{\mz}(M\otimes_R\up{\varphi^n}R)=0$ for some $i$ with $0\leq i\leq\gr_R(\mz,M)$. Then, by Theorem \ref{t2} and Remark \ref{rl1}, we have
$\Tor_j^R(M,\up{\varphi^n}R)=0$ for all  $j\geq 1$. In view of  \cite[Theorem]{AvMi}, $\pd_R(M)<\infty$, as claimed by (iii). Now, we proceed to give a proof for part (ii).

We are going to reduce to the $F$-finite case. To this end, we pass to the local ring $S$ that follows from Remark \ref{FF} and, without losing of generality, we may assume that $R$ is $F$-finite. 
We know, by \cite[1.7]{PS}, the following equality holds for each $\fp \in \Spec(R)$: 
\begin{equation}\tag{\ref{fr}.2}
\pd_{R_\fp}(M_\fp)=\pd_{R_{\fp}}(M_\fp\otimes_{R_\fp}\up{\varphi^n}R_\fp).
\end{equation}
Therefore, it follows from (\ref{fr}.2) and Proposition \ref{grz}  that $\gr_R(\mz,M)=\gr_R(\mz,M\otimes_R\up{\varphi^n}R)$ and so the assertion follows. \\
(iv). This is clear by parts (ii) and (iii).
\end{proof}


To prove Theorem \ref{tt}, we need:

\begin{dfn}(\cite{AB})\label{gra} If $R$ is Noetherian (not necessarily local, or of prime characteristic) and if $M,N \in \Md(R)$, then the \emph{grade} of the pair $(M,N)$ is defined as:
$$\gr_R(M,N)=\inf\{i\mid\Ext^i_R(M,N)\neq 0\}.$$
\end{dfn}

Note that the grade of $(M, N)$ is not necessarily finite, in general. If $R$ is local and $M,N \in \md(R)$, then we have $\gr_R(M,N)=\gr_R(\ann(M),N)$, which is equal to the length of maximal regular sequence on $N$ in $\ann_R(M)$; in this case the grade is finite. Furthermore, by \cite[4.5]{AB}, we have:
   \begin{equation}\tag{\ref{gra}.1}
   \gr_R(M,N)=\inf\{\depth_{R_{\fp}}(N_\fp)\mid\fp\in\Supp_R(M)\}.
   \end{equation}
For simplicity, we denote  the grade of  $(M,R)$ by $\gr_R(M)$.

The following lemma plays a crucial role in the sequel.

\begin{lem}\label{duality}
	Let $M,N, K \in \Md(R)$ and let $\ell \geq 1$ be an integer. Assume $\Ext^j_R(N,K)=0$ for all $j$ with $1\leq j\leq \ell-1$, and $\gr_R(\Tor_i^R(M,N),K)>\ell-i$ for all $i$ with $1\leq i\leq \ell$. Then the following holds:
	\begin{enumerate}[\rm(i)]
		\item $\Ext^i_R(M,\Hom_R(N,K))\cong\Ext^i_R(M\otimes_RN,K)$ for all $i$ with $0\leq i<\ell$.
		\item There is an injection $\Ext^\ell_R(M,\Hom_R(N,K))\hookrightarrow\Ext^\ell_R(M\otimes_RN,K)$.
	\end{enumerate}	
\end{lem}

\begin{proof} There are two spectral sequences converging to the same point:
	\begin{enumerate}[]
		\item$
		E_2^{pq}:=\Ext^p_R(\Tor_q^R(M,N),K) \Longrightarrow H^{p+q}$ and
		$
		F_2^{pq}:=\Ext^p_R(M,\Ext^q_R(N,K)) \Longrightarrow H^{p+q}
		$.
	\end{enumerate}	
	As $\gr_R(\Tor_i^R(M,N),K)>\ell-i$ for all $1\leq i\leq \ell$, we have that $E_2^{pq}=0$ if $1\leq q\leq \ell$ and $p\leq \ell-q$, and obtain an isomorphism $E_2^{i0}\cong H^i$ for all integers $i\le \ell$.
	Since $\Ext^i_R(N,K)=0$ for all $1\le i\le \ell-1$, we have $F_2^{pq}=0$ if $1\le q\le \ell-1$. Hence, we get $F_2^{i0}\cong H^i$ for all integers $i\le \ell-1$ and an injection $F_2^{\ell0}\hookrightarrow H^\ell$. Thus there are isomorphisms $\Ext^i_R(M\otimes_R N,K)=E_2^{i0}\cong H^i\cong F_2^{i0}=\Ext^i_R(M,\Hom_R(N,K))$ for all integers $i\le \ell-1$, and also an injection
	$\Ext^\ell_R(M,\Hom_R(N,K))\cong F_2^{\ell0}\hookrightarrow H^\ell\cong\Ext^\ell_R(M\otimes_RN,K)$.
\end{proof}

We are now ready to prove Theorem \ref{tt}.

\begin{proof}[Proof of Theorem \ref{tt}] We can pass to the local ring $S$ that exists by Remark \ref{FF} and assume $R$ is F-finite, i.e., $\up{\varphi^n}R \in \md(R)$.
Set $\ell=d-n$. Then $\ell\geq 1$. Also, by (ii) and the Theorem \ref{Ld}, we have
	\begin{equation}\tag{\ref{tt}.1}
	\Ext^{\ell}_R(M\otimes_R\up{\varphi^r}R,R)=0.
	\end{equation}
	Let $\fp\in\Supp_R(\Tor_i^R(M,\up{\varphi^r}R))$ for some $i\geq 1$. It follows from (i) that 
	\begin{equation}\tag{\ref{tt}.2}
	\depth(R_{\fp})=\dim(R_{\fp})\geq d-n>d-n-i.
	\end{equation}
    Hence, by (\ref{gra}.1) and (\ref{tt}.2), we deduce that  
    \begin{equation}\tag{\ref{tt}.3}
    \gr_R(\Tor_i^R(M,\up{\varphi^r}R))>d-n-i=\ell-i \text{ for all } i \text{ with } 1\leq i\leq \ell. 
    \end{equation}
    As $R$ is Gorenstein, it follows $\up{\varphi^r}R\cong\Hom_R(\up{\varphi^r}R,R)$ \cite[Theorem 1.1]{gotoo}. Also, since $\up{\varphi^r}R$ is a maximal Cohen-Macaulay $R$-module, we see $\Ext^j_R(\up{\varphi^r}R,R)=0$ for all $j\geq 1$. Hence, in view of Lemma \ref{duality}(ii), (\ref{tt}.1) and (\ref{tt}.3) we conclude that $$\Ext^{\ell}_R(M,\up{\varphi^r}R)=\Ext^{\ell}_R(M,\Hom_R(\up{\varphi^r}R,R))=0.$$
	
Now, by the four-term exact sequence (\ref{a1}.4) we obtain $\Tor_1^R(\Tr\Omega^{\ell}M,\up{\varphi^r}R)=0$. So \cite[Theorem]{AvMi} implies that $\Tor_i^R(\Tr\Omega^{\ell}M,\up{\varphi^r}R)=0$ for all $i\geq 1$ and  $\pd_R(\Tr\Omega^{\ell}M)<\infty$. Another use of (\ref{a1}.4) implies that $\Ext^{\ell}_R(M,R)=0$. On the other hand, as $M$ satisfies $(S_n)$, it is an $n$-th syzygy module; see \cite[4.25]{AB}. Therefore, $\Omega^{\ell}M$ is a $d$-th syzygy module, or equivalently, is a maximal Cohen-Macaulay module. This implies that $\Tr\Omega^{\ell}M$ is also maximal Cohen-Macaulay. As $\Tr \Tr(-)\approx(-)$, we deduce that $\Omega^{\ell}M$ is free and $\pd_R(M)\leq\ell$. Consequently, the fact that $\Ext^{\ell}_R(M,R)=0$ yields $\pd_R(M)<\ell=d-n$, as required.
\end{proof}


\section{A relation between the local cohomology and the Tate homology}

In this section, we determine a new relation between the local cohomology of tensor products of modules and the Tate homology over Gorenstein rings.  Our main result is Theorem \ref{t12}, which will be a new tool in the study of the depth of tensor products of modules. To the best of our knowledge, Theorem \ref{t12} is new, even if the specialization-closed subset $\mz$ considered is a closed subset of $\Spec(R)$.

Theorem \ref{t12} has various applications that contribute to the literature; we state and prove these applications following the proof of the theorem in this section. Several of the applications we give should be of independent interest. For example, Theorem \ref{t12} improves a result that has been initially proved by Dao \cite[7.7]{Daa}, and subsequently studied by Celikbas \cite[3.4]{Olgur} and Celikbas, Iyengar, Piepmeyer and Wiegand \cite[3.14]{CIPW}; see Corollary \ref{c2}. 

Recall that $R$ denotes a commutative Noetherian ring throughout.


\begin{thm}\label{t12} Let  $\mz\subset\Spec R $ be a specialization-closed subset and let $M, N\in\md(R)$. Assume $\G-dim_R(M)<\infty$. Assume further, for an integer $n\geq 0$, the following hold:
\begin{enumerate}[\rm(i)]
\item  $\depth_{R_\fp}(M_\fp)+\depth_{R_\fp}(N_\fp)\geq\depth R_{\fp}+n$ for each $\fp\in\mz$.
 \item $\Supp_R\big(\Tor_i^R(M,N)\big)\bigcap \Supp_R\big(\widehat{\Tor}_j^R(M,N)\big) \subseteq\mz$ for all $i\geq 1$ and for all $j\in\ZZ$ (e.g., $\NF(M)\subseteq\mz$).  
\end{enumerate}
Then it follows:
\begin{enumerate}[\rm(a)]
	\item $\hh^i_{\mz}(M\otimes_RN)\cong\widehat{\Tor}_{-i}^R(M,N)$ for each $i=0, \ldots, n-1$.
	\item There is an injection $\widehat{\Tor}_{-n}^R(M,N)\hookrightarrow\hh^n_{\mz}(M\otimes_RN)$.
\end{enumerate}	
\end{thm}

\begin{proof}
	Set $\mathcal{W}:=\Supp_R(M)$ and $\mz':=\mz\cap\mathcal{W}$. Clearly, $\mz'$ is specialization-closed. First we prove:
	
	\textbf{Claim I.} $\hh^i_{\mz}(M\otimes_RN)\cong\hh^i_{\mz'}(M\otimes_RN)$ for all $i\geq0$.\\
	Proof of Claim. Set $L=M\otimes_RN$ and let $0\to L\to\E^0(L)\to\E^1(L)\to\cdots$ be the minimal injective resolution of $X$. Note that $\Supp_R(\E^i(L))\subseteq\Supp_R(L)\subseteq\mathcal{W}$ for all $i\geq0$. Therefore, $\Gamma_{\mathcal{W}}(L)=L$ and also $\Gamma_{\mathcal{W}}(\E^i(L))=\E^i(L)$ for all $i\geq0$.
	It follows from Theorem \ref{LC}(iv) that $\Gamma_{\mz'}(L)=\Gamma_{\mz}(\Gamma_{\mathcal{W}}(L))=\Gamma_{\mz}(L)$. Similarly,
	$\Gamma_{\mz'}(\E^i(L))=\Gamma_{\mz}(\Gamma_{\mathcal{W}}(\E^i(L)))=\Gamma_{\mz}(\E^i(L))$ for all $i\geq0$ which implies that $\hh^i_{\mz}(L)\cong\hh^i_{\mz'}(L)$ for all $i\geq0$. Thus, the proof of the claim is completed.
	
	By Claim I, without loss of generality, by replacing $\mz$ with $\mz'$, we may assume that $\mz\subseteq\Supp_R(M)$. Consider the following exact sequence
	\begin{equation}\tag{\ref{t12}.1}
	0\rightarrow M\rightarrow X\rightarrow G\rightarrow0
	\end{equation}
	where $\pd_R(X)<\infty$ and $\G-dim_R(G)=0$ (see \cite[2.17]{cfh}).
	By \cite[2.9]{CJ}, the exact sequence (\ref{t12}.1) induces a
	doubly infinite long exact sequence
	$$\cdots\rightarrow\widehat{\Tor}_{j+1}^R(G,N)\rightarrow\widehat{\Tor}_{j}^R(M,N)\rightarrow\widehat{\Tor}_{j}^R(X,N)\rightarrow
	\widehat{\Tor}_{j}^R(G,N)\rightarrow\cdots,$$
	of stable homology modules. Also by Theorem \ref{Tate hom}(i), $\widehat{\Tor}_{j}^R(X,N)=0$ for all $j\in\ZZ$. Hence we get the following isomorphism
	\begin{equation}\tag{\ref{t12}.2}
	\widehat{\Tor}_{j+1}^R(G,N)\cong\widehat{\Tor}_{j}^R(M,N) \text{ for all } j\in\ZZ.
	\end{equation}
	As $G$ is totally reflexive, by Theorem \ref{Tate hom}(iv), (\ref{t12}.2) and assumption (ii), we have
	\begin{equation}\tag{\ref{t12}.3}
	\Supp_R(\Tor_i^R(G,N))=\Supp_R(\widehat{\Tor}_i^R(G,N))=\Supp_R(\widehat{\Tor}_{i-1}^R(M,N))\subseteq\mz \text{ for all } i>0.
	\end{equation}
	Applying the functor $-\otimes_RN$ to the exact sequence (\ref{t12}.1), we get the following exact sequence
	\begin{equation}\tag{\ref{t12}.4}
	\cdots\rightarrow\Tor_i^R(M,N)\to\Tor_i^R(X,N)\to\Tor_i^R(G,N)\rightarrow\cdots.
	\end{equation}
	In view of assumption (ii), (\ref{t12}.3) and (\ref{t12}.4) we have 
	\begin{equation}\tag{\ref{t12}.5}
	\Supp_R(\Tor_i^R(X,N))\subseteq\Supp_R(\Tor_i^R(M,N))\cup\Supp_R(\Tor_i^R(G,N))\subseteq\mz \text{ for all } i>0.
	\end{equation}
	 On the other hand, by the exact sequence (\ref{t12}.1) and Auslander--Bridger formula, we see that $\depth_{R_{\fp}}(M_\fp)=\depth_{R_{\fp}}(X_\fp)$ for all $\fp\in\mz\subseteq\Supp_R(M)$. Therefore, by assumption (i)		
	\begin{equation}\tag{\ref{t12}.6}
	\depth_{R_{\fp}}(X_\fp)+\depth_{R_{\fp}}(N_\fp)\geq\depth R_{\fp}+ n \quad \forall\fp\in\mz.
	\end{equation}
	Next we claim the following.
	
	{\bf Claim II.} $\Tor_i^R(X,N)=0$ for all $i>0$.\\
	Proof of claim. Set $t=\sup\{j\mid\Tor_{j}^R(X,N))\neq0\}$. Assume contrarily that $t>0$ and that $\fp\in\Ass_R(\Tor_t^R(X,N))$. By (\ref{t12}.5), $\fp\subseteq\mz$. As $\depth_{R_\fp}(\Tor_t^R(X,N)_\fp)=0$, by \cite[Theorem 1.2]{A2} we have the equality
	$\depth_{R_\fp}(X_\fp)+\depth_{R_\fp}(N_\fp)=\depth R_\fp-t,$
which is a contradiction by (\ref{t12}.6). Therefore $t=0$ and the proof of the claim is completed. 

	 By Claim II, the long exact sequence (\ref{t12}.4) induces the following exact sequence
	 \begin{equation}\tag{\ref{t12}.7}
	 0\rightarrow\Tor_1^R(G,N)\rightarrow M\otimes_RN\rightarrow X\otimes_RN\rightarrow G\otimes_RN\rightarrow0.
	 \end{equation}
	 Applying the functor $\Gamma_{\mz}(-)$ to the exact sequence (\ref{t12}.7), we obtain the following injection:
	 \begin{equation}\tag{\ref{t12}.8}
	 \hh^0_{\mz}(\Tor_1^R(G,N))\hookrightarrow\hh^0_{\mz}(M\otimes_RN).
	 \end{equation}
	 Note by (\ref{t12}.3) that $\Tor_1^R(G,N)$ is torsion with respect to $\mz$. Hence, by (\ref{t12}.2), Theorem \ref{Tate hom}(iv) and (\ref{t12}.8) we get the following injection:
	 $$\widehat{\Tor}_0^R(M,N)\cong\widehat{\Tor}_1^R(G,N)\cong\Tor_{1}^R(G,N)\cong\hh^0_{\mz}(\Tor_1^R(G,N))\hookrightarrow\hh^0_{\mz}(M\otimes_RN).$$	
	 Therefore, from now on we may assume that $n>0$. By Theorem \ref{t4} and Claim II,
	 \begin{equation}\tag{\ref{t12}.9}
	 \depth_{R_\fp}(X_\fp)+\depth_{R_\fp}(N_\fp)=\depth R_{\fp}+\depth_{R_\fp}((X\otimes_RN)_\fp)\quad \forall\fp\in\Supp_R(X\otimes_RN).
	 \end{equation}	
	  In view of (\ref{t12}.6) and (\ref{t12}.9), we have
	$\depth_{R_\fp}((X\otimes_RN)_\fp)\geq n$ for all $\fp\in\mz$. By Proposition \ref{grz} 
	\begin{equation}\tag{\ref{t12}.10}
	\hh^i_{\mz}(X\otimes_RN)=0 \text{  for  all  } 0\leq i\leq n-1.
	\end{equation}
	Consider the following two exact sequences, induce from (\ref{t12}.7):
	\begin{equation}\tag{\ref{t12}.11}
	0\rightarrow\Tor_1^R(G,N)\rightarrow M\otimes_RN\to Y\to0\ \ \text{ and }\ \ 0\to Y\rightarrow X\otimes_RN\rightarrow G\otimes_RN\rightarrow0.
	\end{equation}
	Applying the functor $\Gamma_{\mz}(-)$ to the above exact sequences and using Theorem \ref{LC}(ii), (iii), (\ref{t12}.10) and (\ref{t12}.2), one can easily obtain the following isomorphisms:
	\begin{equation}\tag{\ref{t12}.12}
	\hh^0_{\mz}(M\otimes_RN)\cong\hh^0_{\mz}(\Tor_1^R(G,N))\cong\Tor_1^R(G,N)\cong\widehat{\Tor}_1^R(G,N)\cong\widehat{\Tor}_0^R(M,N),
	\end{equation}
	\begin{equation}\tag{\ref{t12}.13}
	\hh^i_{\mz}(M\otimes_RN)\cong\hh^i_{\mz}(Y)\cong\hh^{i-1}_{\mz}(G\otimes_RN) \text{  for  all  } 0<i<n.
	\end{equation}
	Also, we have the following injection
	\begin{equation}\tag{\ref{t12}.14}
	\hh^{n-1}_{\mz}(G\otimes_RN)\hookrightarrow\hh^n_{\mz}(Y)\cong\hh^n_{\mz}(M\otimes_RN).
	\end{equation}
	 Note by (\ref{t12}.2) and assumption (ii) that $\Supp_R(\widehat{\Tor}^R_i(G,N))\subseteq\mz$ for all $i\in\ZZ$. As $G$ is totally reflexive, by Lemma \ref{pr} we have $\Supp_R(\Ext^i_R(\Tr G,N))\subseteq\mz$ for all $i>0$. As $M$ has finite Gorenstein dimension, by assumption (i) and Auslander--Bridger formula, we have $\depth_{R_\fp}(N_\fp)\geq\G-dim_{R_\fp}(M_\fp)+n\geq n$ for all $\fp\in\mz\subseteq\Supp_R(M)$. Therefore, by Proposition \ref{grz} we conclude that $\gr_R(\mz,N)\geq n$. Since $G$ is totally reflexive, so is $\Tr G$ by \cite[Lemma 4.9]{AB}.
	 Hence, by using Theorem \ref{Tate hom} and Lemma \ref{l2} and noting that $G^*\approx\Omega^2\Tr G$, we get the following isomorphisms:
	\[\begin{array}{rl}\tag{\ref{t12}.15}
	\hh^{j-1}_{\mz}(G\otimes_RN)&\cong\Ext^j_R(\Tr G,N)\\
	&\cong\widehat{\Ext}^j_R(\Tr G,N)\\
	&\cong\widehat{\Ext}^{j-2}_R(G^*,N)\\
	&\cong\widehat{\Tor}_{-j+1}^R(G,N),
	\end{array}\]
	for all  $1\leq j\leq n$. Now the first assertion is clear by (\ref{t12}.12), (\ref{t12}.13), (\ref{t12}.15) and (\ref{t12}.2).
	The second assertion follows from (\ref{t12}.14), (\ref{t12}.15) and (\ref{t12}.2).
\end{proof}

Next we start proving several corollaries of Theorem\ref{t12}. For the first corollary, see \ref{cx} for the definition of the complexity.

\begin{cor}\label{t11} Assume $R$ is a local complete intersection and let $\mz\subset\Spec R $ be a specialization-closed. Let $M, N\in\md(R)$, and let $c$ and $n$ be integers such that $c>\cx_R(M,N)$ and $n\geq c-1$. Assume the following hold:
\begin{enumerate}[\rm(i)]
\item{$\depth_{R_\fp}(M_\fp)+\depth_{R_\fp}(N_\fp)\geq \depth R_\fp+n$ for all $\fp\in\mz$.}
\item $\hh^i_{\mz}(M\otimes_RN)=0$ for all $i=n-c+1, \ldots, n$.
\item{$\Supp_R(\Tor_{i}^R(M,N))\subseteq\mz$ for all $i\geq 1$ (e.g., $\NF(M)\cap\NF(N)\subseteq\mz$).}
\end{enumerate}
Then it follows $\gr_R(\mz,M\otimes_RN)\geq n+1$, and $\Tor_i^R(M,N)=0$ for all $i\geq 1$.
\end{cor}

\begin{proof} We first note, by Theorem \ref{Tate2} and assumption (iii), that we have $\Supp_R(\widehat{\Tor}_i^R(M,N))\subseteq\mz$ for all $i\in\ZZ$. Hence, by Theorem \ref{t12} and assumption (ii), we get $\widehat{\Tor}_{-i}^R(M,N)=0$ for all $n-c+1\leq i\leq n$.
Consider to the following exact sequence $0\to L\to X\to M\to0$ where $X$ is totally reflexive and $\pd_R(L)<\infty$. By Theorem \ref{Tate hom}(i) we have $\widehat{\Tor}_i^R(L,N)=0$ for all $i\in\ZZ$. Therefore, the above exact sequence induces the following isomorphism	$\widehat{\Tor}_i^R(M,N)\cong\widehat{\Tor}_i^R(X,N)$ for all $i\in\ZZ$ (see for example \cite[2.9]{CJ}). Set $Y:=\Omega^{-(n+1)}X$. Note that $Y$ is totally reflexive and $X\approx\Omega^{n+1}Y$. In view of Theorem \ref{Tate hom} we obtain the following isomorphisms:
\begin{equation}\tag{\ref{t11}.1}
\Tor_i^R(Y,N)\cong\widehat{\Tor}_i^R(Y,N)\cong\widehat{\Tor}_{i-n-1}^R(X,N)\cong\widehat{\Tor}_{i-n-1}^R(M,N)=0 \text{ for all } 1\leq i\leq c.
\end{equation}
It follows from (\ref{t11}.1) and \cite[3.5]{CST} that $\widehat{\Tor}_i^R(Y,N)=0$ for all $i\in\ZZ$. Equivalently, it follows that $\widehat{\Tor}_i^R(M,N)\cong\widehat{\Tor}_i^R(X,N)=0$ for all $i\in\ZZ$. Thus, by Theorem \ref{t12}(i) and assumption (ii), we see that $\hh^i_{\mz}(M\otimes_RN)=0$ for all $0\leq i\leq n$. In other words, $\gr_R(\mz,M\otimes_RN)>n$. Also, by Theorem \ref{Tate hom}(iv) we have  $\Tor_{i}^R(M,N)\cong\widehat{\Tor}_i^R(M,N)=0$ for all $i\gg0$. Now the last assertion follows from assumptions (i), (iii) and the dependency formula (\ref{depen}.1).
\end{proof}	

The following is an immediate consequence of Corollary \ref{t11}.

\begin{cor}\label{cxmm} Assume $R$ is a local complete intersection, $M, N\in\md(R)$, and let $c$ and $n$ be integers. Assume $\NF(M)\cap\NF(N)\subseteq\{\fm\}$, $c>\cx_R(M,N)$ and that $n\geq c-1$. Assume further the  following hold:
	\begin{enumerate}[\rm(i)]
		\item $\depth_{R}(M)+\depth_{R}(N)\geq \depth R+n$.
		\item $\hh^i_{\fm}(M\otimes_RN)=0$ for all $i=n-c+1, \ldots, n$.
	\end{enumerate}
	Then it follows $\depth_R(M\otimes_RN)\geq n+1$ and $\Tor_i^R(M,N)=0$ for all $i\geq 1$.
\end{cor}




Next, in Corollary \ref{c2}, we improve a result of Celikbas, Iyengar, Piepmeyer and Wiegand \cite[3.14]{CIPW}. Note that the conclusion of Corollary \ref{c2} was obtained in \cite{CIPW} for the case where $c$ is at least the codimension of the ring in question; see also \cite[3.4]{Olgur} and \cite[7.7]{Daa}.

\begin{cor}\label{c2}(\cite{CIPW}) Assume $R$ is a local complete intersection, and let $M, N\in\md(R)$. Assume further hat the following conditions hold for some integer $c$ such that  $c>\cx_R(M,N)$.
\begin{enumerate}[\rm(i)]
\item $M$ and $N$ satisfy $(S_{c-1})$.
\item $M\otimes_RN$ satisfies $(S_{c})$.
\item $\Tor_{i}^R(M,N)_\fp=0$ for all $i\geq 1$ and all $\fp \in \X^{c-1}(R)$. 
\end{enumerate}
Then it follows that $\Tor_{i}^R(M,N)=0$ for all $i\geq 1$.
\end{cor}

\begin{proof} We consider to the specialization-closed subset $\mz=\{\fp\in\Spec R\mid \Ht\fp\geq c\}$ of $\Spec R$. Then it follows from assumption (iii) that $\Supp_R(\Tor_i^R(M,N))\subseteq\mz$ for all $i\geq 1$. Now the assertion follows from Proposition \ref{Serre}(ii) and Corollary \ref{t11} by letting $n$ equal to $c-1$.
\end{proof}

Corollary \ref{c} is another application of Theorem \ref{t12} which determines a new bound on depth of tensor products of modules satisfying the depth formula; see \ref{df}. As Corollary \ref{c} does not assume any Tor-vanishing and as the depth formula is associated with the Tor-vanishing, the conclusion of Corollary \ref{c} seems quite interesting to us, cf., \cite[3.1]{CST}. Note that Corollary \ref{cd}, advertised in the introduction, follows from Corollary \ref{c} and (\ref{cx}.1).  

\begin{cor}\label{c} Assume $R$ is a local complete intersection. Assume further $\NF(M)\cap\NF(N)\subseteq\{\fm\}$ for some $M, N\in\md(R)$. If $\depth_R(M)+\depth_R(N)-\depth R\geq \cx_R(M,N)$, then it follows that $\depth_R(M\otimes_RN)+ \depth R\leq \depth_R(M)+\depth_R(N)$.
\end{cor}

\begin{proof} Set $c= \cx_R(M,N)+1$ and $n=\depth_R(M)+\depth_R(N)-\depth R$. Then we have $n\geq c-1$. Assume contrarily that $\depth_R(M\otimes_RN)>n$. Therefore, $\hh^i_{\fm}(M\otimes_RN)=0$ for $n-c+1\leq i\leq n$. By Corollary \ref{t11},  it follows $\Tor_i^R(M,N)=0$ for all $i\geq 1$. In view of Theorem \ref{t4}, we have $\depth_R(M)+\depth_R(N)=\depth R+\depth_R(M\otimes_RN).$ Thus, $\depth_R(M\otimes_RN)=n$ which is a contradiction.
\end{proof}


\begin{eg} $\phantom{}$ The first item says that the bound $\depth_R(M\otimes_RN)+ \depth R\leq \depth_R(M)+\depth_R(N)$ presented in Corollary \ref{c} can be achieved.
	The second item shows that the locally free assumption is necessary in Corollary \ref{c}.
	\begin{enumerate}[\rm(i)]
		\item Let $R=k[\![x,y]\!]$ and let $M= N=\fm$. Then it follows $\cx_R(M,N)=0$.   So we have $\depth_R(M)=1$ and $\depth_R(M)+\depth_R(N)=\depth R+n$.
		In  particular, the assumptions of Corollary \ref{c} hold. Moreover, in view of \cite[3.1]{finitsup}, we see $\hh^0_{\fm}(M\otimes M)\simeq k\neq 0$. Therefore, it follows that $$\depth_R(M\otimes_RN)+ \depth R=0+2=1+1= \depth_R(M)+\depth_R(N).$$
	
		\item Let $R$ be a $3$-dimensional complete intersection and $x$ be a non zero-divisor on $R$. We look at $M=N=R/xR$. Then we have $\depth_R(M)+\depth_R(N)-\depth R=2+2-3=1>0= \cx_R(M,N)$.
	The following
	$\depth_R(M\otimes_RN)+ \depth R=2+3\nleqslant2+2=\depth_R(M)+\depth_R(N)$ shows that
the locally free assumption in 	  Corollary \ref{c} is really needed.
	\end{enumerate}
\end{eg}
Recall that $R$ is said to be an \emph{isolated singularity} if $R$ is local and $R_{\fp}$ is regular for each $\fp \in \X^{d-1}(R)$, where $d=\dim(R)$. The following result is immediate from Corollary \ref{c}; cf., \cite[4.7]{CST}.

\begin{cor}\label{+1c} Assume $R$ is a local hypersurface singularity, and let $M, N\in\md(R)$ such that $M$ is maximal Cohen-Macaulay. If $\depth_R(N)\geq 1$, then $\depth_R(N) \geq \depth_R(M\otimes_RN)$. Therefore, if $\depth_R(N)=1$, then $M\otimes_RN$ cannot be reflexive.
\end{cor}

\begin{cor}\label{ct} Assume $R$ is a Gorenstein local ring, and let $M, N\in\md(R)$ be maximal Cohen-Macaulay. Assume further that $\NF(M)\cap\NF(N)\subseteq\{\fm\}$ (e.g., $R$ is an isolated singularity.) Then at least one of the following conditions holds:
\begin{enumerate}[\rm(i)]
	\item $M\otimes_RN$ is a maximal Cohen-Macaulay $R$-module.
	\item $\depth_R(M\otimes_RN)=\inf\{i\geq0\mid\widehat{\Tor}_{-i}^R(M,N)\neq0\}.$
\end{enumerate}
\end{cor}

\begin{proof}
Assume that $M\otimes_RN$ is not maximal Cohen--Macaulay and set $n:=\depth(M\otimes_RN)+1\leq\depth R$. Then
$\depth_R(M)+\depth_R(N)\geq\depth R+n$. In view of Theorem \ref{t12}, we have  
\begin{equation}\tag{\ref{ct}.1}
\hh^i_{\fm}(M\otimes_RN)\cong\widehat{\Tor}_{-i}^R(M,N) \text{ for   all }  0\leq i<n.
\end{equation}
Note $\depth_R(M\otimes_RN)=\inf\{i\geq0\mid\hh^i_{\fm}(M\otimes_RN)\neq0\}.$
Therefore, by (\ref{ct}.1) we have  $\widehat{\Tor}_{-i}^R(M,N)=0$ for all $0\leq i<n-1$ and  $\widehat{\Tor}_{-(n-1)}^R(M,N)\neq0$, which completes the proof.
\end{proof}

\begin{cor} Assume $R$ is a $d$-dimensional hypersurface that has an isolated singularity, where $d\geq 2$, and $M, N\in \md(R)$ are nonfree and maximal Cohen-Macaulay. Then at least one of the following  holds:
\begin{enumerate}[\rm(i)]
\item $M\otimes_RN$ is torsion-free, $\depth_R(M\otimes_RN)=1$, and $\Tor_{2i-1}^R(M,N) \neq 0 = \Tor_{2i}^R(M,N)$ for all $i\geq 1$.
\item $M\otimes_RN$ has torsion, $\depth_R(M\otimes_RN)=0$, and $\Tor_{2i}^R(M,N) \neq 0 $ for all $i\geq 1$.
\end{enumerate}
\end{cor}

\begin{proof} It follows that $\depth_R(M\otimes_RN)\leq 1$; see \cite[4.7]{CST}. Therefore, since $d\geq 2$, $M\otimes_R N$ is not maximal Cohen-Macaulay. Hence, Corollary \ref{ct} shows that $\depth_R(M\otimes_RN)=\inf\{i\geq0\mid\widehat{\Tor}_{-i}^R(M,N)\neq0\}$. 

If $\depth_R(M\otimes_RN)=0$, then $M\otimes_RN$ has torsion. Moreover, we have that $\widehat{\Tor}_{0}^R(M,N)\neq 0$ and the claim follows from \ref{Tate hom}(iii) and the fact that $M\cong \Omega^2_R(M)$. 

Next assume $\depth_R(M\otimes_RN)=1$. In this case, as $1=\depth_R(M\otimes_RN)=\inf\{i\geq0\mid\widehat{\Tor}_{-i}^R(M,N)\neq0\}$, it follows that $\widehat{\Tor}_{0}^R(M,N)=0 \neq \widehat{\Tor}_{-1}^R(M,N)$. Furthermore, since $M$ and $N$ are locally free on the punctured spectrum of $R$, it follows that $M\otimes_RN$ is torsion-free. 
\end{proof}

For some special cases, we have the following variant of Corollary \ref{ct}.

\begin{cor}\label{atleast} Assume $R$ is a local complete intersection. Assume further $M, N\in\md(R)$ such that $\NF(M)\cap\NF(N)\subseteq\{\fm\}$. If $\depth_R(M)+\depth_R(N)\geq\depth R+\cx_R(M,N)$, then one of the following conditions holds:
\begin{enumerate}[\rm(i)]
\item{$\depth_R(M)+\depth_R(N)=\depth R+\depth_R(M\otimes_RN)$.}
\item{$\depth_R(M\otimes_RN)=\inf\{i\geq0\mid\widehat{\Tor}_{-i}^R(M,N)\neq0\}.$}
\end{enumerate}
\end{cor}

\begin{proof}
Set $n=\depth_{R}(M)+\depth_{R}(N)-\depth R$. By our assumption, we have $n\geq\cx_R(M,N)$. It follows from Corollary \ref{c} that $\depth_R(M\otimes_RN)\leq n$. If $\depth_R(M\otimes_RN)=n$, then (i) holds and we have nothing to prove. So let $\depth_R(M\otimes_RN)<n$. Then, by setting $\mz:=\V(\fm)$ and using Theorem \ref{t12}(i), we see
$\hh^i_{\fm}(M\otimes_RN)\cong\widehat{\Tor}_{-i}^R(M,N)$ for  all $0\leq i<n$. Now it is clear that (ii) holds.  
\end{proof}

As another application, we have the following non-vanishing result.

\begin{cor} Assume $R$ is a local complete intersection, and $M, N\in\md(R)$ such that $M$ is maximal Cohen-Macaulay and $\NF(M)\cap\NF(N)\subseteq\{\fm\}$. Assume further $\cx_R(M,N)<\depth_R(N)$. Then it follows that $\hh^i_{\fm}(M\otimes_RN)\neq0$ for some $i$, where $\depth_R(N)-\cx_R(M,N)\leq i\leq\depth_R(N)$. Therefore, $\depth_R(M\otimes_RN)\leq\depth_R(N)$.
\end{cor}

\begin{proof}
Set $c=\cx_R(M,N)+1$ and $n:=\depth_R(N)$. Assume contrarily that $\hh^i_{\fm}(M\otimes_RN)=0$ for all $i$ with $n-c+1\leq i\leq n$. By Corollary  \ref{t11}, $\Tor_i^R(M,N)=0$ for all $i>0$ and so by Theorem \ref{t4} we observe that $\depth_R(M\otimes_RN)=\depth_R(N)=n$ which is a contradiction, because $\hh^{n}_{\fm}(M\otimes_RN)=0$. 
\end{proof}

We proceed by giving some further applications of Theorem \ref{t12} on local cohomology modules of tensor product of modules over hypersurface rings. 

\begin{prop}\label{cx1} Let $\mz\subset\Spec R $ be  a specialization-closed subset and let $M, N\in\md(R)$ . Assume $\CI_R(M)<\infty$ and $\cx_R(M)\leq 1$ (e.g., $R$ is a hypersurface.) Assume, for an integer $n\geq0$, the following conditions hold:
\begin{enumerate}[\rm(i)]
	\item  $\depth_{R_\fp}(M_\fp)+\depth_{R_\fp}(N_\fp)\geq\depth R_{\fp}+n$ for all $\fp\in\mz$.
	\item  $\Supp_R(\Tor_i^R(M,N))\subseteq\mz$ for all $i\geq 1$ (e.g. $\NF(M)\cap\NF(N)\subseteq\mz$).  
\end{enumerate}
Then the following statements hold:
\begin{enumerate}[\rm(a)]
	\item If $n\geq3$, then $\hh^i_{\mz}(M\otimes_RN)\cong\hh^{i+2}_{\mz}(M\otimes_RN)$ for all $i=0, \ldots, n-3$. 
	\item If $\hh^n_{\mz}(M\otimes_RN)=0$, then $\hh^{n-2i}_{\mz}(M\otimes_RN)=0$ for all $i\geq0$ and $\widehat{\Tor}_{-n+2j}^R(M,N)=0$ for all $j\in\ZZ$. In particular, if $n$ is even, then $M\otimes_RN$ is torsion-free.
\end{enumerate}	
\end{prop}

\begin{proof}
	Set $t:=\depth R-\depth_R(M)$. As $M$ has a bounded Betti numbers, by \cite[Theorem 7.3]{AGP}, the minimal resolution of
	$\Omega^{t}M$ is periodic of period at most two, hence so is the minimal complete resolution of $M$. Therefore, we get the following isomorphisms:
	\begin{equation}\tag{\ref{cx1}.1}
	\widehat{\Tor}_i^R(M,N)\cong\widehat{\Tor}_{i+2}^R(M,N)  \text{ for all } i\in\ZZ.
	\end{equation}
	 Note that by Theorem \ref{Tate2} and assumption (ii), $\Supp_R(\widehat{\Tor}_i^R(M,N))\subseteq\mz$ for all $i\in\ZZ$. Now the first assertion follows from Theorem \ref{t12}(i) and (\ref{cx1}.1). To prove the second assertion, suppose that $\hh^n_{\mz}(M\otimes_RN)=0$. By Theorem \ref{t12}(ii), $\widehat{\Tor}_{-n}^R(M,N)=0$ and so by (\ref{cx1}.1), $\widehat{\Tor}_{-n+2j}^R(M,N)=0$ for all $j\in\ZZ$. Therefore, by using Theorem \ref{t12}(i) we see that $\hh^{n-2i}_{\mz}(M\otimes_RN)=0$ for all integer $i\geq0$.
\end{proof}

\begin{cor}\label{cq} Assume $R$ is a  local hypersurface ring, and let $M, N\in\md(R)$. Assume $M$ and $N$  are both maximal Cohen-Macaulay modules such that $\NF(M)\cap\NF(N)\subseteq\{\fm\}$. 
If $p$ (respectively, $q$) is an odd (respectively, even) integer such that $\max\{p,q\}<\dim R$ and $\hh^q_{\fm}(M\otimes_RN)=\hh^p_{\fm}(M\otimes_RN)=0$, then either $M$ or $N$ is free.
\end{cor}

\begin{proof} Note, in view of Corollary \ref{cx1}(b), we have $\widehat{\Tor}_{i}^R(M,N)=0$ for all $i\in\ZZ$. Therefore, $\Tor_i^R(M,N)=0$ for all $i\gg0$. Hence, by \cite[1.9]{HW1}, it follows that either $\pd(M)<\infty$ or $\pd(N)<\infty$, i.e., either $M$ or $N$ is free.
\end{proof}

\begin{eg}$\phantom{}$
\begin{enumerate}[(i)]
\item Let $R=k[\![x,y,z,w]\!]/(xy)$ and $M=N=R/xR$. Let $p=2$ and $q=1$. Then $\max\{p,q\}<\dim R$. Note that $M$ is maximal Cohen-Macaulay and 
$\hh^q_{\fm}(M\otimes_RN)=\hh^p_{\fm}(M\otimes_RN)=0$. Neither $M$ nor $N$ is free. Moreover, M is not locally free on the punctured spectrum of $R$. This example shows that the assumption on non-free locus in Corollary \ref{cq} is necessary.
\item Let $R$ be a local Cohen-Macaulay ring of dimension $d\geq 4$, and let $M=N=\fm$. Then $M$ and $N$ are not free, but are locally free over the punctured spectrum of $R$. So, in view of \cite[5.5(i)]{finitsup}, we see that $ \hh^{2}_{\fm}(M\otimes_RN)=\hh^{3}_{\fm}(M\otimes_RN) =0$. This example shows that the maximal Cohen-Macaulay assumption in Corollary \ref{cq} is necessary.
\end{enumerate}
\end{eg}

We end this section by the following result which is an extension of \cite[1.3]{CST}.

\begin{cor}\label{hhyper} Assume $R$ is a $d$-dimensional hypersurface with an isolated singularity, and let $M, N\in\md(R)$ be non-free and maximal Cohen-Macaulay. Then at least one of the following conditions holds:
\begin{enumerate}[\rm(i)]
	\item $\hh^i_{\fm}(M\otimes_RN)\neq 0$ for all $0\leq i\leq d$.
	\item $\hh^i_{\fm}(M\otimes_RN)\neq 0=\hh^j_{\fm}(M\otimes_RN)$ for all even integers $i$ and odd integers $j$ with $0\leq i,j<d$.
	\item $\hh^i_{\fm}(M\otimes_RN)\neq 0=\hh^j_{\fm}(M\otimes_RN)$ for all odd integers $i$ and even integers $j$ with $0\leq i,j<d$.
\end{enumerate}	
\end{cor}

\begin{proof}
This claims are consequences of Corollaries \ref{cx1} and \ref{cq}.
\end{proof}

\section{Splitting criteria for vector bundles}

In this section, we are concerned with some splitting criteria for vector bundles on schemes. Recall that a vector bundle is called \emph{trivial} if it is isomorphic to a direct sum of line bundles. Let $\mathcal{E}$ be a vector bundle on the projective scheme $X$ equipped with a very ample line bundle $\mathcal{O}(1)$. The dual of $\mathcal{E}$ is denoted by $\mathcal{E}^*:=\mhom(\mathcal{E},\mathcal{O}_X)$. We say that $\mathcal{E}$ has the \emph{cohomological rigidity property} provided that the the following condition holds: 
\begin{equation}\notag{}
\text{If } \hh^i_\ast(X,\mathcal{\mathcal{E}}\otimes \mathcal{\mathcal{\mathcal{E}}}^\ast):=\bigoplus_{n\in\mathbb{Z}}\hh^i(X,\mathcal{E}\otimes \mathcal{E}^\ast(n))=0 \text{ for some  } i\geq 1 \Longrightarrow \mathcal{E} \text{ is trivial. }
\end{equation}

The cohomological rigidity property, considered as above, was initially studied by Kempf \cite{Kem} and Luk and Yau \cite{Luk} by using different approaches. A commutative algebra interpretation of the cohomological rigidity property has been examined by Huneke and Wiegand \cite{HW1}. More precisely, given a reflexive module $M$ over a hypersurface ring $R$, Huneke and Wiegand determined new criteria for the freeness of the module $M$ in terms of the vanishing of $\hh^i_{\fm}(M\otimes_RM^*)$ for some $i$. Furthermore, Huneke and Wiegand generalized the aforementioned result of Luk and Yau via the machinery of commutative algebra. 

Our work in this section is motivated by the work of Huneke and Wiegand \cite{HW1}. The main purpose of this section is to obtain some new criteria for the freeness of modules in terms of the vanishing of local cohomology. Along the way, as an application, we obtain a new splitting criteria - for an arithmetically Cohen-Macaulay vector bundle on a smooth complete intersection of an odd dimension --  in terms of the cohomological rigidity property; see Corollary \ref{artCM}. 

We start by the following result which is an extension of a result of Huneke and Wiegand; see \cite[4.1(1)]{HW1}, and also \ref{LC} for the terminology.

\begin{thm}\label{t7} Let  $R$ be a ring and let $\mz\subset\Spec R $ be a specialization-closed subset where $\gr_R(\mz,R)\geq1$. Assume $M\in\md(R)$ is such that $\NF(M)\subseteq\mz$. If $\hh^1_{\mz}(M\otimes_RM^*)=0$, then $M=F\oplus T$, where $F$ is free and $T$ is $\mz$-torsion. In particular, if $M$ is torsion-free with respect to $\mz$, then $M$ is free.
\end{thm}

\begin{proof}
	In view of Lemma \ref{gr} and our assumption, $\gr_R(\mz,M^*)\geq\min\{2,\gr_R(\mz,R)\}\geq1$. Hence, by using Lemma \ref{l2}(ii) we conclude that
	\begin{equation}\tag{\ref{t7}.1}
	\Ext^2_R(\Tr M,M^*)=0.
	\end{equation}
	 It follows from (\ref{t7}.1) and (\ref{a1}.5) that $\Tor_1^R(\Omega^{-2}M,M^*)=0$. Note that $\Omega^{-2}M\approx\Tr\Omega^2\Tr M\approx\Tr M^*$ (see \ref{a1}). Therefore, $\Tor_1^R(\Tr M^*,M^*)=0$. In other words, by (\ref{a1}.4), the natural map $M^\ast\otimes M^{\ast\ast }\to \Hom(M^{\ast},M^{\ast})$ is surjective. In view of \cite[A.1]{ag},  we conclude that $M^*$ is free. It follows from (\ref{t7}.1) that $\Ext^2_R(\Tr M,R)=0$.
	 Therefore, by (\ref{a1}.3), the natural map $e_M:M\to M^{**}$ is surjective and we get the following exact sequence $0\rightarrow\Ext^1_R(\Tr M,R)\rightarrow M\rightarrow M^{**}\rightarrow0$. As $M^{**}$ is free, the above exact sequence splits. Hence, $M\cong F\oplus T$, where $T=\Ext^1_R(\Tr M,R)$ which is $\mz$-torsion and $F:=M^{**}$.
\end{proof}

\begin{cor}\label{c7} Assume $R$ is local and, $M\in\md(R)$, and $\fa$ is an ideal of $R$ of positive grade. Then the following statements hold:
\begin{enumerate}[\rm(i)]
\item If $\NF(M)\subseteq\V(\fa)$ and $\hh^1_{\fa}(M\otimes_RM^*)=0$, then $M=F\oplus T$, where $F$ is free and $T$ is $\fa$-torsion. Therefore, if $M$ is torsion-free with respect to $\fa$, then $M$ is free.
\item If $\NF(M)\subseteq\V(\fm)$ and $\hh^1_{\fm}(M\otimes_RM^*)=0$, then $M=F\oplus T$, where $F$ is free and $T$ is a finite length module. Therefore, if $\depth_R(M)\geq 1$, then $M$ is free.
\end{enumerate}
\end{cor}
\begin{proof} This follows immediately from Theorem \ref{t7}.
\end{proof}

To facilitate things, we  bring the following:
\begin{discussion}\label{dis}
\begin{enumerate}[(i)]
	\item	By isolated Gorenstein singularity we mean a non Gorenstein ring $(R,\fm)$ which is Gorenstein over the punctured spectrum.

\item   Let  $(R,\fm)$ be a  Cohen-Macaulay local ring of dimension $d>1$
with isolated Gorenstein singularity and possessing a canonical module.  In view of \cite[5.8]{finitsup}, we know that $\hh^{i}_{\fm}(\omega_R\otimes _R \omega_R^\ast)\neq0$ if and only if $i\leq1$ or $i=d$. 

\item
Let us reprove the claim   $\hh^{i}_{\fm}(\omega_R\otimes _R \omega_R^\ast)=0$ for all $2\leq i<\dim( R)$, by the methods developed  in this paper. Indeed,  $\omega_R$ is locally free on the punctured spectrum of $R$. Let $i\geq 1$. Then, in view of  Lemma \ref{t1}, we have that $\hh^{i}_{\fm}(\omega_R\otimes _R \omega_R^\ast)=\Ext^{i-1}_R(\omega_R,\omega_R)=0$ for all $i$ where $2\leq i<\depth_R(\omega_R)=\dim(R)$.

\item Any one-dimensional integral domain which is not Gorenstein is a Cohen-Macaulay ring with isolated Gorenstein singularity. For example,
let $R:=k[[x^3,x^4,x^5]]$. 

\item Any two-dimensional normal domain which is not Gorenstein is a Cohen-Macaulay ring with isolated Gorenstein singularity. For example, $R:=k[[x^3,x^2y,xy^2,y^3]]$.

\item To see a $3$-dimensional situation, let $S := k[[x, y, z, u, v]]$ and put $ R := S/(yv-zu, yu-xv, xz-y^2)$, where $k$ is an algebraically closed field of characteristic
different from $2$. By \cite[Example 16.2]{lw}, $R$ is a $3$-dimensional Cohen-Macaulay ring with isolated Gorenstein singularity.
In fact, its canonical module is isomorphic to
the ideal $(u, v)$.
\end{enumerate}
\end{discussion}


\begin{eg}\label{h1implies} $\phantom{}$
The first  item shows the importance of the locally free assumption. The  second (and the third) item shows that the spot $1$ in the vanishing of local cohomology module is crucial.
\begin{enumerate}[(i)]
\item Let $R=k[\![x,y,z]\!]/(x^2)$ and let $M=R/xR$. Note that $M \cong M^\ast$ and $\hh^{ 1}_{\fm} (M\otimes_RM^\ast)=0$.
It is clear that $M$ is not of the form $F\oplus T$, where $F$ is free and $T$ has finite length.

\item Let $R$ be a  Cohen--Macaulay local ring with isolated Gorenstein singularity with a canonical module $\omega_R$, and of dimension $d>2$. Such a thing exists,
e.g., look at  $R:=\frac{\mathbb{C}[[x, y, z, u, v]]}{(yv-zu, yu-xv, xz-y^2)}$.
By the above discussion,  $\omega_R$ is locally free on the punctured spectrum of $R$ and  $\hh^{2}_{\fm}(\omega_R\otimes _R \omega_R^\ast)=0$.  It is clear that $\omega_R$ is not of the form $F\oplus T$, where $F$ is free and $T$ has finite length.

\item Let $(R,\fm,k)$ be a $d$-dimensional $(d>2)$  Cohen--Macaulay  local ring
and set $M:=\fm$. Then $\depth_R(M)=1$ and $M$ is  locally free.  We know that $\fm^\ast\simeq R$. It is easy to see $\hh^i_{\fm}(M\otimes_RM^*)\simeq\hh^i_{\fm}(\fm)=0$ for all $2\leq i\leq d-1$. However, $M$ is not free. In fact $\pd_R(M)=\infty$ provided $R$ is singular. Also, $\pd_R(M)=d-1>0$ provided $R$ is nonsingular.

\end{enumerate}
\end{eg}

\begin{cor}\label{c1} Assume $R$ is local complete intersection ring of dimension $d\geq2$ and let $\mz\subset \Spec R$ be   specialization-closed. Let $M\in\md(R)$ be such that $\NF(M)\subseteq\mz$. If $\hh^{n}_{\mz}(M\otimes_RM^*)=0$ for an odd integer $n$ with $1\leq n\leq \gr_R(\mz,M)$, then $\pd_R(M)<\infty$.
\end{cor}

\begin{proof}
If $n=1$, then the assertion is clear by Corollary \ref{c7}. Now let $n>1$. In view of Proposition \ref{grz} and \cite[1.4.1(b)]{BH}, $M$ is reflexive and so $M\approx M^{**}\approx\Omega^2\Tr M^*$. It follows from Lemma \ref{l2} that $\Ext^{n-1}_R(M,M)\cong\Ext^{n+1}_R(\Tr M^*,M)=0$. Hence, by \cite[4.2]{AvBu}, $\pd_R(M)<\infty$.
\end{proof}


\begin{eg}  $\phantom{}$
\begin{enumerate}[(i)]
\item Here, we show that the complete-intersection assumption in Corollary \ref{c1} is important. To this end,  let $R$ be of isolated Gorenstein singularity and of dimension at least four. Let
$n:=3$ and set $M:=\omega_R$. In view of Discussion  \ref{dis}(ii) we know $\hh^{3}_{\fm}(M\otimes_RM^*)=0$. In order to complete this item, we need to recall that $\pd(M)=\infty$. 
\item Let $R=k[\![x,y,z,u,v]\!]/(xy)$, $M=R/xR$ and let $n=3$. Then $4=\depth_R(M)\geq n$, $M\simeq M^\ast$  and that $M\otimes_RM\simeq M$. Therefore, 
$\hh^3_{\fm}(M\otimes_RM^\ast)\simeq\hh^3_{\fm}(M)=0$. However, $\pd(M)=\infty$. This example shows the importance of the assumption on the non-free locus in Corollary \ref{c1}.
\item Let $R=k[\![x,y,u,v]\!]/(xy-uv)$ and let $M=(x,u)\subseteq R$. Recall from Example \ref{t3ri} that $M$ is locally free on the punctured spectrum of $R$ and is maximal Cohen-Macaulay. Set $n=2$. Then $n<\depth_R(M)$ and $\hh^{2}_{\fm}(M\otimes_R M^*)\simeq\hh^{2}_{\fm}(\fm)=0$. However, we have $\pd_R(M)=\infty$. This example shows that the oddness of the integer $n$ is necessary in Corollary \ref{c1}.
\end{enumerate}
\end{eg}


\begin{cor}\label{ceo} Let $R$ be a complete intersection local ring of dimension $d\geq 2$, and let $M\in\md_0(R)$. If $M$ is maximal Cohen--Macaulay and $\hh^i_{\fm}(M\otimes M^*)=0$ for an odd integer $i$ with $0<i<d$, then $M$ is free.
\end{cor}

\begin{proof}
This is an immediate consequence of Proposition \ref{c1}.
\end{proof}

\begin{eg}\label{ceoex} $\phantom{}$The first item shows that oddness of $i$ in Corollary \ref{ceo} is crucial. The second item shows that the maximal Cohen--Macaulay assumption is crucial.
\begin{enumerate}[\rm(i)]
\item	 $R=k[\![x,y,u,v]\!]/(xy-uv)$ and $M=(x,u)$. Then, by Example \ref{t3ri}, it follows that $M$ is  maximal Cohen-Macaulay and  $\hh_{\fm}^2(M\otimes_RM^\ast)=0$. However $M$ is not free.

\item Let $R=k[\![x,y,z,v,w]\!]$ and let $M= \Syz^{d-1}k$. Then, by \cite[5.5(ii)]{finitsup}, we see $ \hh^{3}_{\fm}(M\otimes_RM^\ast)=0$. However, $M$ is not free. As $M$ is not maximal Cohen--Macaulay, this example shows that the depth assumption on $M$ is necessary in Corollary \ref{ceo}.


\end{enumerate}
\end{eg}

\begin{prop}\label{tp3} Assume $R$ is local and a complete intersection of dimension $d$, where $d\geq4$, and let $\mz\subseteq \Spec R $ be a specialization-closed subset. Assume further $M\in\md(R)$ such that $\NF(M)\subseteq\mz$ and $\gr_R(\mz,M)\geq3$. Then the following hold:
	\begin{enumerate}[\rm(i)]
		\item{If $\hh^3_{\mz}(M\otimes_RM^*)=0$, then it follows that $\pd_R(M)\leq1$.}
		\item{If $\hh^2_{\mz}(M\otimes_RM^*)=0=\hh^3_{\mz}(M\otimes_RM^*)$, then it follows that $M$ is free.}
	\end{enumerate}
\end{prop}


\begin{proof} First note, by Proposition \ref{grz}, we have that $M$ satisfies $(S_3)$. Therefore, $M$ is reflexive and so $M\approx M^{**}\approx\Omega^2\Tr M^*$. Assume that $\hh^i_{\mz}(M\otimes_RM^*)=0$ for some $i$ with $2\leq i\leq 3$. Then, in view of Lemma \ref{l2} we conclude that $\Ext^{i-1}_R(M,M)\cong\Ext^{i+1}_R(\Tr M^*,M)=0.$ Now the assertions follow from \cite[2.5]{Jo}.
\end{proof}

Next we prove a generalization of \cite[4.2]{HW1}. In the following, for simplicity, we denote the Matlis dual functor by $(-)^v:=\Hom_R(-,E_R(k))$, where $E_R(k)$ is the injective hull of the residue field $k$.

\begin{prop}\label{tomax} Assume $R$ is Gorenstein and local of dimension $d$, and let $\mz\subset\Spec R $ be a specialization-closed subset. Let $M, N\in\md(R)$ be such that $\NF(N)\cup\NF(M)\subseteq\mz$. Then, if $1< i<\gr(\mz,R)$, it follows that $$\hh_{\mz}^i(M^\ast\otimes_RN^\ast)\cong \hh_{\fm}^{d-i+1}(M \otimes_RN )^v.$$ 
\end{prop}

\begin{proof}
There is a natural map $\Phi:M^*\otimes_RN^*\to(M\otimes_RN)^*$ taking $f\otimes g$ to the map
$x\otimes y\mapsto f(x)\cdot g(y)$. It is easy to see that $\Phi$ is an isomorphism if either $M$ or $N$ is free. Hence, by our assumption, $\ker(\Phi)$ and $\coker(\Phi)$ are $\mz$-torsion. It follows from parts (ii) and  (iii) of Theorem \ref{LC} that	
\begin{equation}\tag{\ref{tomax}.1}
\hh_{\mz}^i(M^\ast\otimes_RN^\ast)\cong\hh_{\mz}^i((M\otimes_RN)^\ast) \text{ for  all } i>1.
\end{equation}
Set $L=M\otimes_RN$. Consider the natural map $\Psi:L\to L^{**}$. As $\NF(L)\subseteq\mz$, we deduce that $\Psi_{\fp}$ is an isomorphism for all $\fp\in\Spec R\setminus\mz$. Hence, by using Proposition \ref{grz} and (\ref{gra}.1) we observe that
$\gr_R(\mz,R)\leq\min\{\gr_R(\ker(\Psi)), \gr_R(\coker(\Psi))\}$. In other words, we have that $\Ext^i_R(\ker(\Psi),R)=0=\Ext^i_R(\coker(\Psi),R)$ for $i<\gr_R(\mz,R)$. Hence, the natural map $\Psi$ induces the following isomorphism:
\begin{equation}\tag{\ref{tomax}.2}
\Ext^i_R(L,R)\cong\Ext^i_R(L^{**},R) \text{ for  all } i<\gr_R(\mz,R)-1.
\end{equation}
Note that $L^{**}\approx\Omega^2\Tr L^*$. Therefore, by Lemma \ref{l2}(i), we obtain the following isomorphisms:
\begin{equation}\tag{\ref{tomax}.3}
\hh^i_{\mz}(L^*)\cong\Ext^{i+1}_R(\Tr L^*,R)\cong\Ext^{i-1}_R(L^{**},R) \text{ for all } 1<i<\gr_R(\mz,R).
\end{equation}
Now the assertion is clear, by (\ref{tomax}.1), (\ref{tomax}.2), (\ref{tomax}.3) and the Theorem \ref{Ld}.
\end{proof}

\begin{thm}\label{t3} Assume $R$ is local and a complete intersection of dimension $d$, where $d\geq4$, and let $M\in\md_0(R)$ such that $\depth_R(M)\geq3$. Then the following hold:
\begin{enumerate}[\rm(i)]
\item If $\hh^i_{\fm}(M\otimes_RM^*)=0$ for some $i\in\{3, d-2\}$, then it follows $\pd_R(M)\leq1$.
\item If $\hh^{i}_{\fm}(M\otimes_RM^*)=0=\hh^j_{\fm}(M\otimes_RM^*)$ for some $i\in\{2, d-1\}$ and $j\in\{3,d-2\}$, then $M$ is free.
\end{enumerate}
\end{thm}

\begin{proof}
This follows easily from Proposition \ref{tp3} and Proposition \ref{tomax}.
\end{proof}

\begin{eg} Let $R=k[\![x,y,z,u,v]\!]$ and let $M=\Syz^4k$. Then $\depth_R(M)=4$. In view of \cite[5.5(ii)]{finitsup}, we know that $\hh^3_{\fm}(M\otimes_RM^*)=0$.
Recall that $\pd_R(M)=1$. Thus the example shows:
\begin{enumerate}[\rm(i)]
\item In Theorem \ref{t3}(i), the bound on the projective dimension is sharp.
\item In Theorem \ref{t3}(ii), the vanishing in two spots is necessary.
\end{enumerate}
\end{eg}

At this point we do not know whether or not the depth assumption on $M$ in Theorem \ref{t3} is necessary. In some cases one can deduce a similar result for reflexive modules.

\begin{thm}\label{t3.1}(\v{C}esnavi\v{c}ius \cite{ces}) Let $R$ be a graded normal complete intersection over a field of dimension $d\geq4$. Assume $M\in\md_0(R)$, where $\depth_R(M)\geq2$. Then $M$ is free provided that at least one of the following conditions holds:
\begin{enumerate}[\rm(i)]
\item $\hh^2_{\fm}(M\otimes_RM^*)=\hh^3_{\fm}(M\otimes_RM^*)=0$.
\item $\hh^{d-2}_{\fm}(M\otimes_RM^*)=\hh^{d-1}_{\fm}(M\otimes_RM^*)=0$.
\end{enumerate}
\end{thm}
\begin{proof} Note that there is a homogeneous ideal $I$ of $k[X_0,\ldots,X_n]$ such that $R=\frac{k[X_0,\ldots,X_n]}{I}$. Let $X:= \textmd{Proj}(R)$  and $\mathcal{E}:=\widetilde{M}$.
Then $X\subset \mathbb{P}^n$ is globally complete intersection,  $\mathcal{E}$ is  a vector bundle  and that $\bigoplus_{n\in \mathbb{Z}} \hh^0(X,\mathcal{O}_X(n))\simeq R$. Also,
$$\bigoplus_{i\in\mathbb{Z}}\hh^1 (X,\mathcal{E}nd_{\mathcal{O}_X}(\mathcal{\mathcal{E}})(i))=\bigoplus_{i\in\mathbb{Z}}\hh^1 (X,\mathcal{\mathcal{E}}\otimes \mathcal{\mathcal{\mathcal{E}}}^\ast(i)) =\hh^2_{\fm}(M\otimes_RM^*)=0.$$Similarly, one can show that $ \hh^2_\ast(X,\mathcal{E}nd_{\mathcal{O}_X}(\mathcal{\mathcal{E}}))=0$.
In view of \cite[1.2]{ces}, $\mathcal{E}$ is direct sum of powers of $\mathcal{O}(1)$.
 It follows that $M= \hh^0_\ast(X, \mathcal{E}  )$ is free.
The second part follows from the first part and Proposition \ref{tomax}.
\end{proof}

If $R$ is a Gorenstein local ring of even dimension, and $M\in\md_0(R)$ is maximal Cohen--Macaulay such that $\CI_R(M)<\infty$ (e.g., $R$ is complete intersection) and $\depth_R(M\otimes_RM^{\ast})\geq 1$, then $M$ is free; see \cite[3.9]{CeD}. Note that one needs a ring of even dimension for this result to hold; see \cite[3.12]{CeD}. In the following we show that one can replace the condition $\depth_R(M\otimes_RM^{\ast})\geq 1$ with the vanishing of $\hh^n_{\fm}(M\otimes_RM^*)$ for some integer $n$ with $0\leq n<d$, and generalize \cite[3.9]{CeD} to obtain a stronger result.

\begin{thm}\label{gor} Assume $R$ is local and Gorenstein of even dimension $d$. If $M\in\md_0(R)$ is maximal Cohen--Macaulay such that $\CI_R(M)<\infty$ and $\hh^n_{\fm}(M\otimes_RM^*)=0$ for some integer $n$ with $0\leq n<d$, then $M$ is free.
\end{thm}

\begin{proof} Note that, since $M\in\md_0(R)$, we may assume $d\geq 2$. Note also that, the case where $n=0$ is \cite[3.9]{CeD}. Moreover, the case where $n=1$ is a consequence of Corollary \ref{c7}(ii). Hence we may assume $n\geq 2$.

If $n$ is odd, then the assertion follows from Corollary \ref{c1}. Next assume $n$ is even. Then Proposition \ref{tomax} implies that $\hh^{d-n+1}_{\fm}(M\otimes_RM^*)=0$. As $M$ is reflexive, $M\approx M^{**}\approx\Omega^2\Tr M^*$. Therefore, by Lemma \ref{l2}, we have that $\Ext^{d-n}_R(M,M)\cong\Ext^{d-n+2}_R(\Tr M^*,M)=0$. Consequently, as $d-n$ is even, we conclude that $\pd_R(M)<\infty$, i.e., $M$ is free; see \cite[4.2]{AvBu}. 
\end{proof}



Let $X$ be a  projective variety over an algebraically closed field $k$. For each vector bundle $\mathcal{E}$, we denote $\Gamma_\ast(\mathcal{E}):=\bigoplus_{i\in\mathbb{Z}}\Gamma (X, \mathcal{E} (i))$.
Recall that  $ \mathcal{E}$  is called \emph{arithmetically Cohen--Macaulay} if $\hh^i (X, \mathcal{E}(j))=0$ for all $i = 1, \ldots , \dim X-1$ and all $j\in \mathbb{Z}$. The following result should be compared with \cite[1.5]{D1} and  \cite[1.2]{ces}. In fact, for an arithmetically Cohen-Macaulay vector bundle on a smooth complete intersection of odd dimension, we have a stronger result as we state next:

\begin{cor}\label{artCM} Let $k$ be an algebraically closed  and let $X\subset \mathbb{P}^m_k$ be a globally complete intersection of dimension $d\geq3$. Assume that $\mathcal{E}$ is a vector bundle and that $\hh^{n}(X, \mathcal{E}\otimes\mathcal{E}^*(j))=0$ for all $j\in\ZZ$ and for some  $0<n<d$. The following statements hold:
\begin{enumerate}[\rm(i)]
\item If $n$ is even and $\depth (\Gamma_\ast(\mathcal{E}))\geq n+1$, then $\pd (\Gamma_\ast(\mathcal{E}))<\infty$ over its affine cone. In particular, if $\mathcal{E}$ is arithmetically Cohen--Macaulay, then  it is a direct sum of powers of $\mathcal{O}(1)$.
\item If $d$ is odd and $\mathcal{E}$ is arithmetically Cohen--Macaulay, then $\mathcal{E}$ is a direct sum of powers of $\mathcal{O}(1)$.
\end{enumerate} 
\end{cor}

\begin{proof} Note that there is a homogeneous ideal $I$ of $k[X_0,\ldots,X_m]$ generated by regular sequence such that  $X= \textmd{Proj}(R)$ where $R=\frac{k[X_0,\ldots,X_m]}{I}$. Recall that $\bigoplus_{n\in \mathbb{Z}} \hh^0(X,\mathcal{O}_X(n))\simeq R$ and that $\dim(R)=\dim(X)+1$.
There is an $R$-module $M$ such that  $\mathcal{E} =\widetilde{M}$. In fact $M=\Gamma_\ast(\mathcal{E})$ which is graded and reflexive. Recall that arithmetically Cohen-Macaulay bundles correspond to maximal Cohen--Macaulay modules over the associated graded ring. From this, $M_{\fm}$ is maximal Cohen--Macaulay if and only if $\mathcal{E}$ is arithmetically Cohen-Macaulay.
Also, we have the following isomorphism:

\begin{equation}\tag{\ref{artCM}.1}
\hh^{n+1}_{\fm}(M\otimes_RM^*)\cong\bigoplus_{i\in\mathbb{Z}}\hh^n (X,\mathcal{\mathcal{E}}\otimes \mathcal{\mathcal{\mathcal{E}}}^\ast(i))=0.
\end{equation}
	\begin{enumerate}[\rm(i)]
	\item  In view of Proposition \ref{c1} and (\ref{artCM}.1), we see $\pd(M_{\fm})$ is finite over $R_{\fm}$. Also, due to \cite[1.5.15(e)]{BH}, we know that $\pd_R(M)<\infty$.	
	\item Note that $M=\Gamma_\ast(\mathcal{E})$ is  maximal Cohen--Macaulay and that $\mathcal{E} =\widetilde{M}$. In view of the Theorem \ref{gor} and (\ref{artCM}.1), $\pd(M_{\fm})$ is zero over $R_{\fm}$. Also, due to\cite[1.5.15(e)]{BH}, we know that $\pd_R(M)=0$.
	\end{enumerate} 
As $M$ is graded, there is a finite set $L\subset \mathbb{Z}$ such that $M=\bigoplus_{ \ell\in L}R(\ell)$. Therefore the following observation completes the proof: $\mathcal{E} =\widetilde{M}=\bigoplus_{ \ell\in L}\widetilde{R(\ell)}=\bigoplus_{ \ell\in L}\mathcal{O}(1)^{\ell}$.
\end{proof}

\begin{eg}$\phantom{}$
\begin{enumerate}[\rm(i)]
\item Let $R=k[x,y,u,v]/(xy-uv)$ and $M=(x,u)$. Let $X=\textmd{Proj}(R)$ and $\mathcal{E}:=\widetilde{M}$. Then, by Example \ref{ceoex}(i), we have $\hh^1_{\ast}(X,\mathcal{E}nd_{\mathcal{O}_X}(\mathcal{\mathcal{E}}))=0$. However, $\mathcal{E}$ is not trivial.
\item Let $X\subset \mathbb{P}^5_k$ be a smooth hypersurface of degree $d \geq 2$,  and let $\mathcal{E}$ be an indecomposable arithmetically Cohen--Macaulay vector bundle of rank two. Then, by \cite[1.1.2]{muhan}, it follows that $\hh^1_{\ast}(X,\mathcal{E}nd_{\mathcal{O}_X}(\mathcal{\mathcal{E}}))=0$. Since $\mathcal{E}$ is indecomposable, this example shows that $n$ needs to be an even integer in Corollary \ref{artCM}(i).
\end{enumerate}
\end{eg}

Next is an application of Lemma \ref{duality}; we use it to prove Theorem \ref{ri}.

\begin{prop}\label{ll} Assume $R$ is Gorenstein and local, and let $M\in\md(R)$. If $M$ is locally free on $\X^{n-1}(R)$ for some integer $n\geq 1$ and satisfies Serre's condition $(S_{n+1})$, then the following hold:
	\begin{enumerate}[\rm(i)]
		\item $\Ext^i_R(M,M)\cong\Ext^i_R(M\otimes_RM^*,R)$ for all $i$ with $1\leq i\leq n-1$.
		\item There is an injection $\Ext^n_{R}(M,M)\hookrightarrow\Ext^n_R(M\otimes_RM^*,R)$.
	\end{enumerate}
\end{prop}

\begin{proof} First note that, by \cite[4.25]{AB}, $M$ is $(n+1)$-torsionfree, i.e., $\Ext^i_R(\Tr M,R)=0$ for all $i=1, \ldots, n+1$ Consequently, $M$ is reflexive and $\Ext^i_R(M^*,R)=0$ for all $i=1, \ldots, n-1$. By our assumption, $\gr_R(\Tor_i^R(M,M^*))\geq n$ for all $i\geq 1$. Now the assertion is clear by Lemma \ref{duality}.
\end{proof}

From now on we use the following result of Jothilingam without further reference; see \cite[Theorem]{Joth} and \cite[3.1.2]{Dao}.	

\begin{lem}\label{Jo}(Jothilingam) Assume $R$ is local ring and let $M\in\md(R)$ be Tor-rigid. If $\Ext^n_R(M,M)=0$ for some $n\geq 1$, then $\pd_R(M)<n$.
\end{lem}

In the following $\overline{G}(R)_{\mathbb{Q}}$, i.e., $G(R)_{\mathbb{Q}}$ denotes the reduced Grothendicek group with rational coefficients; we refer the reader to \cite{D1} for the definition and basic properties of this group.

\begin{thm}\label{ri} Assume $R$ is $d$-dimensional, Gorenstein and local, and let $M\in\md(R)$. Assume $M$ is locally free on $\X^{n-1}(R)$ for some integer $n$, where $0<n<d$. Assume further that $M$ satisfies  Serre's condition $(S_{n+1})$ and $\hh^{d-n}_{\fm}(M\otimes_RM^*)=0$. 
\begin{enumerate}[\rm(i)]
\item If $n$ is even and $\CI_R(M)<\infty$, then it follows that $\pd_R(M)<d-n$.
\item If $M$ is Tor-rigid, then it follows that $\pd_R(M)<n$.
\item If $R$ is a hypersurface which is quotient of an unramified regular ring, and the class of $M$ is zero in $\overline{G}(R)_{\mathbb{Q}}$, then it follows that $\pd_R(M)<n$.
\end{enumerate}
\end{thm}

\begin{proof} We have, by Theorem \ref{Ld}, that $\Ext^n_R(M\otimes_RM^*,R)=0$. Then, in view of Proposition \ref{ll}, it follows that:
\begin{equation}\tag{\ref{ri}.1}
\Ext^n_R(M,M)=0.
\end{equation}
(i) This follows from (\ref{ri}.1) and \cite[4.2]{AvBu}.\\
(ii) This follows from (\ref{ri}.1) and Lemma \ref{Jo}. \\
(iii) This follows from (\ref{ri}.1) and \cite[5.4]{D1}.
\end{proof}

The following is an immediate consequence of Theorem \ref{ri}(ii).

\begin{cor} If $R$ is regular and local of dimension $d\geq1$, and $M\in\md(R)$ is reflexive such that $\hh^{d-1}_{\fm}(M\otimes_RM^*)=0$, then $M$ is free.	
\end{cor}

The following result is to be compared with \cite[4.1(2)]{HW1}.

\begin{prop}\label{pr3} Assume $R$ is Cohen-Macaulay and local, and let  $\mz\subset\Spec R $ be a specialization-closed subset. Let $M\in\md(R)$ be Tor-rigid and $\NF(M)\subseteq\mz$. If $\hh^n_{\mz}(M\otimes_RM^*)=0$ for some $n$, where $2\leq n\leq\gr_R(\mz,M)$, then $\pd_R(M)<n-1$.
\end{prop}

\begin{proof} First note, by Proposition \ref{grz}, that $\gr_R(\mz,M)=\inf\{\depth_{R_\fp}(M_\fp)\mid\fp\in\mz\}\geq2$. Therefore, $M_\fp$ is free for all $\fp\in\X^{1}(R)$. It follows from \cite[1.4.1(b)]{BH} that $M$ is reflexive and so $M\approx M^{**}\approx\Omega^2\Tr M^*$. Hence, by Lemma \ref{l2} and our assumption, we have $\Ext^{n-1}_R(M,M)\cong\Ext^{n+1}_R(\Tr M^*,M)=0$. Now the assertion is clear by Lemma \ref{Jo}.
\end{proof}

The Tor-rigidity condition is necessary in Proposition \ref{pr3}; see Example \ref{t3ri}. We now give an example to consider the other hypotheses.

\begin{eg} $\phantom{}$
\begin{enumerate}[(i)]

\item Let $R=k[\![x,y,z,u,v]\!]$ and let $M=\Syz^4k$. Then $\depth_R(M)=4$. In view of \cite[5.5(ii)]{finitsup} we know that $\hh^3_{\fm}(M\otimes_RM^*)=0$.
We apply Proposition \ref{pr3} for the case where $n=3<\depth_R(M)$ to see $\pd_R(M)<n-1=2$. Note that $\pd_R(M)=1$. This example shows that the bound on the projective dimension of $M$ in Proposition \ref{pr3} is sharp.

\item Let $R=k[\![x,y]\!]$, $M=\Syz^1k=\fm$ and let $\fa=xR$ for some $0\neq x \in\fm$. Then $\hh^2_{\fa}(M\otimes_RM^*)=0$ because $\fa$ is principal. However, $\pd_R(M)=1=n-1$. This example shows, setting $n=2$, the bound on $n$ is sharp in Proposition \ref{pr3}.
\end{enumerate}
\end{eg}

In the following we establish a non-vanishing result, which is new even for regular local rings.

\begin{cor}\label{ccc} Assume $R$ is Gorenstein and local ring, and let $M\in\md_0(R)$ be  Tor-rigid (e.g. $R$ is regular.) Assume further $R$ has odd dimension $d\geq3$ and $\depth_R(M)=\frac{d+1}{2}$. Then it follows that $\hh^i_{\fm}(M\otimes_RM^*)\neq0$ for all $i$ with $0<i<d$.
\end{cor}

\begin{proof}
Set $t:=\frac{d+1}{2}$. Assume contrarily that $\hh^i_{\fm}(M\otimes_RM^*)=0$ for some $i<d$. First we deal with case $0<i\leq t$. By Corollary \ref{c7} we may assume that $i\neq1$.
Hence, $2\leq i\leq t$. In view of Proposition \ref{pr3} we have $\pd_R(M)<i-1$. It follows from the Auslander--Buchsbaum formula that
$d-t<i-1\leq t-1$. Hence $d<2t-1=2\frac{d+1}{2}-1=d$ which is a contradiction. Now let $i>t$. By Proposition \ref{tomax}, $\hh^{d-i+1}_{\fm}(M\otimes_RM^*)=0$. Note that $d-i+1<d-t+1=t+1$ and so we get a contradiction by the first part of proof.
\end{proof}

The following example shows that the locally free assumption in Corollary \ref{ccc} is crucial.

\begin{eg}\label{notvb2} $\phantom{}$ Assume $R$ is $3$-dimensional, Gorenstein, and local (e.g., $R=k[\![x,y,z]\!]$). Let $M= R/xR$ for a non zero-divisor $x\in \fm$. Then $\depth_R(M)=2=\frac{3+1}{2}$ and $\pd_R(M)=1$. Hence, $M$ is Tor-rigid. Note that, as $M$ is torsion, we have $M^{\ast}=0$. Thus $\hh^i_{\fm}(M\otimes_RM^*)=0$ for all $i$.
\end{eg}

An application of Proposition \ref{ll} is a slight generalization of Theorem \ref{t3}.\\


\begin{thm}\label{tf} Assume $R$ is local, a complete intersection, and has dimension $d\geq 2$. Let $M\in\md(R)$ and assume $M$ is locally free on $\X^{1}(R)$ and it satisfies $(S_{3})$. Then the following hold:
\begin{enumerate}[\rm(i)]
\item If $\hh^{d-2}_{\fm}(M\otimes_RM^*)=0$, then $\pd_R(M)\leq 1$.
\item If $\hh^{d-2}_{\fm}(M\otimes_RM^*)=0=\hh^{d-1}_{\fm}(M\otimes_RM^*)$, then $M$ is free.
\end{enumerate}	
\end{thm}

\begin{proof}
This follows from the Theorem \ref{Ld}, Proposition \ref{ll} and \cite[2.5]{Jo}.	
\end{proof}

\section*{Acknowledgements}
The authors are grateful to K\c{e}stutis \v{C}esnavi\v{c}ius, Hiroki Matsui and Takeshi Yoshizawa for their suggestions and comments on an earlier version of the manuscript.

\bibliographystyle{amsplain}

\end{document}